\documentclass[reqno,11pt]{amsart}
\usepackage[left=1in,right=1in,top=1in,bottom=1in]{geometry}
\usepackage{enumitem}
\usepackage{amsmath}
\usepackage{graphicx}
\usepackage{tikz}
\usepackage{esint}
\usetikzlibrary{arrows, decorations.markings,matrix,trees}
\usepackage{hyperref}
\hypersetup{
   colorlinks=true,
   citecolor=blue,
   filecolor=blue,
   linkcolor=blue,
   urlcolor=blue
}

\newcommand{\al}{\alpha}       
    
\newcommand{\lda}{\lambda}
\newcommand{\om}{\Omega}            
\newcommand{\pa}{\partial}
\newcommand{\va}{\varepsilon}       
\newcommand{\ud}{\mathrm{d}}
\newcommand{\be}{\begin{equation}} 
\newcommand{\ee}{\end{equation}}

\newcommand{\cB}{\mathcal{B}}

\newcommand{\cC}{\mathcal{C}}
\newcommand{\C}{\mathbb{C}}

\newcommand{\cL}{\mathcal{L}} 

\newcommand{\Z}{\mathbb{Z}}

\newcommand{\R}{\mathbb{R}}

\newcommand{\ga}{\gamma}

\newcommand{\sg}{\sigma} 
\newcommand{\ift}{\infty} 
\newcommand{\wt}{\widetilde}

\newcommand{\f}{\frac}

\newcommand{\ol}{\overline}

\newcommand{\op}{\operatorname}

\DeclareMathOperator{\dist}{dist}

\def\<{\langle}\def\>{\rangle}
\def\({\left(}\def\){\right)}
\numberwithin{equation}{section}
\theoremstyle{plain}
\newtheorem{thm}{Theorem}[section]
\newtheorem{cor}[thm]{Corollary}

\newtheorem{lem}[thm]{Lemma}
\newtheorem{prop}[thm]{Proposition}

\theoremstyle{definition}
\newtheorem{defn}[thm]{Definition}

\theoremstyle{remark}
\newtheorem{rem}[thm]{Remark}

\title[Parabolic Green's Functions in Homogenization]{Sharp Convergence Rates for Parabolic Green's Functions in Time-Independent Periodic Homogenization}

\date{}

\author{Wei Wang}
\address{School of Mathematical Sciences, Peking University, Beijing 100871, China}
\email{wwmath166@outlook.com,\,\,2201110024@stu.pku.edu.cn}

\begin{document}

\begin{abstract}
We study Dirichlet Green's functions associated with second-order parabolic systems with rapidly oscillating periodic coefficients that are symmetric and independent of time. For bounded $C^{1,1}$ domains, we obtain a sharp zeroth-order convergence estimate from the oscillating Green's function to its homogenized counterpart, with the optimal rate $O(\va)$ and Gaussian off-diagonal decay. For bounded $C^{2,1}$ domains, we also prove a first-order expansion for the spatial gradient in terms of Dirichlet correctors, with an $O(\va)$ error up to a logarithmic factor. In this time-independent symmetric setting, these results improve the convergence rates established in \cite{Gen23} for parabolic systems with time-dependent periodic coefficient matrices.
\end{abstract}


\maketitle

\section{Introduction}

\subsection{Background and main results}

Let $n\geq2$ and $m\geq1$ be integers. In this paper, we study the asymptotic behavior of the matrix-valued Green's functions associated with the time-independent parabolic homogenization system
\[
\pa_t+\cL_\va,
\quad
\cL_\va=-\op{div}\left[A\left(\f{x}{\va}\right)\nabla\right],
\quad
\va\in(0,1].
\]
Here
\[
A(y)=(a_{ij}^{\al\beta}(y))_{\substack{i,j\in\Z\cap[1,n]\\\al,\beta\in\Z\cap[1,m]}},
\quad y\in\R^n,
\]
is a real coefficient matrix satisfying the following assumptions. Throughout the paper, repeated indices are summed over their natural ranges.

\begin{itemize}
\item Symmetry: for any $i,j\in\Z\cap[1,n]$ and $\al,\beta\in\Z\cap[1,m]$,
\begin{equation}\label{intro-symmetry}
a_{ij}^{\al\beta}(y)=a_{ji}^{\beta\al}(y).
\end{equation}

\item Uniform ellipticity:
\begin{equation}\label{intro-ellipticity}
\mu|\xi|^2
\leq
a_{ij}^{\al\beta}(y)\xi_i^\al\xi_j^\beta
\leq
\mu^{-1}|\xi|^2
\end{equation}
for any $\xi=(\xi_i^\al)\in\R^{m\times n}$ and $y\in\R^n$.

\item Periodicity:
\begin{equation}\label{intro-periodicity}
A(y+z)=A(y)
\quad
\text{for any }z\in\Z^n.
\end{equation}

\item H\"older continuity:
\begin{equation}\label{intro-holder}
|A(x)-A(y)|
\leq
\tau |x-y|^\nu,
\quad
x,y\in\R^n,
\end{equation}
for some $\tau>0$ and $\nu\in(0,1)$.
\end{itemize}

Let $\cL_0=-\op{div}(\widehat A\nabla)$ be the homogenized operator associated with $\cL_\va$. The homogenized matrix $\widehat A$ is defined through the usual elliptic cell problems. More precisely, for $j\in\Z\cap[1,n]$ and $\beta\in\Z\cap[1,m]$, let
\[
\chi_j^\beta=(\chi_j^{\al\beta})_{\al\in\Z\cap[1,m]}
\]
be the periodic corrector satisfying
\begin{equation}\label{intro-cell-problem}
\left\{
\begin{aligned}
&-\op{div}(A(y)\nabla \chi_j^\beta)=
\op{div}(A(y)\nabla P_j^\beta)\quad
\text{in }\mathbb T^n,
\\
&\chi_j^\beta\in H^1(\mathbb T^n,\R^m),
\quad
\int_{\mathbb T^n}\chi_j^\beta(y)\ud y=0,
\end{aligned}
\right.
\end{equation}
where $P_j^\beta(y)=y_j e^\beta$ and $e^\beta$ denotes the $\beta$-th coordinate vector in $\R^m$. Then
\begin{equation}\label{intro-homogenized-coefficients}
\widehat a_{ij}^{\al\beta}
=
\int_{\mathbb{T}^n}
\left[
a_{ij}^{\al\beta}(y)
+
a_{ik}^{\al\gamma}(y)
\f{\pa}{\pa y_k}\chi_j^{\gamma\beta}(y)
\right]\ud y.
\end{equation}

Since $A$ is independent of time, it follows from \cite[Section 3]{GS15} that $\pa_t+\cL_0$ is the homogenized operator of $\pa_t+\cL_\va$. Under the symmetry assumption \eqref{intro-symmetry}, the homogenized matrix $\widehat A$ is also symmetric:
\be
\widehat{a}_{ij}^{\al\beta}
=
\widehat{a}_{ji}^{\beta\al}
\quad
\text{for any }i,j\in\Z\cap[1,n]
\text{ and }\al,\beta\in\Z\cap[1,m].
\label{syl0}
\ee
Moreover, if $A$ satisfies \eqref{intro-ellipticity}, then there exists $\mu_1>0$, depending only on $\mu$, such that
\be
\mu_1|\xi|^2
\leq
\widehat{a}_{ij}^{\al\beta}\xi_i^\al\xi_j^\beta
\leq
\mu_1^{-1}|\xi|^2
\quad
\text{for any }\xi=(\xi_i^\al)\in\R^{m\times n}.
\label{ell0}
\ee
After replacing $\mu$ with $\min(\mu,\mu_1)$, we use the same ellipticity constant $\mu$ for both $A$ and $\widehat A$.

The homogenization theory for elliptic systems is well developed. Under the ellipticity and periodicity assumptions alone, one has, up to a subsequence, weak $H^1$ convergence of solutions $u_\va$ of $\cL_\va u_\va=F\in H^{-1}$ to a solution of the homogenized equation $\cL_0u_0=F$; see \cite[Chapter 1]{ZKO94}. For classical treatments of periodic homogenization and the method of two-scale asymptotic expansions, we refer to \cite{BLP78,OSY92,ZKO94}. Quantitative convergence rates in various settings were subsequently obtained in \cite{Gri04,Gri06,KLS12,KLS13,OV07,Sus13a,Sus13b}. Birman--Suslina developed a complementary operator-theoretic approach, including estimates with correctors for periodic differential operators; see, for example, \cite{BS04,BS06}. Another major direction is the study of estimates that are uniform with respect to the parameter $\va$. For the Dirichlet problem, Avellaneda--Lin \cite{AL87} developed the compactness method and proved uniform Lipschitz estimates for solutions of $\cL_\va u_\va=0$ under \eqref{intro-ellipticity}--\eqref{intro-holder}. Shen \cite{She08} established uniform $W^{1,p}$ estimates for coefficients in the $\op{VMO}$ class. For Neumann boundary conditions, Kenig--Lin--Shen \cite{kLS13b} obtained uniform regularity estimates under the additional symmetry assumption, and the symmetry condition was later removed by Armstrong--Shen \cite{AS16}. Shen's monograph \cite{She18} gives a systematic treatment of quantitative periodic homogenization for second-order elliptic systems, including convergence rates, uniform interior and boundary estimates, and kernel expansions. In recent years, related quantitative and uniform regularity theories have developed in several directions, including higher-order elliptic systems \cite{NSX18,NX19}, elliptic systems with lower-order terms \cite{Xu16Lower,Xu16LowerNeu,WZ23}, and reiterated or multiscale periodic media \cite{NSX20,NXZ26}.

It is natural to consider the corresponding parabolic operators $\pa_t+\cL_\va$. For time-independent coefficient matrices, preliminary results can be found in \cite[Chapter 2]{ZKO94}. A more general problem allows the coefficients to depend on both space and time. In this setting, Geng--Shen established interior Lipschitz estimates and boundary $W^{1,p}$ estimates in \cite{GS15}, together with quantitative convergence results in \cite{GS17} and asymptotic expansions of fundamental solutions in \cite{GS20b}. Boundary Lipschitz estimates parallel to those of Avellaneda--Lin were obtained by Geng--Shi in \cite{GS20}, where Green's functions for parabolic operators were also constructed. The estimates of Green's functions used in this line of parabolic homogenization are naturally connected with the parabolic system theory of Hofmann--Kim \cite{HK04}, Cho--Dong--Kim \cite{CDK08,CDK12}, and Dong--Kim \cite{DK14}. More recent results of parabolic homogenization include multiscale non-self-similar and locally periodic problems \cite{GengShen20NS,GengNiu25}, re-invented spatial and temporal scales \cite{Niu24Para}, locally periodic coefficients \cite{Xu23Para}, and fundamental solution expansions on non-self-similar scales \cite{MengNiu24}. For Neumann boundary conditions, convergence rates for initial-Neumann problems were already included in \cite{GS17}, and recent estimates of the regularity of the boundary and the Neumann function were obtained in \cite{Shi25Neu}.

We now fix the notation for the Green's functions used in this paper. Let $\om\subset\R^n$ be a bounded $C^{1,1}$ domain. For $\va\in[0,1]$, where $\va=0$ refers to the homogenized operator $\cL_0$, let
\[
G_\va(x,t;y,s)
=(G_\va^{\al\beta}(x,t;y,s))_{\al,\beta\in\Z\cap[1,m]}
\]
denote the Dirichlet Green's function, or Dirichlet heat kernel, with pole at $(y,s)$ for $\pa_t+\cL_\va$ in $\om\times\R$. It is characterized by
\[
\left\{
\begin{aligned}
(\pa_t+\cL_\va)G_\va(\cdot,\cdot;y,s)
&=
\delta_{(y,s)}(\cdot,\cdot)\op{Id}_m
&&\text{in }\om\times\R,
\\
G_\va(\cdot,\cdot;y,s)
&=0
&&\text{on }\pa\om\times\R,
\\
G_\va(x,t;y,s)
&=0
&&\text{if }t\leq s,
\\
\lim_{t\to s^+}G_\va(\cdot,t;y,s)
&=
\delta_y(\cdot)\op{Id}_m
&&\text{in the sense of distributions in }\om.
\end{aligned}
\right.
\]
See Proposition \ref{defn-parabolic-Green} for the existence and basic properties of $G_\va$.

Our first main result gives the sharp zeroth-order convergence rate for the Dirichlet Green's functions, with the same Gaussian off-diagonal decay as the parabolic kernel itself.

\begin{thm}\label{intro-thm-parabolic-zero}
Let $\va\in(0,1]$, $n\geq2$, and let $\om\subset\R^n$ be a bounded $C^{1,1}$ domain. Assume that $A$ satisfies \eqref{intro-symmetry}--\eqref{intro-holder}. Then, for any $x,y\in\om$ and any $t>s$,
\begin{equation}\label{intro-parabolic-zero}
\left|
G_\va(x,t;y,s)-G_0(x,t;y,s)
\right|
\leq
\f{C\va}{(t-s)^{\f{n+1}{2}}}
\exp\left\{
-\f{\kappa|x-y|^2}{t-s}
\right\},
\end{equation}
where $C,\kappa>0$ depend only on $n$, $m$, $\mu$, $\tau$, $\nu$, and $\om$.
\end{thm}

\begin{rem}\label{rem-intro-parabolic-zero}
Several comments clarify the position of Theorem \ref{intro-thm-parabolic-zero} among known convergence results.

\begin{enumerate}
\item At the level of fundamental solutions, the estimate \eqref{intro-parabolic-zero} is consistent with the convergence result in \cite{GS20b}. In both cases, the convergence rate is $O(\va)$, which is sharp in general.

\item The time independence of $A$ is used to recover this sharp rate. For space-time periodic coefficients, related convergence estimates were obtained in \cite{Gen23}, but the general rate there is $O(\va^{\f{1}{2}-})$. Thus, Theorem \ref{intro-thm-parabolic-zero} shows that the sharper elliptic-in-space rate obtained in \cite{KLS14} persists for the parabolic Green's function when the coefficients are independent of time.

\item In the scalar case, \cite[Corollary 2.7]{ZKO94} gives an $O(\va)$ convergence result for parabolic Green's functions. However, that estimate does not include the Gaussian exponential decay in \eqref{intro-parabolic-zero}. Therefore, Theorem \ref{intro-thm-parabolic-zero} may be viewed as a Gaussian off-diagonal version of the sharp $O(\va)$ convergence estimate for systems. It is proved in a more restrictive setting than the fully time-dependent theory of \cite{Gen23}, but in a more general setting than the scalar result of \cite[Corollary 2.7]{ZKO94}.
\end{enumerate}
\end{rem}

The first-order expansion involves the matrix of Dirichlet correctors. These correctors were introduced in this elliptic homogenization setting by Avellaneda--Lin \cite{AL87} and are particularly well suited to boundary estimates for Dirichlet problems. Since the coefficients in the present paper are independent of time, the parabolic Dirichlet corrector appearing in the time-dependent theory reduces to the elliptic Dirichlet corrector. For $j\in\Z\cap[1,n]$ and $\beta\in\Z\cap[1,m]$, let
\[
\Phi_{\va,j}^\beta=(\Phi_{\va,j}^{\al\beta})_{\al\in\Z\cap[1,m]}
\]
solve the Dirichlet problem
\begin{equation}\label{intro-dirichlet-corrector}
\left\{
\begin{aligned}
\cL_\va(\Phi_{\va,j}^\beta)&=0
&&\text{in }\om,
\\
\Phi_{\va,j}^\beta&=P_j^\beta
&&\text{on }\pa\om.
\end{aligned}
\right.
\end{equation}
Using this Dirichlet corrector, we obtain the following first-order expansion for the spatial gradient of the parabolic Green's function.

\begin{thm}\label{intro-thm-parabolic-first}
Let $\va\in(0,1]$, $n\geq2$, and let $\om\subset\R^n$ be a bounded $C^{2,1}$ domain. Assume that $A$ satisfies \eqref{intro-symmetry}, \eqref{intro-ellipticity}, \eqref{intro-periodicity}, and \eqref{intro-holder}. Then, for any $x,y\in\om$, any $t>s$, and any $\al,\ga\in\Z\cap[1,m]$,
\begin{equation}\label{intro-parabolic-first}
\begin{aligned}
&\left|
\f{\pa}{\pa x_i}G_\va^{\al\gamma}(x,t;y,s)
-
\f{\pa\Phi_{\va,j}^{\al\beta}}{\pa x_i}(x)
\f{\pa}{\pa x_j}G_0^{\beta\gamma}(x,t;y,s)
\right|
\\
&\quad\quad\quad\leq
\f{
C\va
\log(2+\va^{-1}(t-s)^{\f{1}{2}})
}{(t-s)^{\f{n+2}{2}}}
\exp\left\{-\f{\kappa|x-y|^2}{t-s}\right\}.
\end{aligned}
\end{equation}
Here, the repeated indices $j$ and $\beta$ are summed over their natural ranges, and $C,\kappa>0$ depend only on $n$, $m$, $\mu$, $\tau$, $\nu$, and $\om$.
\end{thm}

\begin{rem}\label{rem-intro-parabolic-first}
Theorem \ref{intro-thm-parabolic-first} is the first-order counterpart of Theorem \ref{intro-thm-parabolic-zero}. Compared with the corresponding first-order estimate in the fully space-time periodic setting of \cite{Gen23}, the rate is improved from $O(\va^{\f{1}{2}-})$ to $O(\va)$, up to the logarithmic factor naturally associated with the first-order Dirichlet corrector estimate. This improvement relies on the time-independent structure of the coefficients.
\end{rem}

\subsection{Strategies in the proof}

We briefly describe the strategy of the proof. The main difficulty is to recover the sharp elliptic-in-space convergence rate after passing from the resolvent kernels to the parabolic Green's function. In the fully time-dependent setting of \cite{Gen23}, the boundary cut-off procedure near $\pa\om$ leads to a loss in the convergence order. In the present time-independent setting, we avoid this loss by starting from elliptic resolvent estimates and then passing to the parabolic problem through an inverse Laplace transform.

The first step is to refine the dependence on the resolvent parameter in the local elliptic estimates. We use the sector
\[
\Sigma_{\theta_0}
=
\left\{
\lambda\in\C\backslash\{0\}:\ |\arg\lambda|>\theta_0
\right\},
\quad
\theta_0\in\(0,\f{\pi}{2}\),
\]
which is separated from the positive real axis. For any $\lambda\in\Sigma_{\theta_0}$, we prove a Caccioppoli inequality of the form
\[
\left(\fint_{U_r(x_0)}
\left(|u_{\va,\lambda}|^2+r^2|\nabla u_{\va,\lambda}|^2\right)
\right)^{\f{1}{2}}
\leq
\f{C(Ck)^{2k}}{(1+|\lambda|r^2)^k}
\left(\fint_{U_{2r}(x_0)}|u_{\va,\lambda}|^2\right)^{\f{1}{2}}
+\cdots,
\]
where the dependence on the integer $k$ is kept explicit. This refines the corresponding estimate in \cite{Wan23} by tracking the growth in $k$. It then allows us to choose $k$ of the order of $|\lambda|^{\f{1}{2}}|x-y|$, which produces exponential off-diagonal decay for the resolvent Green's functions.

Moreover, we establish the zeroth-order and first-order resolvent convergence estimates
\[
\left|G_{\va,\lambda}(x,y)-G_{0,\lambda}(x,y)\right|
\leq
\f{C(Ck)^{2k}\va}
{(1+|\lambda||x-y|^2)^k |x-y|^{n-1}},
\]
and
\[
\left|
\f{\partial}{\partial x_i}
G_{\va,\lambda}^{\al\gamma}(x,y)
-
\f{\partial}{\partial x_i}
\Phi_{\va,j}^{\al\beta}(x)
\f{\partial}{\partial x_j}
G_{0,\lambda}^{\beta\gamma}(x,y)
\right|
\leq
\f{
C(Ck)^{2k}\va
\log(\va^{-1}|x-y|+2)
}{
(1+|\lambda||x-y|^2)^k |x-y|^n
}.
\]
Choosing $k$ in terms of $|\lambda|^{\f{1}{2}}|x-y|$ then yields exponential decay in the quantity $|\lambda|^{\f{1}{2}}|x-y|$. These resolvent estimates form the elliptic core of the paper.

The second step is to pass from the resolvent Green's functions to the parabolic Green's functions. Since when the coefficient matrix $ A $ satisfies \eqref{intro-symmetry} and \eqref{intro-ellipticity}, the Dirichlet realization of $\cL_\va$ is nonnegative and self-adjoint in $L^2(\om;\C^m)$, the semigroup generated by $-\cL_\va$ admits the inverse Laplace representation
\[
G_\va(x,t;y,s)
=
\f{1}{2\pi\op{i}}
\int_{\mathcal C_\gamma}
\exp\{(t-s)\zeta\}
G_{\va,-\zeta}(x,y)\ud\zeta,
\quad t>s,
\]
where $\mathcal C_\gamma=\{\gamma+\op{i}\eta:\eta\in\R\}$ and $\gamma>0$. Along this contour, $-\zeta$ lies in the resolvent sector after fixing $\theta_0<\f{\pi}{2}$. We then differentiate the resolvent kernels with respect to $\zeta$, integrate by parts along $\mathcal C_\gamma$, and optimize the parameter $\gamma$ according to the two regimes $|x-y|^2\leq t-s$ and $|x-y|^2>t-s$. This gives the Gaussian factor in Theorem \ref{intro-thm-parabolic-zero}.

The proof of Theorem \ref{intro-thm-parabolic-first} requires an additional argument. The main point is that the differentiated first-order resolvent error cannot be controlled by a direct pointwise estimate alone. For $\zeta$ such that $-\zeta\in\Sigma_{\theta_0}$, and for fixed $1\leq i\leq n$ and $\al,\ga\in\Z\cap[1,m]$, set
\[
E_{\va,\zeta}^{\al\gamma}(x,y)
=
\f{\pa}{\pa x_i}G_{\va,-\zeta}^{\al\gamma}(x,y)
-
\f{\pa\Phi_{\va,j}^{\al\beta}}{\pa x_i}(x)
\f{\pa}{\pa x_j}G_{0,-\zeta}^{\beta\gamma}(x,y).
\]
We first smooth the error by composing it with powers of the homogenized resolvent kernel and estimate
\[
E_{\va,\zeta}\star G_{0,-\zeta}^{\star M}.
\]
We then recover $\pa_\zeta^M E_{\va,\zeta}$ through a resolvent identity, where $\star$ denotes spatial kernel composition; see \eqref{stardefinition}. This step uses the local first-order approximation estimate for resolvent solutions, the Dirichlet corrector estimates, and the Bessel-type convolution bounds developed later in the paper. The logarithmic factor in Theorem \ref{intro-thm-parabolic-first} comes from the first-order corrector estimate.

\subsection{Organization of this paper}

The paper is organized as follows. In Section \ref{PreSection}, we collect preliminary estimates that are used throughout the paper. In Section \ref{Green-functions-section}, we recall the existence and pointwise bounds for the resolvent Green's functions of $\cL_\va-\lambda I$. In Section \ref{ConGreenResolvent}, we prove Theorem \ref{intro-thm-parabolic-zero} and \ref{intro-thm-parabolic-first}.

\subsection{Notations and conventions}

We first fix the notation used throughout the paper.

\begin{itemize}
\item Throughout the paper, $C$ denotes a positive constant. When necessary, we write $C(a,b,\ldots)$ to indicate its dependence on the parameters $a,b,\ldots$. The value of $C$ may change from line to line.

\item We adopt the Einstein summation convention: repeated indices are implicitly summed over their full range.

\item For a fixed number $\theta_0\in(0,\f{\pi}{2})$, we set
\[
\Sigma_{\theta_0}
:=
\left\{
\lambda\in\C\backslash\{0\}:\ |\arg\lambda|>\theta_0
\right\}.
\]
Thus, $\Sigma_{\theta_0}$ is separated from the positive real axis. This is the sector in which the resolvent estimates are used.

\item For $d\in\Z$ with $d\geq2$, we define
\[
B_r^d(x):=\{y\in\R^d:\ |y-x|<r\}.
\]
We denote the origin of $\R^d$ by $0^d$. When $d=n$, we suppress the superscript and write $B_r(x)$ instead of $B_r^n(x)$, and $B_r$ instead of $B_r^n(0)$.

\item If $U\subset\R^n$ is a domain, $x\in\ol U$ and $r>0$, we write
\[
U_r(x):=U\cap B_r(x),
\quad
T_r^U(x):=\partial U\cap B_r(x),
\quad\text{and}\quad
\ol U_r(x):=\ol U\cap\ol B_r(x).
\]
Note that if $r\in(0,\dist(x,\pa U))$, then $T_r^U(x)=\emptyset$. When $x=0$, we simply write $U_r$ and $T_r^U$.
\end{itemize}

\begin{defn}[Local description of $C^k$ and $C^{k,\alpha}$ domains]\label{DefnLocal}
Let $U\subset\R^n$ be a bounded domain and let $k\in\Z_+$. We say that $U$ is a bounded $C^k$ domain with parameters $\cC_k(U)=(r_0,M_0,\omega_k)$, where $r_0>0$, $M_0>0$, and $\omega_k:[0,+\infty)\to[0,+\infty)$ is continuous, non-decreasing, and satisfies
\[
\lim_{t\to0^+}\omega_k(t)=0,
\]
if the following condition holds.

For any $x_0\in\partial U$, after a translation and a rotation of coordinates, we may assume that $x_0=0$ and write points $y\in\R^n$ as $y=(y',y_n)\in\R^{n-1}\times\R$. In these coordinates, there exists a $C^k$ function $\psi:\R^{n-1}\to\R$ such that
\[
U\cap B_r
=
\{(y',y_n)\in\R^{n-1}\times\R:\ y_n>\psi(y')\}\cap B_r
\]
for any $0<r<60(M_0+1)r_0$. Moreover,
\[
\psi(0^{n-1})=0,
\quad
D\psi(0^{n-1})=0,
\quad\text{and}\quad
\sum_{j=1}^{k}\|D^j\psi\|_{L^\infty(\R^{n-1})}\leq M_0.
\]
Finally, for any $x',y'\in\R^{n-1}$,
\[
|D^k\psi(x')-D^k\psi(y')|
\leq
\omega_k(|x'-y'|).
\]

If, in addition, for some $\alpha\in(0,1]$ the modulus $\omega_k$ may be chosen so that
\[
\omega_k(t)\leq M_\alpha t^\alpha
\quad\text{for any }t\geq0,
\]
then we say that $U$ is a bounded $C^{k,\alpha}$ domain with parameters $\cC_{k,\alpha}(U)=(r_0,M_0,M_\alpha)$. Equivalently, in the above local coordinates,
\[
[D^k\psi]_{C^{0,\alpha}(\R^{n-1})}
\leq M_\alpha.
\]

In these local coordinates, we use the notation
\[
U_r(0)=U\cap B_r
=
\{(y',y_n): y_n>\psi(y')\}\cap B_r,
\]
\[
\ol U_r(0)=\ol U\cap\ol B_r
=
\{(y',y_n): y_n\geq\psi(y')\}\cap\ol B_r,
\]
and
\[
T_r^U(0)=\partial U\cap B_r
=
\{(y',y_n): y_n=\psi(y')\}\cap B_r.
\]
For a general boundary point $x_0\in\partial U$, the same notation is used after translating and rotating the coordinate system at $x_0$.
\end{defn}

\section{Preliminaries}\label{PreSection}

\subsection{Caccioppoli's inequality}

We begin with a localized Caccioppoli inequality for the resolvent equation. The point of the lemma is to keep the dependence on the iteration parameter $k$ explicit, since this dependence will later be used to obtain exponential off-diagonal decay for the resolvent kernels.

\begin{lem}[Caccioppoli's inequality]\label{Caccio-data}
Let $\va\in[0,1]$, $n\geq2$, and let $\lambda\in\Sigma_{\theta_0}\cup\{0\}$ with $\theta_0\in(0,\f{\pi}{2})$. Let $U\subset\R^n$ be a bounded $C^1$ domain with parameters $\cC_1(U)$. Assume that $A$ satisfies \eqref{intro-symmetry}--\eqref{intro-ellipticity}. Let $x_0\in\ol U$ and $r\in(0,\f{r_0}{2})$. Set $q=\f{2n}{n+2}$ if $n\geq3$; if $n=2$, let $q>1$ be fixed. Assume that
\[
F\in L^q(U_{2r}(x_0);\C^m),
\quad
f\in L^2(U_{2r}(x_0);\C^{m\times n}).
\]

When $T_{2r}^{U}(x_0)\neq\emptyset$, assume that $u_{\va,\lambda}\in H^1(U_{2r}(x_0);\C^m)$ is a weak solution of
\[
\left\{
\begin{aligned}
(\cL_\va-\lambda I)u_{\va,\lambda}
&=F+\op{div}(f)
&&\text{in }U_{2r}(x_0),
\\
u_{\va,\lambda}&=0
&&\text{on }T_{2r}^{U}(x_0).
\end{aligned}
\right.
\]
If $T_{2r}^{U}(x_0)=\emptyset$, assume that $B_{2r}(x_0)\subset U$ and $u_{\va,\lambda}\in H^1(B_{2r}(x_0);\C^m)$ is a weak solution of
\[
(\cL_\va-\lambda I)u_{\va,\lambda}=F+\op{div}(f)
\quad\text{in }B_{2r}(x_0).
\]
Then, for any $k\in\Z_+$,
\begin{align*}
&\left(\fint_{U_r(x_0)}
\left(|u_{\va,\lambda}|^2+r^2|\nabla u_{\va,\lambda}|^2\right)
\right)^{\f{1}{2}}
\\
&\quad\leq
\f{C(Ck)^{2k}}
{(1+|\lambda|r^2)^k}
\left(\fint_{U_{2r}(x_0)}|u_{\va,\lambda}|^2\right)^{\f{1}{2}}+
C(Ck)^{2k}
r\left\{
r\left(\fint_{U_{2r}(x_0)}|F|^q\right)^{\f{1}{q}}
+\left(\fint_{U_{2r}(x_0)}|f|^2\right)^{\f{1}{2}}
\right\},
\end{align*}
where, in the interior case, $U_\rho(x_0)$ is replaced by $B_\rho(x_0)$. The constant $C$ depends only on $\mu$, $n$, $m$, $\theta_0$, $\cC_1(U)$, and also on $q$ when $n=2$.
\end{lem}

\begin{proof}
We prove only the boundary case; the interior case is identical, with $U_\rho(x_0)$ replaced by $B_\rho(x_0)$. We first assume that $|\lambda|>0$; the case $\lambda=0$ follows from the same argument after dropping the term involving $|\lambda|$. Up to translation, we let $x_0=0$.

For notational convenience, set
\be
A_\va(x)=
\begin{cases}
A(\f{x}{\va}),&\va\in(0,1],\\
\widehat A,&\va=0.
\end{cases}\label{Avax}
\ee
This convention allows the proof to cover both the oscillatory and homogenized operators. Let
\[
N=2(k+1),
\quad
\delta=\f{r}{N},
\quad\text{and}\quad
r_j=r+j\delta,
\quad j\in\Z\cap[0,N].
\]
Thus, $r_0=r$, $r_N=2r$, and $r_{j+1}-r_j=\delta$. For each $j=0,\ldots,N-1$, choose $\eta_j\in C_0^1(B_{r_{j+1}})$ such that
\[
0\leq \eta_j\leq1,
\quad
\eta_j\equiv1\text{ in }B_{r_j},
\quad\text{and}\quad
|\nabla\eta_j|\leq C(n)\delta^{-1}.
\]
Since $u_{\va,\lambda}=0$ on $T_{2r}^{U}$, the function $\eta_j^2u_{\va,\lambda}$ is an admissible test function in $U_{r_{j+1}}$.

We first prove a one-step estimate. Let $r\leq s<t\leq2r$, let $\ell=t-s$, and let $\eta\in C_0^\infty(B_t)$ satisfy
\[
0\leq\eta\leq1,
\quad
\eta\equiv1\text{ in }B_s,
\quad\text{and}\quad
|\nabla\eta|\leq C(n)\ell^{-1}.
\]
Testing the equation by $\eta^2u_{\va,\lambda}$ gives
\[
\begin{aligned}
&\int_{U_t}\eta^2\<A_\va\nabla u_{\va,\lambda},\overline{\nabla u_{\va,\lambda}}\>
-\lambda\int_{U_t}\eta^2|u_{\va,\lambda}|^2
\\
&\quad =
-2\int_{U_t}\eta\<A_\va\nabla u_{\va,\lambda},
\overline{u_{\va,\lambda}\nabla\eta}\>
+\int_{U_t}F\overline{\eta^2u_{\va,\lambda}}
-\int_{U_t}\<f,\overline{\nabla(\eta^2u_{\va,\lambda})}\>.
\end{aligned}
\]
Since $\lambda\in\Sigma_{\theta_0}$, the ellipticity of $A_\va$ and the standard real-imaginary part argument as in the proof of \cite[Lemma 2.4]{Wan23} imply
\be
\begin{aligned}
&\int_{U_t}\eta^2|\nabla u_{\va,\lambda}|^2
+|\lambda|\int_{U_t}\eta^2|u_{\va,\lambda}|^2
\\
&\quad\leq
C\int_{U_t}\eta|\nabla u_{\va,\lambda}|
|\nabla\eta||u_{\va,\lambda}|
+C\left|\int_{U_t}F\overline{\eta^2u_{\va,\lambda}}\right|
+C\left|\int_{U_t}\<f,\overline{\nabla(\eta^2u_{\va,\lambda})}\>\right|.
\end{aligned}
\label{eta2nablauvalda}
\ee
The first term is bounded by Cauchy's inequality:
\be
\int_{U_t}\eta|\nabla u_{\va,\lambda}|
|\nabla\eta||u_{\va,\lambda}|
\leq
\gamma\int_{U_t}\eta^2|\nabla u_{\va,\lambda}|^2
+
\f{C}{\gamma\ell^2}\int_{U_t}|u_{\va,\lambda}|^2 .
\label{etanablauvalda}
\ee
For the $F$-term in \eqref{eta2nablauvalda}, the scale-invariant Sobolev--Poincar\'e inequality for $\eta u_{\va,\lambda}$ and Cauchy's inequality imply that, for any $\gamma\in(0,1)$,
\be
\begin{aligned}
\left|\int_{U_t}F\overline{\eta^2u_{\va,\lambda}}\right|
&\leq
\|F\|_{L^q(U_t)}\|\eta u_{\va,\lambda}\|_{L^{q'}(U_t)}
\\
&\leq
\gamma\int_{U_t}\eta^2|\nabla u_{\va,\lambda}|^2
+\f{C}{\gamma\ell^2}\int_{U_t}|u_{\va,\lambda}|^2
+
\f{C}{\gamma}\cL^n(U_t)t^2
\left(\fint_{U_t}|F|^q\right)^{\f{2}{q}},
\end{aligned}
\label{Fterm}
\ee
where $q'=\f{2n}{n-2}$ if $n\geq3$, while $q'=\f{q}{q-1}$ if $n=2$. In the two-dimensional case, the constant also depends on the fixed choice of $q>1$. Similarly, for any $\gamma\in(0,1)$,
\be
\begin{aligned}
\left|\int_{U_t}f\cdot
\overline{\nabla(\eta^2u_{\va,\lambda})}\right|
&\leq
C\int_{U_t}\eta |f||\nabla u_{\va,\lambda}|
+
C\int_{U_t}|f||\nabla\eta||u_{\va,\lambda}|
\\
&\leq
\gamma\int_{U_t}\eta^2|\nabla u_{\va,\lambda}|^2
+\f{C}{\gamma\ell^2}\int_{U_t}|u_{\va,\lambda}|^2
+\f{C}{\gamma}\int_{U_t}|f|^2 .
\end{aligned}
\label{fterm}
\ee
Combining \eqref{eta2nablauvalda}--\eqref{fterm}, choosing $\gamma>0$ sufficiently small, and absorbing the gradient terms, we obtain
\be
\begin{aligned}
&\int_{U_s}|\nabla u_{\va,\lambda}|^2
+|\lambda|\int_{U_s}|u_{\va,\lambda}|^2
\\
&\quad\quad\leq
\f{C}{\ell^2}
\int_{U_t}|u_{\va,\lambda}|^2
+
C\cL^n(U_t)t^2
\left(\fint_{U_t}|F|^q\right)^{\f{2}{q}}
+
C\cL^n(U_t)
\fint_{U_t}|f|^2 .
\end{aligned}
\label{onestepst}
\ee

Applying \eqref{onestepst} with $s=r_j$, $t=r_{j+1}$, and $\ell=\delta$, and using $r\leq r_j<r_{j+1}\leq2r$, we get
\be
\begin{aligned}
&\left(\fint_{U_{r_j}}|\nabla u_{\va,\lambda}|^2\right)^{\f{1}{2}}
+
|\lambda|^{\f{1}{2}}
\left(\fint_{U_{r_j}}|u_{\va,\lambda}|^2\right)^{\f{1}{2}}
\\
&\quad\quad\leq
C\left\{
\f{1}{\delta}
\left(\fint_{U_{r_{j+1}}}|u_{\va,\lambda}|^2\right)^{\f{1}{2}}
+
r\left(\fint_{U_{2r}}|F|^q\right)^{\f{1}{q}}
+
\left(\fint_{U_{2r}}|f|^2\right)^{\f{1}{2}}
\right\}.
\end{aligned}
\label{lda12}
\ee
Set
\[
V_j=
\left(\fint_{U_{r_j}}|u_{\va,\lambda}|^2\right)^{\f{1}{2}},
\quad
G=
r\left(\fint_{U_{2r}}|F|^q\right)^{\f{1}{q}}
+
\left(\fint_{U_{2r}}|f|^2\right)^{\f{1}{2}}.
\]
Then \eqref{lda12} gives
\[
|\lambda|^{\f{1}{2}}V_j
\leq
C\left(\f{1}{\delta} V_{j+1}+G\right).
\]
Together with the trivial estimate $V_j\leq CV_{j+1}$, this yields
\[
V_j
\leq
\f{CV_{j+1}}{(1+|\lambda|\delta^2)^{\f{1}{2}}}
+
C\delta G .
\]
Iterating from $j=0$ to $j=N-1$, we obtain
\[
V_0
\leq
\left(\f{C}{(1+|\lambda|\delta^2)^{\f{1}{2}}}\right)^N
V_N
+
C\delta G\left\{
\sum_{j=0}^{N-1}
\left(\f{C}{(1+|\lambda|\delta^2)^{\f{1}{2}}}\right)^j\right\}.
\]
Since $N=2(k+1)$ and $\delta=\f{r}{N}$,
\[
\left(\f{C}{(1+|\lambda|\delta^2)^{\f{1}{2}}}\right)^N
=
\f{C^N}
{\left(1+\f{|\lambda|r^2}{N^2}\right)^{k+1}}.
\]
Using
\[
1+\f{|\lambda|r^2}{N^2}
\geq
\f{1+|\lambda|r^2}{N^2},
\]
we obtain
\[
\left(\f{C}{(1+|\lambda|\delta^2)^{\f{1}{2}}}\right)^N
\leq
\f{C(C(k+1))^{2k+2}}
{(1+|\lambda|r^2)^k}.
\]
Moreover,
\[
\delta
\sum_{j=0}^{N-1}
\left(\f{C}{(1+|\lambda|\delta^2)^{\f{1}{2}}}\right)^j
\leq
C(C(k+1))^{2k+2}r .
\]
Therefore,
\[
\begin{aligned}
V_0
&\leq
\f{C(C(k+1))^{2k+2}}
{(1+|\lambda|r^2)^k}
\left(\fint_{U_{2r}}|u_{\va,\lambda}|^2\right)^{\f{1}{2}}
\\
&\quad\quad+
C(C(k+1))^{2k+2}r
\left\{
r\left(\fint_{U_{2r}}|F|^q\right)^{\f{1}{q}}
+
\left(\fint_{U_{2r}}|f|^2\right)^{\f{1}{2}}
\right\}.
\end{aligned}
\]

It remains to estimate the gradient. Applying the one-step estimate \eqref{onestepst} with $s=r_0$ and $t=r_1$, we have
\[
\left(\fint_{U_r}|\nabla u_{\va,\lambda}|^2\right)^{\f{1}{2}}
\leq
C\left(\f{1}{\delta} V_1+G\right).
\]
Multiplying by $r$ and using $\f{r}{\delta}=N=2(k+1)$ gives
\[
r\left(\fint_{U_r}|\nabla u_{\va,\lambda}|^2\right)^{\f{1}{2}}
\leq
C(k+1)V_1+CrG.
\]
Iterating the estimate for $V_j$ from $j=1$ to $j=N-1$ yields
\[
V_1
\leq
\left(\f{C}{(1+|\lambda|\delta^2)^{\f{1}{2}}}\right)^{N-1}
V_N
+
C\delta G
\left\{\sum_{j=0}^{N-2}
\left(\f{C}{(1+|\lambda|\delta^2)^{\f{1}{2}}}\right)^j\right\}.
\]
As before,
\[
(k+1)
\left(\f{C}{(1+|\lambda|\delta^2)^{\f{1}{2}}}\right)^{N-1}
\leq
\f{C(C(k+1))^{2k+2}}
{(1+|\lambda|r^2)^k},
\]
and
\[
(k+1)\delta\left\{
\sum_{j=0}^{N-2}
\left(\f{C}{(1+|\lambda|\delta^2)^{\f{1}{2}}}\right)^j\right\}
\leq
C(C(k+1))^{2k+2}r .
\]
Consequently,
\[
\begin{aligned}
r\left(\fint_{U_r}|\nabla u_{\va,\lambda}|^2\right)^{\f{1}{2}}
&\leq
\f{C(C(k+1))^{2k+2}}
{(1+|\lambda|r^2)^k}
\left(\fint_{U_{2r}}|u_{\va,\lambda}|^2\right)^{\f{1}{2}}
\\
&\quad\quad+
C(C(k+1))^{2k+2}
r\left\{
r\left(\fint_{U_{2r}}|F|^q\right)^{\f{1}{q}}
+
\left(\fint_{U_{2r}}|f|^2\right)^{\f{1}{2}}
\right\}.
\end{aligned}
\]
Combining this with the estimate for $V_0$ gives
\[
\begin{aligned}
&\left(\fint_{U_r}
\left(|u_{\va,\lambda}|^2+r^2|\nabla u_{\va,\lambda}|^2\right)
\right)^{\f{1}{2}}
\\
&\quad\leq
\f{C(C(k+1))^{2k+2}}
{(1+|\lambda|r^2)^k}
\left(\fint_{U_{2r}}|u_{\va,\lambda}|^2\right)^{\f{1}{2}}+
C(C(k+1))^{2k+2}
r\left\{
r\left(\fint_{U_{2r}}|F|^q\right)^{\f{1}{q}}
+
\left(\fint_{U_{2r}}|f|^2\right)^{\f{1}{2}}
\right\}.
\end{aligned}
\]
Since $k\in\Z_+$, the factor $(C(k+1))^{2k+2}$ may be absorbed into $(Ck)^{2k}$ after increasing $C$. Indeed,
\[
(C_0(k+1))^{2k+2}
=
\left(1+\f{1}{k}\right)^{2k+2}C_0^{2k+2}k^{2k+2}
\leq
(Ck)^{2k},
\]
because $\left(1+\f{1}{k}\right)^{2k+2}$ is uniformly bounded and $k^2\leq C^k$ for $k\in\Z_+$. This proves the desired estimate.
\end{proof}
\subsection{Uniform regularity}

We next record the localized regularity estimates that will be used to convert the preceding energy estimate into pointwise bounds. The first result gives H\"older and $L^\infty$ bounds with explicit dependence on the parameter $k$.

\begin{prop}[Localized H\"older and $L^\infty$ estimates]\label{Holder-Linfty-data}
Let $\va\in[0,1]$, $n\geq2$, and let $\lambda\in\Sigma_{\theta_0}\cup\{0\}$ with $\theta_0\in(0,\f{\pi}{2})$. Let $U\subset\R^n$ be a bounded $C^1$ domain with parameters $\cC_1(U)$. Assume that $A$ satisfies \eqref{intro-symmetry}--\eqref{intro-holder}. Let $p\in(n,+\ift)$, $q=\f{np}{n+p}$, and $\sigma=1-\f np$. Let $x_0\in\ol U$, $r\in(0,\f{r_0}{4})$, and suppose that
\[
(\cL_\va-\lambda I)u_{\va,\lambda}
=
F+\op{div}(f)
\]
holds either in $U_{2r}(x_0)$ with $u_{\va,\lambda}=0$ on $T_{2r}^{U}(x_0)$, or in $B_{2r}(x_0)$ when $B_{2r}(x_0)\subset U$. Assume that
\[
F\in L^q(U_{2r}(x_0);\C^m),
\quad
f\in L^p(U_{2r}(x_0);\C^{m\times n}).
\]
Then there exists an integer $N_0=N_0(n)\in\Z_+$ such that, for any $ k\in\Z_+$ and any $s\in(0,+\ift)$,
\be
\begin{aligned}
\left[u_{\va,\lambda}\right]_{C^{0,\sigma}(U_r(x_0))}
&\leq
\f{C(Ck)^{2k}}
{(1+|\lambda|r^2)^k r^\sigma}
\left(\fint_{U_{2r}(x_0)}|u_{\va,\lambda}|^2\right)^{\f{1}{2}}
\\
&\quad+
C(Ck)^{2k}
(1+|\lambda|r^2)^{N_0} r^{1-\sigma}
\left\{
r\left(\fint_{U_{2r}(x_0)}|F|^q\right)^{\f{1}{q}}
+
\left(\fint_{U_{2r}(x_0)}|f|^p\right)^{\f{1}{p}}
\right\},
\end{aligned}
\label{HolderEstimate}
\ee
and
\be
\begin{aligned}
\|u_{\va,\lambda}\|_{L^\infty(U_r(x_0))}
&\leq
\f{C(Ck)^{2k}}
{(1+|\lambda|r^2)^k}
\left(\fint_{U_{2r}(x_0)}|u_{\va,\lambda}|^s\right)^{\f{1}{s}}
\\
&\quad+
C(Ck)^{2k}(1+|\lambda|r^2)^{N_0}
\left\{
r^2\left(\fint_{U_{2r}(x_0)}|F|^q\right)^{\f{1}{q}}
+
r\left(\fint_{U_{2r}(x_0)}|f|^p\right)^{\f{1}{p}}
\right\}.
\end{aligned}
\label{Linftyestimate}
\ee
Here, in the interior case, $U_\rho(x_0)$ is replaced by $B_\rho(x_0)$. The constant $C$ depends only on $\mu$, $n$, $m$, $\theta_0$, $p$, $s$, $\tau$, $\nu$, and $\cC_1(U)$.
\end{prop}

\begin{proof}
We prove the boundary case; the interior case is identical, with $U_\rho(x_0)$ replaced by $B_\rho(x_0)$. By \cite[Corollary 3.3]{Wan23}, applied on the pair of nested domains $U_r(x_0)\subset U_{\f{3r}{2}}(x_0)$, we have
\[
\begin{aligned}
\left[u_{\va,\lambda}\right]_{C^{0,\sigma}(U_r(x_0))}
&\leq
\f{C}
{r^\sigma}
\left(\fint_{U_{\f{3r}{2}}(x_0)}|u_{\va,\lambda}|^2\right)^{\f{1}{2}}
\\
&\quad+
C(1+|\lambda|r^2)^{N_0} r^{1-\sigma}
\left\{
r\left(\fint_{U_{\f{3r}{2}}(x_0)}|F|^q\right)^{\f{1}{q}}
+
\left(\fint_{U_{\f{3r}{2}}(x_0)}|f|^p\right)^{\f{1}{p}}
\right\},
\end{aligned}
\]
where $N_0=N_0(n)>0$. Applying Lemma \ref{Caccio-data} on the pair of radii $\f{3r}{2}<2r$, and using $p\in(n,+\ift)\geq2$ to control the local $L^2$ average of $f$ by its local $L^p$ average, gives \eqref{HolderEstimate}.

We next prove \eqref{Linftyestimate}. For any $x\in U_r(x_0)$, the elementary estimate by the average and the H\"older seminorm gives
\[
|u_{\va,\lambda}(x)|
\leq
\left(\fint_{U_{\f{3r}{2}}(x_0)}|u_{\va,\lambda}|^2\right)^{\f12}
+
Cr^\sigma
[u_{\va,\lambda}]_{C^{0,\sigma}(U_{\f{3r}{2}}(x_0))}.
\]
Using Lemma \ref{Caccio-data} and \eqref{HolderEstimate} on the pair of radii $\f{3r}{2}<\f{7r}{4}$, we obtain
\be
\begin{aligned}
|u_{\va,\lambda}(x)|
&\leq
\f{C(Ck)^{2k}}
{(1+|\lambda|r^2)^k}
\left(\fint_{U_{\f{7r}{4}}(x_0)}|u_{\va,\lambda}|^2\right)^{\f{1}{2}}
\\
&\quad+
C(Ck)^{2k}
(1+|\lambda|r^2)^{N_0} r
\left\{
r\left(\fint_{U_{\f{7r}{4}}(x_0)}|F|^q\right)^{\f{1}{q}}
+
\left(\fint_{U_{\f{7r}{4}}(x_0)}|f|^p\right)^{\f{1}{p}}
\right\}.
\end{aligned}
\label{uvaldax}
\ee
It remains to replace the intermediate $L^2$ average by the desired $L^s$ average. By \cite[Corollary 3.3]{Wan23}, applied on the pair $U_{\f{7r}{4}}(x_0)\subset U_{2r}(x_0)$, we have
\[
\begin{aligned}
\left(\fint_{U_{\f{7r}{4}}(x_0)}|u_{\va,\lambda}|^2\right)^{\f{1}{2}}
&\leq
C\|u_{\va,\lambda}\|_{L^{\ift}(U_{\f{7r}{4}}(x_0))}\leq
C\left(\fint_{U_{2r}(x_0)}|u_{\va,\lambda}|^s\right)^{\f{1}{s}}\\
&\quad\quad+
C(1+|\lda|r^2)^{N_0}\left\{
r^2\left(\fint_{U_{2r}(x_0)}|F|^q\right)^{\f{1}{q}}
+
r\left(\fint_{U_{2r}(x_0)}|f|^p\right)^{\f{1}{p}}
\right\}.
\end{aligned}
\]
Combining this estimate with \eqref{uvaldax} gives \eqref{Linftyestimate}.
\end{proof}

The next proposition is the localized Lipschitz estimate for the resolvent equation. It will be used later to obtain pointwise bounds for gradients of resolvent Green's functions.

\begin{prop}[Localized Lipschitz estimate for $\cL_\va-\lambda I$]\label{Lip-lambda}
Let $\va\in[0,1]$, $n\geq2$, and let $\lambda\in\Sigma_{\theta_0}\cup\{0\}$ with $\theta_0\in(0,\f{\pi}{2})$. Let $U\subset\R^n$ be a bounded $C^{1,\eta}$ domain with parameters $\cC_{1,\eta}(U)$, where $\eta\in(0,1)$. Assume that $A$ satisfies \eqref{intro-symmetry}--\eqref{intro-holder}. Let $p\in(n,+\ift)$, $x_0\in\ol U$, and $r\in(0,\f{r_0}{4})$.

When $T_{2r}^{U}(x_0)\neq\emptyset$, assume that $u_{\va,\lambda}\in H^1(U_{2r}(x_0);\C^m)$ is a weak solution of
\[
\left\{
\begin{aligned}
(\cL_\va-\lambda I)u_{\va,\lambda}
&=F
&&\text{in }U_{2r}(x_0),
\\
u_{\va,\lambda}&=0
&&\text{on }T_{2r}^{U}(x_0).
\end{aligned}
\right.
\]
If $T_{2r}^{U}(x_0)=\emptyset$, assume that $B_{2r}(x_0)\subset U$ and $u_{\va,\lambda}\in H^1(B_{2r}(x_0);\C^m)$ is a weak solution of
\[
(\cL_\va-\lambda I)u_{\va,\lambda}=F
\quad\text{in }B_{2r}(x_0).
\]
Assume that $F\in L^p(U_{2r}(x_0);\C^m)$ in the boundary case and $F\in L^p(B_{2r}(x_0);\C^m)$ in the interior case. Then there exists an integer $N_0=N_0(n)\in\Z_+$ such that, for any $ k\in\Z_+$,
\[
\begin{aligned}
\|\nabla u_{\va,\lambda}\|_{L^\infty(U_r(x_0))}
&\leq
\f{C(Ck)^{2k}}
{(1+|\lambda|r^2)^k r}
\left(\fint_{U_{2r}(x_0)}|u_{\va,\lambda}|^2\right)^{\f{1}{2}}
\\
&\quad\quad+
C(Ck)^{2k}
(1+|\lambda|r^2)^{N_0}
r\left(\fint_{U_{2r}(x_0)}|F|^p\right)^{\f{1}{p}},
\end{aligned}
\]
where, in the interior case, $U_\rho(x_0)$ is replaced by $B_\rho(x_0)$. The constant $C>0$ depends only on $\mu$, $n$, $m$, $\theta_0$, $p$, $\tau$, $\nu$, $\eta$, and $\cC_{1,\eta}(U)$.
\end{prop}

\begin{proof}
We prove only the boundary case; the interior case is identical, with $U_\rho(x_0)$ replaced by $B_\rho(x_0)$. Up to translation, we let $x_0=0$.

By the localized Lipschitz estimate for $\cL_\va-\lambda I$ in \cite[Theorem 3.5]{Wan23}, applied to the pair of domains $U_r\subset U_{\f{3r}{2}}$, there exists $N_1=N_1(n)\in\Z_+$ such that
\be
\begin{aligned}
\|\nabla u_{\va,\lambda}\|_{L^\infty(U_r)}
&\leq
\f{C}{r}
\left(\fint_{U_{\f{3r}{2}}}|u_{\va,\lambda}|^2\right)^{\f{1}{2}}
+
C(1+|\lambda|r^2)^{N_1}
r\left(\fint_{U_{\f{3r}{2}}}|F|^p\right)^{\f{1}{p}}.
\end{aligned}
\label{Lip-lambda-basic}
\ee
Here and below, $C>0$ depends only on the parameters stated in the proposition.

We next estimate the intermediate $L^2$ term by the Caccioppoli estimate. Since $p\in(n,+\ift)$, the local $L^p$ average of $F$ controls the local $L^q$ average appearing in Lemma \ref{Caccio-data}. Applying the same iteration argument as in Lemma \ref{Caccio-data} on the two radii $\f{3r}{2}<2r$, with $f=0$, gives, for any integer $k\in\Z_+$,
\be
\begin{aligned}
\left(\fint_{U_{\f{3r}{2}}}|u_{\va,\lambda}|^2\right)^{\f{1}{2}}
&\leq
\f{C(Ck)^{2k}}
{(1+|\lambda|r^2)^k}
\left(\fint_{U_{2r}}|u_{\va,\lambda}|^2\right)^{\f{1}{2}}
+
C(Ck)^{2k}
r^2
\left(\fint_{U_{2r}}|F|^p\right)^{\f{1}{p}}.
\end{aligned}
\label{Lip-lambda-Caccio}
\ee
Substituting \eqref{Lip-lambda-Caccio} into \eqref{Lip-lambda-basic}, and absorbing the lower-order contribution into the factor $(1+|\lambda|r^2)^{N_0}$ after increasing $N_0$, completes the proof.
\end{proof}

Finally, we need second-order estimates for the homogenized operator. These estimates are used when differentiating the homogenized Green's function in the first-order expansion.

\begin{lem}[Localized second-order estimates for $\cL_0-\lambda I$]\label{L0-second}
Let $n\geq2$, and let $\lambda\in\Sigma_{\theta_0}\cup\{0\}$ with $\theta_0\in(0,\f{\pi}{2})$. Let $U\subset\R^n$ be a bounded $C^{1,1}$ domain with parameters $\cC_{1,1}(U)$. Let $x_0\in\ol U$, $r\in(0,\f{r_0}{4})$, and $p\in(1,+\ift)$.

When $T_{2r}^{U}(x_0)\neq\emptyset$, assume that $u_{0,\lambda}\in H^1(U_{2r}(x_0);\C^m)$ is a weak solution of
\[
\left\{
\begin{aligned}
(\cL_0-\lambda I)u_{0,\lambda}
&=0
&&\text{in }U_{2r}(x_0),
\\
u_{0,\lambda}&=0
&&\text{on }T_{2r}^{U}(x_0).
\end{aligned}
\right.
\]
If $T_{2r}^{U}(x_0)=\emptyset$, assume that $B_{2r}(x_0)\subset U$ and $u_{0,\lambda}\in H^1(B_{2r}(x_0);\C^m)$ is a weak solution of
\[
(\cL_0-\lambda I)u_{0,\lambda}=0
\quad\text{in }B_{2r}(x_0).
\]
Then, for any $ k\in\Z_+$,
\[
\left(\fint_{U_r(x_0)}|\nabla^2 u_{0,\lambda}|^p\right)^{\f{1}{p}}
\leq
\f{C(Ck)^{2k}}
{(1+|\lambda|r^2)^k r^2}
\left(\fint_{U_{2r}(x_0)}|u_{0,\lambda}|^2\right)^{\f{1}{2}},
\]
where, in the interior case, $U_r(x_0)$ and $U_{2r}(x_0)$ are replaced by $B_r(x_0)$ and $B_{2r}(x_0)$, respectively. The constant $C$ depends only on $\mu$, $n$, $m$, $\theta_0$, $p$, and $\cC_{1,1}(U)$.

If, in addition, $U$ is a bounded $C^{2,1}$ domain with parameters $\cC_{2,1}(U)$, then, for any $\rho\in(0,1)$ and any $ k\in\Z_+$,
\[
\|\nabla^2 u_{0,\lambda}\|_{L^\infty(U_r(x_0))}
\leq
\f{C(Ck)^{2k}}
{(1+|\lambda|r^2)^k r^2}
\left(\fint_{U_{2r}(x_0)}|u_{0,\lambda}|^2\right)^{\f{1}{2}},
\]
and
\[
[\nabla^2 u_{0,\lambda}]_{C^{0,\rho}(U_r(x_0))}
\leq
\f{C(Ck)^{2k}}
{(1+|\lambda|r^2)^k r^{2+\rho}}
\left(\fint_{U_{2r}(x_0)}|u_{0,\lambda}|^2\right)^{\f{1}{2}},
\]
where, in the interior case, $U_r(x_0)$ and $U_{2r}(x_0)$ are replaced by $B_r(x_0)$ and $B_{2r}(x_0)$, respectively. In these two estimates, the constant $C$ depends only on $\mu$, $n$, $m$, $\theta_0$, $\rho$, and $\cC_{2,1}(U)$.
\end{lem}

\begin{proof}
We prove only the boundary case; the interior case is identical, with $U_r$ and $U_{2r}$ replaced by $B_r$ and $B_{2r}$. Up to translation, we let $x_0=0$.

We first prove the $W^{2,p}$ estimate. By the localized constant-coefficient second-order estimate in \cite[Lemma 4.3]{Wan23}, applied to the pair of domains $U_r\subset U_{\f{3r}{2}}$, there exists an integer $N_1=N_1(n)\in\Z_+$ such that
\be
\left(\fint_{U_r}|\nabla^2 u_{0,\lambda}|^p\right)^{\f{1}{p}}
\leq
\f{C(1+|\lambda|r^2)^{N_1}}{r^2}
\left(\fint_{U_{\f{3r}{2}}}|u_{0,\lambda}|^2\right)^{\f{1}{2}}.
\label{L0-second-basic}
\ee
This is the usual local $W^{2,p}$ estimate for the constant-coefficient operator $\cL_0$, applied to $\cL_0u_{0,\lambda}=\lambda u_{0,\lambda}$, in the form recorded in \cite[Lemma 4.3]{Wan23}.

It remains to estimate the intermediate $L^2$ norm. Applying Lemma \ref{Caccio-data} with $F=0$ and $f=0$, or equivalently applying the same iteration argument on the pair of radii $\f{3r}{2}<2r$, gives, for any integer $K\in\Z_+$,
\be
\left(\fint_{U_{\f{3r}{2}}}|u_{0,\lambda}|^2\right)^{\f{1}{2}}
\leq
\f{C(CK)^{2K}}
{(1+|\lambda|r^2)^K}
\left(\fint_{U_{2r}}|u_{0,\lambda}|^2\right)^{\f{1}{2}}.
\label{L0-second-Caccio}
\ee
Taking $K=k+N_1$ in \eqref{L0-second-Caccio} and substituting the result into \eqref{L0-second-basic}, we obtain
\[
\left(\fint_{U_r}|\nabla^2 u_{0,\lambda}|^p\right)^{\f{1}{p}}
\leq
\f{C(C(k+N_1))^{2(k+N_1)}}
{(1+|\lambda|r^2)^k r^2}
\left(\fint_{U_{2r}}|u_{0,\lambda}|^2\right)^{\f{1}{2}}.
\]
Since $N_1$ is fixed and $k\in\Z_+$, the factor $(C(k+N_1))^{2(k+N_1)}$ may be absorbed into $(Ck)^{2k}$ after increasing $C$. This proves the first estimate.

Assume now that $U$ is of class $C^{2,1}$. The localized constant-coefficient $C^2$ and $C^{2,\rho}$ estimates in \cite[Lemma 4.3]{Wan23}, applied again on $U_r\subset U_{\f{3r}{2}}$, give
\be
\|\nabla^2 u_{0,\lambda}\|_{L^\infty(U_r)}
\leq
\f{C(1+|\lambda|r^2)^{N_1}}{r^2}
\left(\fint_{U_{\f{3r}{2}}}|u_{0,\lambda}|^2\right)^{\f{1}{2}},
\label{L0-second-Linfty-basic}
\ee
and
\be
[\nabla^2 u_{0,\lambda}]_{C^{0,\rho}(U_r)}
\leq
\f{C(1+|\lambda|r^2)^{N_1}}{r^{2+\rho}}
\left(\fint_{U_{\f{3r}{2}}}|u_{0,\lambda}|^2\right)^{\f{1}{2}}.
\label{L0-second-Holder-basic}
\ee
Combining \eqref{L0-second-Linfty-basic} and \eqref{L0-second-Holder-basic} with \eqref{L0-second-Caccio}, and taking again $K=k+N_1$, yields
\[
\|\nabla^2 u_{0,\lambda}\|_{L^\infty(U_r)}
\leq
\f{C(Ck)^{2k}}
{(1+|\lambda|r^2)^k r^2}
\left(\fint_{U_{2r}}|u_{0,\lambda}|^2\right)^{\f{1}{2}},
\]
and
\[
[\nabla^2 u_{0,\lambda}]_{C^{0,\rho}(U_r)}
\leq
\f{C(Ck)^{2k}}
{(1+|\lambda|r^2)^k r^{2+\rho}}
\left(\fint_{U_{2r}}|u_{0,\lambda}|^2\right)^{\f{1}{2}}.
\]
This completes the proof.
\end{proof}

\section{Existence and Estimates of Green's Functions}\label{Green-functions-section}

In this section, we recall the existence of resolvent Green's functions and collect the pointwise estimates needed later. Throughout this section, if $\va\in[0,1]$, we write $ A_\va(x) $ as in \eqref{Avax}. Thus, $\cL_0=-\op{div}(\widehat A\nabla)$ is included in the notation below.

\begin{prop}[Green's functions of $\cL_\va-\lambda I$ with $n\geq3$]\label{Green-ngeq3}
Let $\va\in[0,1]$, $n\geq3$, and let
$\lambda\in\Sigma_{\theta_0}\cup\{0\}$ with $\theta_0\in(0,\f{\pi}{2})$.
Let $U\subset\R^n$ be a bounded $C^1$ domain with parameters
$\cC_1(U)$. Assume that $A$ satisfies \eqref{intro-symmetry}--\eqref{intro-holder}. Then there exists a unique Green's function
\[
G_{\va,\lambda}
=(G_{\va,\lambda}^{\al\beta})_{\al,\beta\in\Z\cap[1,m]}
:
U\times U\to \C^{m\times m}\cup\{\infty\}
\]
such that, for any $y\in U$ and any $0<r<r_0$,
\[
G_{\va,\lambda}(\cdot,y)
\in
H^1(U\backslash B_r(y);\C^{m\times m})
\cap
W_0^{1,s}(U;\C^{m\times m})
\quad
\text{for any }s\in\left[1,\f{n}{n-1}\).
\]
Moreover, for any $\ga\in\Z\cap[1,m]$ and any
$\phi\in W_0^{1,p}(U;\C^m)$ with $p\in(n,+\ift)$,
\[
\mathcal B_{\va,\lambda,U}[G_{\va,\lambda}^{\gamma}(\cdot,y),\phi]
=
\phi^\gamma(y),
\]
where $G_{\va,\lambda}^{\gamma}(\cdot,y)$ denotes the $\gamma$-th column
of $G_{\va,\lambda}(\cdot,y)$ and
\be
\mathcal B_{\va,\lambda,U}[u,v]
:=
\int_U\<A_{\va}\nabla u,\ol{\nabla v}\>
-\lambda\int_{U}u\ol{v},
\quad
u,v\in H^1(U;\C^{m}),
\label{defnBva}
\ee
is the bilinear form corresponding to $\cL_{\va}-\lambda I$.

In particular, if $F\in L^q(U;\C^m)$ with $q\in(\f{n}{2},+\ift)$, then
\[
u_{\va,\lambda}(x)
=
\int_U G_{\va,\lambda}(x,y)\overline{F(y)}\ud y
\]
is the weak solution of
\[
(\cL_\va-\lambda I)u_{\va,\lambda}=F
\quad\text{in }U,
\quad
u_{\va,\lambda}=0
\quad\text{on }\partial U.
\]
Furthermore, if $G_{\va,\ol\lambda}$ is the Green's function of
$\cL_\va-\ol\lambda I$, then
\[
G_{\va,\lambda}(x,y)=\left[\overline{G_{\va,\ol\lambda}(y,x)}\right]^T,
\]
that is,
\[
G_{\va,\lambda}^{\al\beta}(x,y)
=
\overline{
G_{\va,\ol\lambda}^{\beta\al}(y,x)}
\]
for any $\al,\beta\in\Z\cap[1,m]$ and any $x,y\in U$ with $x\neq y$.

For any $\sigma_1,\sigma_2\in(0,1)$ and any $ k\in\Z_+$, one has
\be
\left|G_{\va,\lambda}(x,y)\right|
\leq
\f{C(Ck)^{2k}}
{(1+|\lambda||x-y|^2)^k |x-y|^{n-2}}
\min\left\{
1,\f{\delta(x)^{\sigma_1}}{|x-y|^{\sigma_1}},
\f{\delta(y)^{\sigma_2}}{|x-y|^{\sigma_2}}
\right\}
\label{Gvaldangeq3}
\ee
for any $x,y\in U$ with $x\neq y$, where $\delta(x)=\dist(x,\pa U)$. The constant $C>0$ depends only on $\mu$, $n$, $m$, $\theta_0$, $\sigma_1$, $\sigma_2$, $\tau$, $\nu$, and $\cC_1(U)$.
\end{prop}

\begin{proof}
The existence statement and the representation formula follow from \cite[Theorem 4.8]{Wan23}. It remains to prove the refined estimate \eqref{Gvaldangeq3}, with the explicit dependence on $k$.

Fix $x,y\in U$ with $x\neq y$ and set $r=|x-y|$. Since $y\notin U_{\f r3}(x)$, \cite[Equation (4.40)]{Wan23}, applied with the fixed value $k=1$, gives
\be
\left(\fint_{U_{\f r3}(x)}
|G_{\va,\lambda}(\zeta,y)|^2\ud\zeta
\right)^{\f12}
\leq
Cr^{2-n}.
\label{leqCrn2}
\ee
Applying the local $L^\infty$ estimate \eqref{Linftyestimate} to
$G_{\va,\lambda}(\cdot,y)$ in $U_{\f r3}(x)$, we obtain, for any $k\in\Z_+$,
\be
\begin{aligned}
|G_{\va,\lambda}(x,y)|
&\leq
\f{C(Ck)^{2k}}
{(1+|\lambda|r^2)^k}
\left(\fint_{U_{\f r3}(x)}
|G_{\va,\lambda}(\zeta,y)|^2\ud\zeta
\right)^{\f12}
\\
&\leq
\f{C(Ck)^{2k}}
{(1+|\lambda||x-y|^2)^k |x-y|^{n-2}} .
\end{aligned}
\label{GvaldaLinfty}
\ee

We next obtain the boundary factor in the $x$-variable. By \eqref{HolderEstimate} and \eqref{leqCrn2},
\be
\begin{aligned}
\left[G_{\va,\lambda}(\cdot,y)\right]_{C^{0,\sigma_1}(U_{\f r6}(x))}
&\leq
\f{C(Ck)^{2k}}
{(1+|\lambda|r^2)^k r^{\sigma_1}}
\left(\fint_{U_{\f r3}(x)}
|G_{\va,\lambda}(\zeta,y)|^2\ud\zeta
\right)^{\f12}
\\
&\leq
\f{C(Ck)^{2k}}
{(1+|\lambda||x-y|^2)^k |x-y|^{n-2+\sigma_1}} .
\end{aligned}
\label{Gvaldasigma1}
\ee
If $\delta(x)\geq\f r6$, the estimate with the factor
$\delta(x)^{\sigma_1}|x-y|^{-\sigma_1}$ follows immediately from
\eqref{GvaldaLinfty} after increasing $C$. If $\delta(x)<\f r6$, choose
$\ol{x}\in\partial U$ such that $|x-\ol{x}|=\delta(x)$. Since
$G_{\va,\lambda}(\cdot,y)=0$ on $\partial U$ in the trace sense and
$\ol{x}\in U_{\f r6}(x)$ in the local boundary coordinates, \eqref{Gvaldasigma1} gives
\[
\begin{aligned}
|G_{\va,\lambda}(x,y)|
&=
|G_{\va,\lambda}(x,y)-G_{\va,\lambda}(\ol{x},y)|
\\
&\leq
\left[G_{\va,\lambda}(\cdot,y)\right]_{C^{0,\sigma_1}(U_{\f r6}(x))}
|x-\ol{x}|^{\sigma_1}
\\
&\leq
\f{C(Ck)^{2k}}
{(1+|\lambda||x-y|^2)^k}
\f{\delta(x)^{\sigma_1}}{|x-y|^{n-2+\sigma_1}} .
\end{aligned}
\]
The same argument applied to $G_{\va,\ol\lambda}$, followed by the duality relation, gives the corresponding boundary factor in the $y$-variable. Combining these two estimates with \eqref{GvaldaLinfty} proves \eqref{Gvaldangeq3}.
\end{proof}

The two-dimensional case is slightly different because the Green's function has logarithmic growth rather than a power singularity. We record the corresponding statement separately.

\begin{thm}[Green's functions of $\cL_\va-\lambda I$ with $n=2$]\label{Green-n2}
Let $\va\in[0,1]$, $n=2$, and let
$\lambda\in\Sigma_{\theta_0}\cup\{0\}$ with $\theta_0\in(0,\f{\pi}{2})$.
Let $U\subset\R^2$ be a bounded $C^1$ domain with parameters
$\cC_1(U)$. Assume that $A$ satisfies
\eqref{intro-symmetry}--\eqref{intro-holder}. Then there exists a unique
Green's function
\[
G_{\va,\lambda}
=
(G_{\va,\lambda}^{\al\beta})_{\al,\beta\in\Z\cap[1,m]}
:
U\times U\to \C^{m\times m}\cup\{\infty\}
\]
such that $G_{\va,\lambda}(\cdot,y)\in \mathrm{BMO}(U;\C^{m\times m})$ and
\[
\|G_{\va,\lambda}(\cdot,y)\|_{\mathrm{BMO}(U)}
\leq C
\]
uniformly for $y\in U$, where $C>0$ depends only on $\mu$, $m$, $\theta_0$, $\tau$, $\nu$, and $\cC_1(U)$.

Moreover, if $u_{\va,\lambda}$ is the weak solution of
\[
(\cL_\va-\lambda I)u_{\va,\lambda}=F
\quad\text{in }U,
\quad
u_{\va,\lambda}=0
\quad\text{on }\partial U,
\]
where $F\in L^q(U;\C^m)$ for some $q>1$, then
\[
u_{\va,\lambda}(x)
=
\int_U G_{\va,\lambda}(x,y)\overline{F(y)}\ud y.
\]
Furthermore, if $G_{\va,\ol\lambda}$ is the Green's function of
$\cL_\va-\ol\lambda I$, then
\be
G_{\va,\lambda}(x,y)
=
\left[
\overline{G_{\va,\ol\lambda}(y,x)}
\right]^T.
\label{duality2}
\ee

Let $R_U=\op{diam}(U)$ and $\delta(x)=\dist(x,\pa U)$. For any
$\sigma\in(0,1)$, the following estimate holds for any $x,y\in U$ with
$x\neq y$:
\[
\left|G_{\va,\lambda}(x,y)\right|
\leq
C
\left\{
1+|\lambda|^{\f{\sigma}{2}}R_U^\sigma
+\log\left(\f{R_U}{|x-y|}\right)
\right\}.
\]
Here $C>0$ depends only on $\mu$, $m$, $\theta_0$, $\sigma$, $\tau$, $\nu$, and $\cC_1(U)$.
\end{thm}

\begin{proof}
The result follows from \cite[Theorem 4.8]{Wan23}.
\end{proof}

We next derive gradient estimates for the Green's functions. These estimates are the pointwise input for the convergence estimates in the next section.

\begin{prop}[Gradient estimates for Green's functions]\label{Green-gradient}
Let $\va\in[0,1]$, $n\geq2$, and let
$\lambda\in\Sigma_{\theta_0}\cup\{0\}$ with $\theta_0\in(0,\f{\pi}{2})$.
Let $U\subset\R^n$ be a bounded $C^{1,\eta}$ domain with parameters
$\cC_{1,\eta}(U)$, where $\eta\in(0,1)$. Assume that $A$ satisfies
\eqref{intro-symmetry}--\eqref{intro-holder}. Then, for any $ k\in\Z_+$, the Green's function $G_{\va,\lambda}$ of
$\cL_\va-\lambda I$ satisfies
\begin{align}
|\nabla_1G_{\va,\lambda}(x,y)|+|\nabla_2G_{\va,\lambda}(x,y)|
&\leq
\f{C(Ck)^{2k}}{(1+|\lambda||x-y|^2)^k |x-y|^{n-1}},
\label{Lipschitz1}
\\
|\nabla_1G_{\va,\lambda}(x,y)|
&\leq
\f{C(Ck)^{2k}\delta(y)}
{(1+|\lambda||x-y|^2)^k |x-y|^{n}},
\label{Lipschitz2}
\\
|\nabla_2G_{\va,\lambda}(x,y)|
&\leq
\f{C(Ck)^{2k}\delta(x)}
{(1+|\lambda||x-y|^2)^k |x-y|^{n}},
\label{Lipschitz3}
\\
|\nabla_1\nabla_2G_{\va,\lambda}(x,y)|
&\leq
\f{C(Ck)^{2k}}
{(1+|\lambda||x-y|^2)^k |x-y|^{n}}
\label{Lipschitz4}
\end{align}
for any $x,y\in U$ with $x\neq y$, where
\[
\delta(x)=\op{dist}(x,\partial U).
\]
Here $C>0$ depends only on $\mu$, $n$, $m$, $\theta_0$, $\tau$, $\nu$, $\eta$, and $\cC_{1,\eta}(U)$.
\end{prop}

\begin{proof}
If $n\geq3$, the estimates \eqref{Lipschitz1}--\eqref{Lipschitz4} follow from the Green-function bound \eqref{Gvaldangeq3}, the localized Lipschitz estimate in Proposition \ref{Lip-lambda}, and the duality relation. We therefore discuss only the two-dimensional case.

Let $x_0,y_0\in U$ with $x_0\neq y_0$, and set $r=|x_0-y_0|$. By \cite[Equation (4.67)]{Wan23},
\[
\left(\fint_{U_{\f r4}(x_0)}
|\nabla_2G_{\va,\lambda}(\zeta,y_0)|^2\ud\zeta
\right)^{\f12}
\leq
\f{C\delta(y_0)}{|x_0-y_0|}.
\]
Applying Proposition \ref{Lip-lambda} in the $x$-variable to
$G_{\va,\lambda}(\cdot,y_0)$ on $U_{\f r4}(x_0)$ gives
\eqref{Lipschitz2}. The estimate \eqref{Lipschitz3} follows from
\eqref{Lipschitz2} and the duality relation \eqref{duality2}.

It remains to prove the mixed derivative estimate. By \cite[Equation (4.69)]{Wan23},
\[
\left(\fint_{U_{\f r4}(x_0)}
|\nabla_2G_{\va,\lambda}(\zeta,y_0)|^2\ud\zeta
\right)^{\f12}
\leq
\f{C(1+|\lambda|r^2)^{N_0}}{r},
\]
where $N_0=N_0(n)>0$. Applying Proposition \ref{Lip-lambda} to
$\nabla_2G_{\va,\lambda}(\cdot,y_0)$ with the integer $k+N_0$ and then absorbing the fixed power $N_0$ into the constant, we obtain
\[
\begin{aligned}
|\nabla_1\nabla_2G_{\va,\lambda}(x_0,y_0)|
&\leq
\f{C(C(k+N_0))^{2(k+N_0)}}
{(1+|\lambda|r^2)^{k+N_0}r}
\left(\fint_{U_{\f r4}(x_0)}
|\nabla_2G_{\va,\lambda}(\zeta,y_0)|^2\ud\zeta
\right)^{\f12}
\\
&\leq
\f{C(Ck)^{2k}}
{(1+|\lambda|r^2)^k r^2}.
\end{aligned}
\]
Since $n=2$ and $r=|x_0-y_0|$, this is \eqref{Lipschitz4}. Finally, \eqref{Lipschitz1} follows from \eqref{Lipschitz2}--\eqref{Lipschitz4} by the same argument as in \cite[pp. 44--45]{Wan23}.
\end{proof}

\section{Convergence of Green's Functions}\label{ConGreenResolvent}

We now prove the resolvent Green-function convergence estimates. The first result is the zeroth-order estimate. The second gives the first-order approximation after insertion of the Dirichlet corrector.

\begin{prop}[Convergence of Green's functions I]\label{Green-convergence-I}
Let $\va\in(0,1]$, $n\geq2$, and let
$\lambda\in\Sigma_{\theta_0}\cup\{0\}$ with $\theta_0\in(0,\f{\pi}{2})$.
Let $U\subset\R^n$ be a bounded $C^{1,1}$ domain with parameters
$\cC_{1,1}(U)$. Assume that $A$ satisfies \eqref{intro-symmetry}--\eqref{intro-holder}. Then, for any $ k\in\Z_+$ and any $x,y\in U$ with $x\neq y$,
\[
\left|G_{\va,\lambda}(x,y)-G_{0,\lambda}(x,y)\right|
\leq
\f{C(Ck)^{2k}\va}
{(1+|\lambda||x-y|^2)^k |x-y|^{n-1}} .
\]
Here $C>0$ depends only on
$\mu$, $n$, $m$, $\theta_0$, $\tau$, $\nu$, and $\cC_{1,1}(U)$.
\end{prop}

\begin{prop}[Convergence of Green's functions II]\label{Green-convergence-II}
Let $\va\in(0,1]$, $n\geq2$, and let
$\lambda\in\Sigma_{\theta_0}\cup\{0\}$ with $\theta_0\in(0,\f{\pi}{2})$.
Let $U\subset\R^n$ be a bounded $C^{2,1}$ domain with parameters
$\cC_{2,1}(U)$. Assume that $A$ satisfies \eqref{intro-symmetry}--\eqref{intro-holder}. Then, for any $ k\in\Z_+$, any $\al,\ga\in\Z\cap[1,m]$, and any
$x,y\in U$ with $x\neq y$,
\be
\left|
\f{\partial}{\partial x_i}
G_{\va,\lambda}^{\al\gamma}(x,y)
-
\f{\partial}{\partial x_i}
\Phi_{\va,j}^{\al\beta}(x)
\f{\partial}{\partial x_j}
G_{0,\lambda}^{\beta\gamma}(x,y)
\right|
\leq
\f{
C(Ck)^{2k}\va
\log\left(\va^{-1}|x-y|+2\right)
}{
(1+|\lambda||x-y|^2)^k |x-y|^n
}.
\label{Green-convergence-IIineq}
\ee
Here and throughout, the repeated indices $\beta$ and $j$ are summed over. The constant $C>0$ depends only on $\mu$, $n$, $m$, $\theta_0$, $\tau$, $\nu$, and $\cC_{2,1}(U)$.
\end{prop}

Before proving these two propositions, we record two auxiliary estimates. The first is a boundary $L^p$ estimate for the non-tangential maximal function, and the second is a localized $L^\infty$ estimate with non-zero boundary data.

For a continuous function $u$ in $U$, we define its non-tangential maximal
function by
\[
(u)^*(y)
=
\sup\left\{
|u(x)|:\ x\in U\text{ and }|x-y|<C_0\op{dist}(x,\partial U)
\right\},
\quad y\in\partial U,
\]
where $C_0=C_0(U)>1$ is sufficiently large.

\begin{thm}[Non-tangential maximal function estimates]\label{NTmax-estimate}
Let $\va\in[0,1]$, $n\geq2$, and let $\lambda\in\Sigma_{\theta_0}\cup\{0\}$ with $\theta_0\in(0,\f{\pi}{2})$. Let $U\subset\R^n$ be a bounded $C^{1,\eta}$ domain with parameters $\cC_{1,\eta}(U)$, where $\eta\in(0,1)$. Assume that $A$ satisfies
\eqref{intro-symmetry}--\eqref{intro-holder}. Let $p\in(1,+\ift]$ and let $g\in L^p(\partial U;\C^m)$.
Let $u_{\va,\lambda}$ be the unique solution of the Dirichlet problem
\[
(\cL_\va-\lambda I)u_{\va,\lambda}=0
\quad\text{in }U,
\quad
u_{\va,\lambda}=g
\quad\text{on }\partial U,
\]
with the property $(u_{\va,\lambda})^*\in L^p(\partial U)$. Then
\[
\left\|(u_{\va,\lambda})^*\right\|_{L^p(\partial U)}
\leq
C\|g\|_{L^p(\partial U)}.
\]
In particular, when $p=+\ift$,
\be
\|u_{\va,\lambda}\|_{L^\infty(U)}
\leq
C\|g\|_{L^\infty(\partial U)}.
\label{MaximumPrinciple}
\ee
Here $C>0$ depends only on $\mu$, $n$, $m$, $\theta_0$, $\tau$, $\nu$, $\eta$, and
$\cC_{1,\eta}(U)$.
\end{thm}

\begin{proof}
This is \cite[Theorem 5.2]{Wan23}.
\end{proof}

The following localized estimate will be used in the approximation argument near the boundary. It reduces non-zero boundary data to the homogeneous case by means of the preceding non-tangential maximal function estimate.

\begin{lem}[Localized $L^\infty$ estimate]\label{local-Linfty-lambda}
Let $\va\in[0,1]$, $n\geq2$, and let $\lambda\in\Sigma_{\theta_0}\cup\{0\}$ with $\theta_0\in(0,\f{\pi}{2})$. Let $U\subset\R^n$ be a bounded $C^{1,1}$ domain with parameters $\cC_{1,1}(U)$. Assume that $A$ satisfies \eqref{intro-symmetry}--\eqref{intro-holder}. Let $x_0\in\ol U$ and $r\in(0,\f{r_0}{3})$.

When $T_{3r}^{U}(x_0)\neq\emptyset$, assume that
$u_{\va,\lambda}\in H^1(U_{3r}(x_0);\C^m)$ is a weak solution of
\[
\left\{
\begin{aligned}
(\cL_\va-\lambda I)u_{\va,\lambda}
&=0
&&\text{in }U_{3r}(x_0),
\\
u_{\va,\lambda}&=f
&&\text{on }T_{3r}^{U}(x_0),
\end{aligned}
\right.
\]
with $\|f\|_{L^\infty(T_{3r}^{U}(x_0))}<+\ift$.
If $T_{3r}^{U}(x_0)=\emptyset$, assume that $B_{3r}(x_0)\subset U$ and
$u_{\va,\lambda}\in H^1(B_{3r}(x_0);\C^m)$ is a weak solution of
\[
(\cL_\va-\lambda I)u_{\va,\lambda}=0
\quad\text{in }B_{3r}(x_0).
\]
Then, for any $ k\in\Z_+$,
\[
\|u_{\va,\lambda}\|_{L^\infty(U_r(x_0))}
\leq
\f{C(Ck)^{2k}}
{(1+|\lambda|r^2)^k}
\left\{
\fint_{U_{3r}(x_0)}|u_{\va,\lambda}|
+
\|f\|_{L^\infty(T_{3r}^{U}(x_0))}
\right\}.
\]
In the interior case, $U_\rho(x_0)$ is replaced by $B_\rho(x_0)$ and the boundary term involving $f$ is omitted. Here $C>0$ depends only on
$\mu$, $n$, $m$, $\theta_0$, $\tau$, $\nu$, and $\cC_{1,1}(U)$.
\end{lem}

\begin{proof}
In the interior case, or in the boundary case with $f\equiv0$, the estimate follows from Proposition \ref{Holder-Linfty-data}, with $s=1$. We therefore consider the boundary case with non-zero boundary data.

Choose a bounded $C^{1,1}$ domain $\widetilde U$ such that
\[
U_r(x_0)\subset \widetilde U\subset U_{2r}(x_0)
\]
and such that the $C^{1,1}$ character of $\widetilde U$ is controlled by $\cC_{1,1}(U)$. Let $v_{\va,\lambda}$ solve
\[
(\cL_\va-\lambda I)v_{\va,\lambda}=0
\quad\text{in }\widetilde U,
\]
with boundary values
\[
v_{\va,\lambda}=f
\quad\text{on }\partial\widetilde U\cap\partial U,
\quad
v_{\va,\lambda}=0
\quad\text{on }\partial\widetilde U\backslash\partial U.
\]
By \eqref{MaximumPrinciple},
\[
\|v_{\va,\lambda}\|_{L^\infty(\widetilde U)}
\leq
C\|v_{\va,\lambda}\|_{L^\infty(\partial\widetilde U)}
\leq
C\|f\|_{L^\infty(T_{3r}^{U}(x_0))}.
\]
Since $u_{\va,\lambda}-v_{\va,\lambda}$ has zero boundary data on
$\partial\widetilde U\cap\partial U$, the homogeneous case gives
\[
\begin{aligned}
\|u_{\va,\lambda}-v_{\va,\lambda}\|_{L^\infty(U_r(x_0))}
&\leq
\f{C(Ck)^{2k}}
{(1+|\lambda|r^2)^k}
\fint_{\widetilde U}
|u_{\va,\lambda}-v_{\va,\lambda}|
\\
&\leq
\f{C(Ck)^{2k}}
{(1+|\lambda|r^2)^k}
\left\{
\fint_{U_{3r}(x_0)}|u_{\va,\lambda}|
+
\|f\|_{L^\infty(T_{3r}^{U}(x_0))}
\right\}.
\end{aligned}
\]
Combining this estimate with the bound for $v_{\va,\lambda}$ proves the desired estimate.
\end{proof}

We now state the local approximation estimate that compares a resolvent solution with its homogenized counterpart. This is the main local input for the zeroth-order convergence of Green's functions.

\begin{lem}[Localized approximation estimate]\label{local-approx-lambda}
Let $\va\in(0,1]$, $n\geq2$, and let
$\lambda\in\Sigma_{\theta_0}\cup\{0\}$ with
$\theta_0\in(0,\f{\pi}{2})$. Let $U\subset\R^n$ be a bounded
$C^{1,1}$ domain with parameters $\cC_{1,1}(U)$. Assume that $A$
satisfies \eqref{intro-symmetry}--\eqref{intro-holder}. Let $x_0\in\ol U$,
$r\in(0,\f{r_0}{4})$, and let $p\in(n,+\ift)$. Assume also that $\va\in(0,r)$.

Assume that
\[
u_{\va,\lambda}\in H^1(U_{4r}(x_0);\C^m),
\quad
u_{0,\lambda}\in W^{2,p}(U_{4r}(x_0);\C^m).
\]
When $T_{4r}^{U}(x_0)\neq\emptyset$, assume that
\[
\left\{
\begin{aligned}
(\cL_\va-\lambda I)u_{\va,\lambda}
&=
(\cL_0-\lambda I)u_{0,\lambda}
&&\text{in }U_{4r}(x_0),
\\
u_{\va,\lambda}&=u_{0,\lambda}
&&\text{on }T_{4r}^{U}(x_0).
\end{aligned}
\right.
\]
If $T_{4r}^{U}(x_0)=\emptyset$, assume that $B_{4r}(x_0)\subset U$ and
\[
(\cL_\va-\lambda I)u_{\va,\lambda}
=
(\cL_0-\lambda I)u_{0,\lambda}
\quad\text{in }B_{4r}(x_0).
\]
Then, for any $ k\in\Z_+$,
\be
\begin{aligned}
&\|u_{\va,\lambda}-u_{0,\lambda}\|_{L^\infty(U_r(x_0))}\\
&\quad\leq
\f{C(Ck)^{2k}}
{(1+|\lambda|r^2)^k}
\fint_{U_{4r}(x_0)}
|u_{\va,\lambda}-u_{0,\lambda}|
\\
&\quad\quad
+
\f{C(Ck)^{2k}\va}
{(1+|\lambda|r^2)^k}
\left\{
r^{1-\f{n}{p}}
\|\nabla^2 u_{0,\lambda}\|_{L^p(U_{4r}(x_0))}+
(1+|\lambda|r^2)
\|\nabla u_{0,\lambda}\|_{L^\infty(U_{4r}(x_0))}
\right\}.
\end{aligned}
\label{local-approx-lambdaineq}
\ee
In the interior case, each local set $U_s(x_0)$ is replaced by $B_s(x_0)$. Here $C>0$ depends only on
$\mu$, $n$, $m$, $\theta_0$, $p$, $\tau$, $\nu$, and $\cC_{1,1}(U)$.
\end{lem}

\begin{proof}
The proof follows from the argument of \cite[Lemma 5.4]{Wan23}. We indicate the modifications needed to keep track of the dependence on $\lambda$ and on the integer $k$.

In the proof of \cite[Lemma 5.4]{Wan23}, the local $L^\infty$ estimate is applied to the comparison error on a pair of comparable domains. In the present resolvent setting, this step is replaced by Lemma \ref{local-Linfty-lambda}. Therefore, each occurrence of the constant $C_{k,\theta_0}$ in \cite[Lemma 5.4]{Wan23} is replaced by
\[
\f{C(Ck)^{2k}}{(1+|\lambda|r^2)^k}.
\]
This gives the first term in \eqref{local-approx-lambdaineq}.

The corrector error terms are estimated as in \cite[Lemma 5.4]{Wan23}, using the Dirichlet corrector estimates and the Green-function bounds in Proposition \ref{Green-gradient}. With the above replacement of $C_{k,\theta_0}$, these terms give
\[
\f{C(Ck)^{2k}\va}
{(1+|\lambda|r^2)^k}
r^{1-\f{n}{p}}
\|\nabla^2 u_{0,\lambda}\|_{L^p(U_{4r}(x_0))}.
\]
The additional zero-order term involving $\lambda$ is controlled by
\[
\f{C(Ck)^{2k}\va}
{(1+|\lambda|r^2)^k}
|\lambda|r^2
\|\nabla u_{0,\lambda}\|_{L^\infty(U_{4r}(x_0))}.
\]
Combining this with the usual first-order corrector contribution
\[
\f{C(Ck)^{2k}\va}
{(1+|\lambda|r^2)^k}
\|\nabla u_{0,\lambda}\|_{L^\infty(U_{4r}(x_0))}
\]
yields the last term in \eqref{local-approx-lambdaineq}. This proves the desired estimate. The interior case is identical, with $U_s(x_0)$ replaced by $B_s(x_0)$.
\end{proof}

The first-order convergence estimate requires a refined local approximation in which the oscillatory gradient is compared with the homogenized gradient multiplied by the Dirichlet corrector.

\begin{lem}[Localized first-order approximation estimate]\label{local-gradient-approx-lambda}
Let $\va\in(0,1]$, $n\geq2$, and let
$\lambda\in\Sigma_{\theta_0}\cup\{0\}$ with
$\theta_0\in(0,\f{\pi}{2})$. Let $U\subset\R^n$ be a bounded
$C^{2,1}$ domain with parameters $\cC_{2,1}(U)$. Assume that $A$
satisfies \eqref{intro-symmetry}--\eqref{intro-holder}. Let $x_0\in\ol U$,
$r\in(0,\f{r_0}{4})$, and let $\rho\in(0,1)$. Assume also that $0<\va<r$.

Assume that
\[
u_{\va,\lambda}\in H^1(U_{4r}(x_0);\C^m),
\quad
u_{0,\lambda}\in C^{2,\rho}(U_{4r}(x_0);\C^m).
\]
When $T_{4r}^{U}(x_0)\neq\emptyset$, assume that
\[
\left\{
\begin{aligned}
(\cL_\va-\lambda I)u_{\va,\lambda}
&=
(\cL_0-\lambda I)u_{0,\lambda}
&&\text{in }U_{4r}(x_0),
\\
u_{\va,\lambda}&=u_{0,\lambda}
&&\text{on }T_{4r}^{U}(x_0).
\end{aligned}
\right.
\]
If $T_{4r}^{U}(x_0)=\emptyset$, assume that $B_{4r}(x_0)\subset U$ and
\[
(\cL_\va-\lambda I)u_{\va,\lambda}
=
(\cL_0-\lambda I)u_{0,\lambda}
\quad\text{in }B_{4r}(x_0).
\]
Then, for any $1\leq\al\leq m$ and any $ k\in\Z_+$,
\be
\begin{aligned}
&\left\|
\f{\partial u_{\va,\lambda}^{\al}}{\partial x_i}
-
\f{\partial\Phi_{\va,j}^{\al\beta}}{\partial x_i}
\f{\partial u_{0,\lambda}^{\beta}}{\partial x_j}
\right\|_{L^\infty(U_r(x_0))}
\\
&\quad\leq
\f{C(Ck)^{2k}}
{(1+|\lambda|r^2)^k r}
\fint_{U_{4r}(x_0)}
|u_{\va,\lambda}-u_{0,\lambda}|
\\
&\quad\quad
+
\f{C(Ck)^{2k}\va}
{(1+|\lambda|r^2)^k}
\Big\{
(1+|\lambda|r^2)r^{-1}
\|\nabla u_{0,\lambda}\|_{L^\infty(U_{4r}(x_0))}
\\
&\quad\quad\quad\quad\quad\quad+\log\left(\va^{-1}r+2\right)
\|\nabla^2 u_{0,\lambda}\|_{L^\infty(U_{4r}(x_0))}
+r^\rho
[\nabla^2 u_{0,\lambda}]_{C^{0,\rho}(U_{4r}(x_0))}
\Big\}.
\end{aligned}
\label{local-gradient-approx-lambdaineq}
\ee
In the interior case, each local set $U_s(x_0)$ is replaced by $B_s(x_0)$. Here and throughout, the repeated indices $\beta$ and $j$ are summed over. The constant $C>0$ depends only on
$\mu$, $n$, $m$, $\theta_0$, $\tau$, $\nu$, $\rho$, and
$\cC_{2,1}(U)$.
\end{lem}

\begin{proof}
The proof follows the proof of \cite[Lemma 5.6]{Wan23}. We explain only the changes caused by the resolvent parameter and by the refined tracking of the constants.

The local H\"older and $L^\infty$ estimates used in
\cite[Lemma 5.6]{Wan23} are replaced here by Proposition
\ref{Holder-Linfty-data} and Lemma \ref{local-Linfty-lambda}; the local
Lipschitz estimate is replaced by Proposition \ref{Lip-lambda}; and the
Green-function estimates are replaced by Proposition \ref{Green-gradient}.
These estimates have the same structure as the estimates used in
\cite[Lemma 5.6]{Wan23}, except that the constant $C_{k,\theta_0}$ is now
tracked as
\[
\f{C(Ck)^{2k}}{(1+|\lambda|r^2)^k}.
\]
Thus the homogeneous part of the comparison argument gives
\[
\f{C(Ck)^{2k}}
{(1+|\lambda|r^2)^k r}
\fint_{U_{4r}(x_0)}
|u_{\va,\lambda}-u_{0,\lambda}|.
\]

The first-order corrector error terms are estimated exactly as in
\cite[Lemma 5.6]{Wan23}, using the estimates for the Dirichlet correctors.
Keeping the above resolvent factor throughout the argument gives
\[
\f{C(Ck)^{2k}\va}
{(1+|\lambda|r^2)^k}
\Big\{
\log\left(\va^{-1}r+2\right)
\|\nabla^2 u_{0,\lambda}\|_{L^\infty(U_{4r}(x_0))}
+
r^\rho
[\nabla^2 u_{0,\lambda}]_{C^{0,\rho}(U_{4r}(x_0))}
\Big\}.
\]
The zero-order term involving $\lambda$ contributes
\[
\f{C(Ck)^{2k}\va}
{(1+|\lambda|r^2)^k}
|\lambda|r
\|\nabla u_{0,\lambda}\|_{L^\infty(U_{4r}(x_0))}.
\]
Together with the usual first-order corrector contribution
\[
\f{C(Ck)^{2k}\va}
{(1+|\lambda|r^2)^k}
r^{-1}
\|\nabla u_{0,\lambda}\|_{L^\infty(U_{4r}(x_0))},
\]
this gives the second term on the right-hand side of \eqref{local-gradient-approx-lambdaineq}. Combining the preceding bounds yields the desired estimate. The interior case is identical, with $U_s(x_0)$ replaced by $B_s(x_0)$.
\end{proof}

We now prove the zeroth-order Green-function convergence. The proof combines a duality estimate from the elliptic theory with the localized approximation lemma above.

\begin{proof}[Proof of Proposition \ref{Green-convergence-I}]
Fix $x_0,y_0\in U$ with $x_0\neq y_0$, and set $ r:=\f{|x_0-y_0|}{16} $. Then $x_0\notin U_{8r}(y_0)$. We prove the estimate at the point $(x_0,y_0)$.

We first recall the duality estimate obtained in the proof of \cite[Theorem 5.1]{Wan23}. Since the equations considered here are the same, the same argument gives, for any $p\in(n,+\ift)$,
\[
\left(
\int_{U_{4r}(y_0)}
|G_{\va,\lambda}(x_0,y)-G_{0,\lambda}(x_0,y)|^{p'}\ud y
\right)^{\f{1}{p'}}
\leq
C\va r^{1-\f{n}{p}},
\]
where $p'=\f{p}{p-1}$. Hence, by H\"older's inequality and the volume comparability of $U_{4r}(y_0)$,
\be
\fint_{U_{4r}(y_0)}
|G_{\va,\lambda}(x_0,y)-G_{0,\lambda}(x_0,y)|\ud y
\leq
C\va r^{1-n}.
\label{dualityuse}
\ee

We apply Lemma \ref{local-approx-lambda} in the $y$-variable to
\[
u_{\va,\ol\lambda}(y)=G_{\va,\lambda}(x_0,y),
\quad
u_{0,\ol\lambda}(y)=G_{0,\lambda}(x_0,y),
\]
on the pair of domains $U_r(y_0)\subset U_{4r}(y_0)$. This is legitimate because $x_0\notin U_{4r}(y_0)$ and, by the symmetry relation for Green's functions, these functions solve the corresponding adjoint homogeneous equations with parameter $\ol\lambda$ in $U_{4r}(y_0)$. Since only $|\lambda|$ appears in the estimates, Lemma \ref{local-approx-lambda} yields
\be
\begin{aligned}
&|G_{\va,\lambda}(x_0,y_0)-G_{0,\lambda}(x_0,y_0)|
\\
&\leq
\f{C(Ck)^{2k}}
{(1+|\lambda|r^2)^k}
\fint_{U_{4r}(y_0)}
|G_{\va,\lambda}(x_0,y)-G_{0,\lambda}(x_0,y)|\ud y
\\
&\quad
+
\f{C(Ck)^{2k}\va}
{(1+|\lambda|r^2)^k}
\left\{
r^{1-\f{n}{p}}
\|\nabla_y^2G_{0,\lambda}(x_0,\cdot)\|_{L^p(U_{4r}(y_0))}
+
(1+|\lambda|r^2)
\|\nabla_yG_{0,\lambda}(x_0,\cdot)\|_{L^\infty(U_{4r}(y_0))}
\right\}.
\end{aligned}
\label{GvaldaG0lda}
\ee
Since $|x_0-y|\sim r$ for $y\in U_{4r}(y_0)$, the gradient estimate \eqref{Lipschitz1}, applied with $\va=0$ and with the fixed value $k=1$, gives
\be
(1+|\lambda|r^2)
\|\nabla_yG_{0,\lambda}(x_0,\cdot)\|_{L^\infty(U_{4r}(y_0))}
\leq
Cr^{1-n}.
\label{thethirdterm}
\ee

Next we estimate the second-order term. If $n\geq3$, then \eqref{Gvaldangeq3}, applied with $\va=0$ and with the fixed value $k=1$, gives
\[
\left(\fint_{U_{8r}(y_0)}
|G_{0,\lambda}(x_0,y)|^2\ud y
\right)^{\f12}
\leq
Cr^{2-n}.
\]
Applying Lemma \ref{L0-second} on comparable boundary or interior balls contained in $U_{8r}(y_0)$ and covering $U_{4r}(y_0)$, we obtain
\[
\left(\fint_{U_{4r}(y_0)}
|\nabla_y^2G_{0,\lambda}(x_0,y)|^p\ud y
\right)^{\f1p}
\leq
Cr^{-n}.
\]
Equivalently,
\[
r^{1-\f{n}{p}}
\|\nabla_y^2G_{0,\lambda}(x_0,\cdot)\|_{L^p(U_{4r}(y_0))}
\leq
Cr^{1-n}.
\]

When $n=2$, the same bound follows from the two-dimensional argument in \cite[Theorem 5.1]{Wan23}, after formula (5.17), using the interior or boundary Lipschitz estimate for $\nabla_yG_{0,\lambda}(x_0,\cdot)$ according as $T_{4r}^{U}(y_0)=\emptyset$ or not. In the present notation, this gives
\[
r^{1-\f{2}{p}}
\|\nabla_y^2G_{0,\lambda}(x_0,\cdot)\|_{L^p(U_{4r}(y_0))}
\leq
Cr^{-1}.
\]
Thus, in all dimensions $n\geq2$,
\[
\begin{aligned}
&r^{1-\f{n}{p}}
\|\nabla_y^2G_{0,\lambda}(x_0,\cdot)\|_{L^p(U_{4r}(y_0))}+
(1+|\lambda|r^2)
\|\nabla_yG_{0,\lambda}(x_0,\cdot)\|_{L^\infty(U_{4r}(y_0))}
\leq
Cr^{1-n}.
\end{aligned}
\]
Combining this estimate with \eqref{dualityuse} and \eqref{GvaldaG0lda}, we obtain
\[
|G_{\va,\lambda}(x_0,y_0)-G_{0,\lambda}(x_0,y_0)|
\leq
\f{C(Ck)^{2k}\va}
{(1+|\lambda|r^2)^k r^{n-1}}.
\]
Since $r=\f{|x_0-y_0|}{16}$, the last estimate implies
\[
|G_{\va,\lambda}(x_0,y_0)-G_{0,\lambda}(x_0,y_0)|
\leq
\f{C(Ck)^{2k}\va}
{(1+|\lambda||x_0-y_0|^2)^k |x_0-y_0|^{n-1}},
\]
after increasing $C>0$. This completes the proof.
\end{proof}

We finish with the first-order convergence estimate. The proof applies the localized first-order approximation lemma to the two Green's functions away from the pole.

\begin{proof}[Proof of Proposition \ref{Green-convergence-II}]
Fix $x_0,y_0\in U$ with $x_0\neq y_0$, and set $ r:=\f{|x_0-y_0|}{16} $.
We may assume that $0<\va<r$. Indeed, if $\va\geq r$, then the desired estimate follows directly from the size estimates \eqref{Lipschitz1}, the uniform bound
\be
\|\nabla\Phi_{\va,j}^{\beta}\|_{L^\infty(U)}\leq C\label{Phiuniform}
\ee
from \cite[Lemma 2.2]{KLS13}, and the fact that
$\va |x_0-y_0|^{-n}\geq c|x_0-y_0|^{1-n}$ in this case.

Fix $\ga\in\Z\cap[1,m]$ and define
\[
u_{\va,\lambda}(x)
=
(G_{\va,\lambda}^{\al\gamma}(x,y_0))_{1\leq\al\leq m},
\quad
u_{0,\lambda}(x)
=
(G_{0,\lambda}^{\al\gamma}(x,y_0))_{1\leq\al\leq m}.
\]
Since $y_0\notin U_{4r}(x_0)$, we have
\[
(\cL_\va-\lambda I)u_{\va,\lambda}
=
(\cL_0-\lambda I)u_{0,\lambda}
=
0
\quad\text{in }U_{4r}(x_0),
\]
and both functions vanish on $T_{4r}^{U}(x_0)$ if
$T_{4r}^{U}(x_0)\neq\emptyset$.

By Proposition \ref{Green-convergence-I}, applied with the fixed value $k=1$, we have
\[
\|u_{\va,\lambda}-u_{0,\lambda}\|_{L^\infty(U_{4r}(x_0))}
\leq
C\va r^{1-n}.
\]
Moreover, since $|x-y_0|\sim r$ for $x\in U_{4r}(x_0)$, it follows from \eqref{Lipschitz1} and Lemma \ref{L0-second} that
\[
\|\nabla u_{0,\lambda}\|_{L^\infty(U_{4r}(x_0))}
\leq
\f{C}{(1+|\lambda|r^2)r^{n-1}},
\]
and
\[
\|\nabla^2u_{0,\lambda}\|_{L^\infty(U_{4r}(x_0))}
+
r^\rho[\nabla^2u_{0,\lambda}]_{C^{0,\rho}(U_{4r}(x_0))}
\leq
Cr^{-n}.
\]
Applying Lemma \ref{local-gradient-approx-lambda} to $u_{\va,\lambda}$ and $u_{0,\lambda}$, and using the preceding estimates, gives, for any $ k\in\Z_+$,
\[
\begin{aligned}
&\left|
\f{\partial}{\partial x_i}
G_{\va,\lambda}^{\al\gamma}(x_0,y_0)
-
\f{\partial\Phi_{\va,j}^{\al\beta}}{\partial x_i}(x_0)
\f{\partial}{\partial x_j}
G_{0,\lambda}^{\beta\gamma}(x_0,y_0)
\right|\leq
\f{
C(Ck)^{2k}\va\log(\va^{-1}r+2)
}
{(1+|\lambda|r^2)^k r^n}.
\end{aligned}
\]
Since $r=\f{|x_0-y_0|}{16}$, after increasing $C$ and absorbing the fixed powers of $16$ into $C(Ck)^{2k}$, we obtain \eqref{Green-convergence-IIineq}. This completes the proof.
\end{proof}

\section{Convergence of Parabolic Green's Functions}

In this section, we prove Theorems \ref{intro-thm-parabolic-zero} and \ref{intro-thm-parabolic-first}. Throughout the section, the coefficient matrix $A$ satisfies \eqref{intro-symmetry}--\eqref{intro-holder}. We consider the parabolic operator
\[
\pa_t+\cL_\va
\quad\text{in }\om\times\R,
\]
where $\om\subset\R^n$ is a bounded $C^{1,1}$ domain. For $\va\in(0,1]$,
\[
\cL_\va=-\op{div}\left[A\left(\f{x}{\va}\right)\nabla\right],
\]
while $\cL_0=-\op{div}(\widehat A\nabla)$ denotes the homogenized operator. We keep the notation $G_{\va,\lambda}(x,y)$ for the elliptic resolvent Green's function of $\cL_\va-\lambda I$ constructed in Section~\ref{Green-functions-section}, where
\[
\lambda\in\Sigma_{\theta_0}
=
\left\{
\lambda\in\C\backslash\{0\}:\ |\arg\lambda|>\theta_0
\right\},
\quad
\theta_0\in\(0,\f{\pi}{2}\).
\]

We first recall the Dirichlet parabolic Green's function and fix the kernel conventions used in the inverse Laplace representation.

\begin{prop}[Parabolic Green's function]\label{defn-parabolic-Green}
Let $\va\in[0,1]$. There exists a matrix-valued function
\[
G_\va(x,t;y,s)
=
(G_\va^{\al\beta}(x,t;y,s))_{\al,\beta\in\Z\cap[1,m]},
\quad
x,y\in\om,\quad t,s\in\R,
\]
called the Dirichlet parabolic Green's function of $\pa_t+\cL_\va$ in $\om\times\R$, so that the following properties hold.
\begin{enumerate}[label=$(\theenumi)$]
\item $G_\va$ is continuous in $(\om\times\R)\times(\om\times\R)\backslash\{(x,t)=(y,s)\}$.

\item If $t\leq s$, then $G_\va(x,t;y,s)=0$. Moreover, $G_\va(\cdot,\cdot;y,s)=0$ on $\pa\om\times\R$.

\item For each fixed pole $(y,s)$,
\[
(\pa_t+\cL_\va)G_\va(\cdot,\cdot;y,s)
=
\delta_{(y,s)}(\cdot,\cdot)\op{Id}_m
\quad\text{in }\om\times\R
\]
in the sense of distributions. Equivalently, for any $F\in C_0^\infty(\om\times\R;\C^m)$,
\[
u_\va(x,t)
=
\int_{\om\times\R}G_\va(x,t;y,s)F(y,s)\ud y\ud s
\]
is a weak solution of
\[
(\pa_t+\cL_\va)u_\va=F
\quad\text{in }\om\times\R,
\quad
u_\va=0
\quad\text{on }\pa\om\times\R.
\]

\item In the sense of distributions in the spatial variable,
\[
\lim_{t\to s^+}G_\va(x,t;\cdot,s)=\delta_x(\cdot)\op{Id}_m.
\]
Equivalently, for any $g\in C_0^\infty(\om;\C^m)$, the function
\[
u_\va^\al(x,t)
=
\int_\om G_\va^{\al\beta}(x,t;y,s)g^\beta(y)\ud y,
\quad t>s,
\]
is a weak solution of
\[
\left\{
\begin{aligned}
(\pa_t+\cL_\va)u_\va&=0
&&\text{in }\om\times(s,+\infty),
\\
u_\va&=0
&&\text{on }\pa \om\times(s,+\infty),
\\
u_\va(\cdot,s)&=g
&&\text{in }\om.
\end{aligned}
\right.
\]
\end{enumerate}
\end{prop}

The existence and the above properties follow from the construction of parabolic Green's matrices in \cite[Section 7]{GS15} and \cite[Theorem 2.7]{CDK08}. We next relate the parabolic Green's function to the elliptic resolvent Green's function.

For $\gamma>0$, set
\be
\mathcal C_\gamma=\{\gamma+\op{i}\eta:\eta\in\R\},
\label{contour}
\ee
oriented upward. This contour is used in the inverse Laplace representation below.

\begin{lem}[Inverse Laplace representation]\label{lem-parabolic-inversion}
For $t>s$,
\begin{equation}\label{parabolic-inversion-zeta}
G_\va(x,t;y,s)
=
\f{1}{2\pi\op{i}}
\int_{\mathcal C_\gamma}
\exp\{(t-s)\zeta\}
G_{\va,-\zeta}(x,y)\ud\zeta.
\end{equation}
The identity is first understood as the inverse Laplace formula for the semigroup generated by $-\cL_\va$, and then as an identity of distribution kernels.
\end{lem}

\begin{proof}
We give the details because the distinction between the operator identity and the kernel identity will be used later. Let $A_\va(x)$ be as in \eqref{Avax}. In $H_0^1(\om;\C^m)$ define
\[
\cB_{\va,0,\om}[u,v]
=
\int_\om \left\langle A_\va\nabla u,\ol{\nabla v}\right\rangle\ud x,
\quad
u,v\in H_0^1(\om;\C^m).
\]
By \eqref{intro-ellipticity},
\[
\operatorname{Re}\cB_{\va,0,\om}[u,u]
=
\cB_{\va,0,\om}[u,u]
\geq
\mu\int_\om |\nabla u|^2\ud x,
\quad
u\in H_0^1(\om;\C^m),
\]
and
\[
|\cB_{\va,0,\om}[u,v]|
\leq
C\|\nabla u\|_{L^2(\om)}
\|\nabla v\|_{L^2(\om)}.
\]
Moreover, \eqref{intro-symmetry} gives
\[
\cB_{\va,0,\om}[u,v]
=
\overline{\cB_{\va,0,\om}[v,u]},
\quad
u,v\in H_0^1(\om;\C^m).
\]
Thus $\cB_{\va,0,\om}$ is a densely defined, closed, symmetric, and nonnegative form on $L^2(\om;\C^m)$.

We define the Dirichlet realization of $\cL_\va$ as the operator $\wt{\cL}_\va$ associated with this form. More explicitly,
\[
D(\wt{\cL}_\va)
=
\left\{
u\in H_0^1(\om;\C^m):
\exists f\in L^2(\om;\C^m)
\text{ such that }
\cB_{\va,0,\om}[u,v]
=
\int_\om f^\al\overline{v^\al}
\text{ for any }v\in H_0^1(\om;\C^m)
\right\},
\]
and for such $u$ we set
\[
\wt{\cL}_\va u=f.
\]
When no confusion can arise, we write $\cL_\va$ for $\wt{\cL}_\va$. The preceding argument shows that this Dirichlet realization is a nonnegative self-adjoint operator in $L^2(\om;\C^m)$. Hence $-\cL_\va$ generates the strongly continuous semigroup
\[
S_\va(\tau)=\exp\{-\tau\cL_\va\},
\quad
\tau>0.
\]

We now fix the kernel terminology. If $T$ is an operator acting on test functions in $\om$, we say that $K_T(x,y)$ is its distribution kernel if, for any $F,H\in C_0^\infty(\om;\C^m)$,
\[
\langle TF,H\rangle_{L^2(\om)}
=
\int_{\om}\int_{\om}
K_T^{\al\beta}(x,y)F^\beta(y)\overline{H^\al(x)}
\ud y\ud x,
\]
whenever the right-hand side is represented by a locally integrable function; otherwise, the same identity is understood in the sense of distributions on $\om\times\om$. If there is no ambiguity, we simply call $K_T$ the kernel of $T$. By the definition of the parabolic Green's function, $G_\va(x,t;y,s)$ is the distribution kernel of $S_\va(t-s)$:
\[
S_\va(t-s)F(x)
=
\int_{\om}G_\va(x,t;y,s)F(y)\ud y,
\quad t>s.
\]
For $\operatorname{Re}\zeta>0$, set
\[
R_\va(\zeta):=(\cL_\va+\zeta I)^{-1}.
\]
The distribution kernel of $R_\va(\zeta)$ is $G_{\va,-\zeta}(x,y)$, because $G_{\va,-\zeta}$ is the elliptic Green's function of
\[
\cL_\va+\zeta I
=
\cL_\va-(-\zeta)I.
\]

The inverse Laplace formula for the semigroup gives, in the strong operator sense on $L^2(\om;\C^m)$,
\[
S_\va(t-s)
=
\f{1}{2\pi\op{i}}
\int_{\mathcal C_\gamma}
\exp\{(t-s)\zeta\}
R_\va(\zeta)\ud\zeta,
\quad t>s,\quad \gamma>0.
\]
Equivalently,
\[
\exp\{-(t-s)\cL_\va\}
=
\f{1}{2\pi\op{i}}
\int_{\mathcal C_\gamma}
\exp\{(t-s)\zeta\}
(\cL_\va+\zeta I)^{-1}\ud\zeta.
\]
Testing this operator identity against $F,H\in C_0^\infty(\om;\C^m)$ gives
\[
\begin{aligned}
&
\int_{\om}\int_{\om}
G_\va^{\al\beta}(x,t;y,s)
F^\beta(y)\overline{H^\al(x)}
\ud y\ud x
\\
&\quad =
\f{1}{2\pi\op{i}}
\int_{\mathcal C_\gamma}
\exp\{(t-s)\zeta\}
\left[
\int_{\om}\int_{\om}
G_{\va,-\zeta}^{\al\beta}(x,y)
F^\beta(y)\overline{H^\al(x)}
\ud y\ud x
\right]
\ud\zeta .
\end{aligned}
\]
By the uniqueness of distribution kernels, we obtain \eqref{parabolic-inversion-zeta}.
\end{proof}

Since $\operatorname{Re}\zeta=\gamma>0$ on $\mathcal C_\gamma$, we have $-\zeta\in\Sigma_{\theta_0}$ for any fixed $\theta_0\in(0,\f{\pi}{2})$. Thus the elliptic resolvent estimates established in Section~\ref{ConGreenResolvent} apply to $G_{\va,-\zeta}$. We also fix the square-root branch and the spatial kernel-composition notation used below. Let
\[
\C_+=\{\zeta\in\C:\operatorname{Re}\zeta>0\}.
\]
For $\zeta\in\C_+$, we use the principal branch of the square root:
\[
\sqrt{\zeta}
=
|\zeta|^{\f12}
\exp\left(\f{\op{i}}{2}\arg\zeta\right),
\quad\arg\zeta\in\(-\f{\pi}{2},\f{\pi}{2}\).
\]
Thus $\op{Re}(\sqrt{\zeta})>0$. In particular, if $\zeta=\gamma+\op{i}\eta\in\mathcal C_\gamma$, then
\be
\op{Re}(\sqrt{\zeta})
=
\left(\f{|\zeta|+\gamma}{2}\right)^{\f12},
\quad
\f{1}{\sqrt2}|\zeta|^{\f12}
\leq
\op{Re}(\sqrt{\zeta})
\leq
|\zeta|^{\f12}.
\label{zeta12}
\ee

If $K$ and $H$ are $\C^{m\times m}$-valued distribution kernels on $\om\times\om$, we denote by $K\star H$ their spatial kernel composition:
\be
(K\star H)^{\al\gamma}(x,y)
:=
\int_\om K^{\al\beta}(x,z)H^{\beta\gamma}(z,y)\ud z,
\label{stardefinition}
\ee
whenever the integral is absolutely convergent. Otherwise, the identity is understood after testing against smooth compactly supported functions. We set
\[
K^{\star q}
:=
\underbrace{K\star\cdots\star K}_{q\text{ factors}},
\quad q\in\Z_+,
\]
and
\be
K^{\star0}:=\mathsf{Id},
\label{K0star}
\ee
where $\mathsf{Id}$ is the identity distribution kernel
\[
\mathsf{Id}^{\al\beta}(x,y)
=
\delta_{\al\beta}\delta_y(x).
\]
Here $\delta_{\al\beta}$ is the Kronecker symbol and $\delta_y$ denotes the Dirac measure at $y$. Equivalently,
\[
\mathsf{Id}\star K=K\star\mathsf{Id}=K
\]
in the sense of distribution kernels.

We shall use the following Bessel-type kernels. Let $a>0$. For $\rho>0$ and $\sigma>0$, define
\be
\mathcal K_{\sigma,\rho}^{(a)}(r)
=
\begin{cases}
r^{\sigma-n}\exp\{-a\rho r\},
&\sg\in(0,n),\\[4pt]
\left(1+\left|\log(\rho r)\right|\right)\exp\{-a\rho r\},
&\sigma=n,\\[4pt]
\rho^{n-\sigma}\exp\{-a\rho r\},
&\sg\in(n,+\ift).
\end{cases}
\label{Bessel-kernel-definition}
\ee

\begin{lem}\label{lem-Bessel-kernel-convolution}
Let $\sigma,\sigma'>0$ and $a>0$. Then, for any $\rho>0$ and any $x,y\in\om$,
\begin{equation}\label{Bessel-kernel-convolution}
\int_\om
\mathcal K_{\sigma,\rho}^{(a)}(|x-z|)
\mathcal K_{\sigma',\rho}^{(a)}(|z-y|)
\ud z
\leq
C
\mathcal K_{\sigma+\sigma',\rho}^{(\f{a}{2})}(|x-y|),
\end{equation}
where $C$ depends only on $n$, $\sigma$, $\sigma'$, and $a$.
\end{lem}

The proof of Lemma \ref{lem-Bessel-kernel-convolution} is a direct calculation with the three regimes in \eqref{Bessel-kernel-definition}; it is postponed to Appendix \ref{ProofofTechnicalLemma}.

\begin{lem}[A choice of $k$ giving exponential decay]\label{lem-special-k}
Let $C_0>1$. There exist constants $C,c>0$, depending only on $C_0$, such that for any $\rho\geq0$,
\[
\inf_{k\in\Z_+}
\f{(C_0k)^{2k}}{(1+\rho^2)^k}
\leq
C\exp\{-c\rho\}.
\]
\end{lem}

\begin{proof}
Let $a>0$ be chosen later. If $0\leq\rho\leq \max\{2,a^{-1}\}$, we take $k=1$, and the estimate follows after increasing $C$. Assume now that $\rho>\max\{2,a^{-1}\}$. Choose
\[
k=\lfloor a\rho\rfloor+1,
\quad
\lfloor \al\rfloor:=\max\{j\in\Z:j\leq\al\}.
\]
Then $k\leq2a\rho$. Hence
\[
\f{(C_0k)^{2k}}{(1+\rho^2)^k}
\leq
\left(\f{C_0k}{\rho}\right)^{2k}
\leq
(2C_0a)^{2k}.
\]
Choosing $a>0$ so small that $2C_0a<1$, we get
\[
(2C_0a)^{2k}
\leq
C\exp\{-c\rho\}.
\]
This proves the lemma.
\end{proof}

\begin{cor}\label{corGvaGo}
Let $\theta_0\in(0,\f{\pi}{2})$ and $\lambda\in\Sigma_{\theta_0}$. Then, for any $x,y\in\om$ with $x\neq y$ and $r=|x-y|$,
\begin{equation}\label{resolvent-zero-exp}
\left|G_{\va,\lambda}(x,y)-G_{0,\lambda}(x,y)\right|
\leq
\f{C\va}{r^{n-1}}
\exp\{-c|\lambda|^{\f12}r\}.
\end{equation}
If, in addition, $\om$ is of class $C^{2,1}$, then
\begin{equation}\label{resolvent-first-exp}
\left|
\f{\pa}{\pa x_i}G_{\va,\lambda}^{\al\gamma}(x,y)
-
\f{\pa\Phi_{\va,j}^{\al\beta}}{\pa x_i}(x)
\f{\pa}{\pa x_j}G_{0,\lambda}^{\beta\gamma}(x,y)
\right|
\leq
\f{C\va\log(\va^{-1}r+2)}{r^n}
\exp\{-c|\lambda|^{\f12}r\},
\end{equation}
where $\Phi_{\va,j}^{\beta}$ is the Dirichlet corrector defined by \eqref{intro-dirichlet-corrector}. In the first estimate $C,c>0$ depend only on the data in Proposition \ref{Green-convergence-I}; in the second estimate they depend only on the data in Proposition \ref{Green-convergence-II}.
\end{cor}

\begin{proof}
Taking $\rho=|\lambda|^{\f12}|x-y|$ in Proposition~\ref{Green-convergence-I} and using Lemma \ref{lem-special-k} gives \eqref{resolvent-zero-exp}. The same choice of $k$ in Proposition~\ref{Green-convergence-II} gives \eqref{resolvent-first-exp}.
\end{proof}

For $n\geq3$, Proposition \ref{Green-ngeq3} and Lemma \ref{lem-special-k} give the Bessel-type bound needed below. In dimension two, however, the logarithmic estimate in Proposition \ref{Green-n2} is not sufficient by itself, because it contains $\log(R_\om|x-y|^{-1})$ rather than $\log(|\zeta|^{\f{1}{2}}|x-y|)$. We therefore record the following massive resolvent estimate in dimension two.

\begin{lem}[Two-dimensional massive resolvent estimate]\label{lem-2d-massive-resolvent}
Let $n=2$ and let $\zeta\in\C_+$. Then, for any $\va\in[0,1]$ and any $x,y\in\om$ with $x\neq y$,
\[
|G_{\va,-\zeta}(x,y)|
\leq
C(1+|\log(|\zeta|^{\f{1}{2}}|x-y|)|)
\exp\{-c|\zeta|^{\f{1}{2}}|x-y|\}.
\]
The constants $C,c>0$ depend only on $\mu$, $m$, $\theta_0$, $\tau$, $\nu$, and $\om$.
\end{lem}

\begin{proof}
Let $r=|x-y|$ and set $\rho=|\zeta|^{\f12}$. By \eqref{Lipschitz1} and Lemma \ref{lem-special-k}, applied with $\lambda=-\zeta$, we have
\[
|\nabla_zG_{\va,-\zeta}(z,y)|
\leq
\f{C}{|z-y|}
\exp\{-c\rho|z-y|\},
\quad z\in\om,\ z\neq y.
\]
Indeed, in dimension $n=2$, \eqref{Lipschitz1} gives
\[
|\nabla_zG_{\va,-\zeta}(z,y)|
\leq
\f{C(Ck)^{2k}}
{(1+|\zeta||z-y|^2)^k |z-y|},
\]
and the choice of $k$ in Lemma \ref{lem-special-k} gives the desired exponential factor.

We now integrate this massive gradient estimate along the same boundary-chain argument used in the proof of the two-dimensional logarithmic estimate for Green's functions; see \cite[Remark 4.12]{Wan23}. More precisely, since $\om$ is a bounded $C^{1,1}$ domain, there is a rectifiable curve $\Gamma\subset\om$ joining $x$ to $\partial\om$ such that the portion of $\Gamma$ contained in each annulus
\[
A_j=
\left\{
z\in\om:\ 2^jr\leq |z-y|<2^{j+1}r
\right\},
\quad j\geq0,
\]
has length at most $C2^jr$. Since $G_{\va,-\zeta}(\cdot,y)=0$ on $\partial\om$ in the trace sense, we obtain
\[
\begin{aligned}
|G_{\va,-\zeta}(x,y)|
&\leq
\int_\Gamma |\nabla_zG_{\va,-\zeta}(z,y)|\ud s(z)
\\
&\leq
C\sum_{j\geq0}
\f{\exp\{-c|\zeta|^{\f{1}{2}} 2^jr\}}{2^jr}
\mathcal H^1(\Gamma\cap A_j)
\\
&\leq
C\sum_{j\geq0}\exp\{-c|\zeta|^{\f{1}{2}} 2^jr\}.
\end{aligned}
\]
It remains to estimate the dyadic sum. Set $u=|\zeta|^{\f{1}{2}} r$. If $0<u\leq1$, choose $J\in\Z_+$ such that
\[
2^Ju\leq1<2^{J+1}u.
\]
Then $J\leq C(1+|\log u|)$, and
\[
\sum_{0\leq j\leq J}\exp\{-c2^ju\}
\leq J+1
\leq C(1+|\log u|).
\]
For the tail, write $j=J+\ell$ with $\ell\geq1$. Since $2^Ju>\f12$, we have
\[
2^ju=2^\ell(2^Ju)\geq 2^{\ell-1},
\]
and therefore
\[
\sum_{j>J}\exp\{-c2^ju\}
\leq
\sum_{\ell\geq1}\exp\{-c2^{\ell-1}\}
\leq C.
\]
Thus, for $0<u\leq1$,
\[
\sum_{j\geq0}\exp\{-c2^ju\}
\leq
C(1+|\log u|).
\]
Since $\exp\{-au\}\geq c$ for $0<u\leq1$, this implies
\[
\sum_{j\geq0}\exp\{-c2^ju\}
\leq
C(1+|\log u|)\exp\{-au\}.
\]
If $u>1$, then
\[
\sum_{j\geq0}\exp\{-c2^ju\}
\leq
C\exp\{-cu\}
\leq
C(1+\log u)\exp\{-au\}
\]
after choosing $0<a<c$. Combining the two cases and taking $u=\rho r$ proves the lemma.
\end{proof}

We next present a resolvent derivative identity. It is an operator identity written in the language of distribution kernels and is the basic tool for obtaining enough decay in the inverse Laplace integral.

\begin{lem}\label{lem-resolvent-derivative-formula}
Let $\zeta\in\C_+$ and set
\[
D_{\va,\zeta}(x,y)
:=
G_{\va,-\zeta}(x,y)-G_{0,-\zeta}(x,y).
\]
Then, for any $M\in\Z_+$,
\begin{equation}\label{Dzeta-derivative-formula}
\pa_\zeta^M D_{\va,\zeta}
=
(-1)^M M!
\left(
\sum_{\ell=0}^M
G_{\va,-\zeta}^{\star\ell}
\star
D_{\va,\zeta}
\star
G_{0,-\zeta}^{\star(M-\ell)}
\right).
\end{equation}
The equality is first understood in the sense of distribution kernels. For $x\ne y$, the right-hand side is represented by the corresponding iterated spatial integral.
\end{lem}

\begin{proof}
Set
\[
R_\va(\zeta)=(\cL_\va+\zeta I)^{-1},
\quad
R_0(\zeta)=(\cL_0+\zeta I)^{-1}.
\]
Since $\operatorname{Re}\zeta>0$, these resolvents are holomorphic in $\zeta$ and satisfy
\[
\pa_\zeta R_\va(\zeta)
=
-R_\va(\zeta)^2,
\quad
\pa_\zeta R_0(\zeta)
=
-R_0(\zeta)^2.
\]
By induction,
\[
\pa_\zeta^M R_\va(\zeta)
=
(-1)^M M!R_\va(\zeta)^{M+1},
\quad
\pa_\zeta^M R_0(\zeta)
=
(-1)^M M!R_0(\zeta)^{M+1}.
\]
Therefore
\[
\pa_\zeta^M(R_\va(\zeta)-R_0(\zeta))
=
(-1)^M M!
\left(
R_\va(\zeta)^{M+1}-R_0(\zeta)^{M+1}
\right).
\]
For two possibly non-commuting operators $A$ and $B$, the identity
\[
A^{M+1}-B^{M+1}
=
\sum_{\ell=0}^M A^\ell(A-B)B^{M-\ell}
\]
holds. Applying this to $A=R_\va(\zeta)$ and $B=R_0(\zeta)$ gives the corresponding operator identity.

It remains to translate the operator identity into a kernel identity. Let $F,H\in C_0^\infty(\om;\C^m)$. The kernel of $R_\va(\zeta)$ is $G_{\va,-\zeta}$, the kernel of $R_0(\zeta)$ is $G_{0,-\zeta}$, and the kernel of $R_\va(\zeta)-R_0(\zeta)$ is $D_{\va,\zeta}$. Thus
\[
\begin{aligned}
&
\langle R_\va(\zeta)^\ell(R_\va(\zeta)-R_0(\zeta))
R_0(\zeta)^{M-\ell}F,H\rangle
\\
&\quad =
\int_{\om}\int_{\om}
(G_{\va,-\zeta}^{\star\ell}\star D_{\va,\zeta}\star G_{0,-\zeta}^{\star(M-\ell)})^{\al\gamma}(x,y)
F^\gamma(y)\overline{H^\al(x)}
\ud y\ud x.
\end{aligned}
\]
For instance, if $0<\ell<M$, then this composed kernel is
\[
\begin{aligned}
&
(G_{\va,-\zeta}^{\star\ell}\star D_{\va,\zeta}\star G_{0,-\zeta}^{\star(M-\ell)})(x,y)
\\
&\quad =
\int_{\om^M}
G_{\va,-\zeta}(z_0,z_1)
\cdots
G_{\va,-\zeta}(z_{\ell-1},z_\ell)
D_{\va,\zeta}(z_\ell,z_{\ell+1})
\\
&\quad\quad
\times
G_{0,-\zeta}(z_{\ell+1},z_{\ell+2})
\cdots
G_{0,-\zeta}(z_M,z_{M+1})
\ud z_1\cdots\ud z_M,
\end{aligned}
\]
where $z_0=x$, $z_{M+1}=y$, and matrix multiplication over adjacent indices is understood. The endpoint cases $\ell=0$ and $\ell=M$ are covered by the convention $K^{\star0}=\mathsf{Id}$. Since the preceding identity holds for all test functions $F$ and $H$, it gives \eqref{Dzeta-derivative-formula} as a distribution kernel identity. Away from the diagonal $x=y$, the estimates in Proposition \ref{Green-ngeq3} and Lemma \ref{lem-2d-massive-resolvent}, together with Lemma \ref{lem-Bessel-kernel-convolution}, show that the spatial compositions on the right-hand side are represented by locally absolutely convergent iterated integrals. Hence the same identity holds pointwise away from the diagonal.
\end{proof}

\begin{lem}\label{lem-resolvent-derivative-exp}
Let $D_{\va,\zeta}$ be as in Lemma~\ref{lem-resolvent-derivative-formula}. Assume that $\zeta\in\C_+$. If $2M>n+1$, then
\begin{equation}\label{Dzeta-derivative-exp}
|\pa_\zeta^M D_{\va,\zeta}(x,y)|
\leq
C_M\va
|\zeta|^{\f{n-1-2M}{2}}
\exp\{-c_M\op{Re}(\sqrt{\zeta})|x-y|\}.
\end{equation}
Here $\sqrt{\zeta}$ is the principal square root fixed in \eqref{zeta12}, and $C_M,c_M>0$ depend only on $n$, $m$, $\mu$, $\theta_0$, $\tau$, $\nu$, $M$, and $\om$.
\end{lem}

\begin{proof}
Put $\rho=\op{Re}(\sqrt{\zeta})$. Since $\zeta\in\C_+$, we have $-\zeta\in\Sigma_{\theta_0}$, and \eqref{zeta12} gives $\rho\sim|\zeta|^{\f12}$. By Proposition~\ref{Green-ngeq3}, Lemma \ref{lem-2d-massive-resolvent}, and Lemma~\ref{lem-special-k}, applied with $\lambda=-\zeta$, there exist constants $C>0$ and $a_0>0$ such that, for any $\va\in[0,1]$,
\be
|G_{\va,-\zeta}(x,y)|
\leq
C\mathcal K_{2,\rho}^{(a_0)}(|x-y|).
\label{Gbesel}
\ee
Moreover, by \eqref{resolvent-zero-exp}, after decreasing $a_0$ if necessary, we obtain
\be
|D_{\va,\zeta}(x,y)|
\leq
C\va\mathcal K_{1,\rho}^{(a_0)}(|x-y|).
\label{Gbesel1}
\ee

Using Lemma~\ref{lem-resolvent-derivative-formula}, we have
\[
\pa_\zeta^M D_{\va,\zeta}
=
(-1)^M M!
\left(
\sum_{\ell=0}^M
G_{\va,-\zeta}^{\star\ell}
\star
D_{\va,\zeta}
\star
G_{0,-\zeta}^{\star(M-\ell)}
\right).
\]
Fix $\ell\in\Z\cap[0,M]$. The corresponding term contains $\ell$ factors of $G_{\va,-\zeta}$, one factor of $D_{\va,\zeta}$, and $M-\ell$ factors of $G_{0,-\zeta}$. Applying Lemma~\ref{lem-Bessel-kernel-convolution} successively and using \eqref{Gbesel}--\eqref{Gbesel1}, we obtain
\[
\begin{aligned}
|
G_{\va,-\zeta}^{\star\ell}
\star
D_{\va,\zeta}
\star
G_{0,-\zeta}^{\star(M-\ell)}(x,y)|
&\leq
C_M\va
\mathcal K_{2\ell+1+2(M-\ell),\rho}^{(\f{a_0}{2^M})}(|x-y|)\\
&=
C_M\va
\mathcal K_{1+2M,\rho}^{(\f{a_0}{2^M})}(|x-y|).
\end{aligned}
\]
Set $a_M=2^{-M}a_0$. Since $2M>n+1$, we have $1+2M>n$. Hence, by the definition of $\mathcal K_{\sigma,\rho}^{(a_M)}$ in the case $\sg\in(n,+\ift)$,
\[
\mathcal K_{1+2M,\rho}^{(a_M)}(|x-y|)
=
\rho^{n-1-2M}
\exp\left\{-a_M\rho|x-y|\right\}.
\]
Summing over $\ell$ gives
\[
\left|
\pa_\zeta^M D_{\va,\zeta}(x,y)
\right|
\leq
C_M\va
\rho^{n-1-2M}
\exp\left\{-a_M\rho |x-y|\right\}.
\]
Since $\rho\sim|\zeta|^{\f12}$ on $\C_+$ and $n-1-2M<0$,
\[
\rho^{n-1-2M}
\leq
C|\zeta|^{\f{n-1-2M}{2}}.
\]
Renaming $a_M$ as $c_M$ proves \eqref{Dzeta-derivative-exp}.
\end{proof}

\subsection{Proof of Theorem \ref{intro-thm-parabolic-zero}}

Set
\[
\tau=t-s>0,
\quad
r=|x-y|.
\]
It follows from \eqref{parabolic-inversion-zeta} that
\[
G_\va(x,t;y,s)-G_0(x,t;y,s)
=
\f{1}{2\pi\op{i}}
\int_{\mathcal C_\gamma}
\exp\{\tau\zeta\}
D_{\va,\zeta}(x,y)\ud\zeta,
\]
where $D_{\va,\zeta}=G_{\va,-\zeta}-G_{0,-\zeta}$ and $\gamma>0$. Choose an integer $M\in\Z_+$ such that $2M>n+1$. Integrating by parts $M$ times along the vertical contour $\mathcal C_\gamma$, and using \eqref{Dzeta-derivative-exp} to justify the vanishing of the boundary terms, gives
\[
G_\va(x,t;y,s)-G_0(x,t;y,s)
=
\f{(-1)^M}{2\pi\op{i}\tau^M}
\int_{\mathcal C_\gamma}
\exp\{\tau\zeta\}
\pa_\zeta^M D_{\va,\zeta}(x,y)\ud\zeta.
\]
Therefore,
\[
\begin{aligned}
&|G_\va(x,t;y,s)-G_0(x,t;y,s)|
\\
&\quad\leq
C\va \tau^{-M}\exp\{\tau\gamma\}
\int_{-\infty}^{+\infty}
|\gamma+\op{i}\eta|^{\f{n-1-2M}{2}}
\exp\{-cr\operatorname{Re}\sqrt{\gamma+\op{i}\eta}\}\ud\eta.
\end{aligned}
\]
We use the elementary contour estimate
\[
\int_{-\infty}^{+\infty}
|\gamma+\op{i}\eta|^{\f{n-1-2M}{2}}
\exp\{
-c r\operatorname{Re}\sqrt{\gamma+\op{i}\eta}\}
\ud\eta
\leq
C\gamma^{\f{n+1-2M}{2}}\exp\{-c r\sqrt\gamma\}.
\]
It follows that
\begin{equation}\label{pre-gaussian-zero}
\left|
G_\va(x,t;y,s)-G_0(x,t;y,s)
\right|
\leq
C\va
\tau^{-M}
\gamma^{\f{n+1-2M}{2}}
\exp\left\{
\tau\gamma-c r\sqrt\gamma
\right\}.
\end{equation}

If $r^2\leq\tau$, take $\gamma=\tau^{-1}$. Then
\[
\tau^{-M}\gamma^{\f{n+1-2M}{2}}
=
\tau^{-\f{n+1}{2}},
\]
and $\exp\{\tau\gamma-cr\sqrt\gamma\}\leq C$. Since $r^2\leq\tau$, the Gaussian factor $\exp\{-\f{\kappa r^2}{\tau}\}$ is bounded below by a positive constant. Hence
\be
|G_\va(x,t;y,s)-G_0(x,t;y,s)|
\leq
\f{C\va}{\tau^{\f{n+1}{2}}}
\exp\left\{-\f{\kappa r^2}{\tau}\right\}.
\label{Gvakappa1}
\ee

If $r^2>\tau$, take
\[
\gamma=\f{ar^2}{\tau^2},
\]
where $a>0$ is sufficiently small. Then
\[
\tau\gamma-c r\sqrt\gamma
=
\f{ar^2}{\tau}
-
\f{c\sqrt a\, r^2}{\tau}
\leq
-\f{\kappa r^2}{\tau}
\]
for some $\kappa>0$. Moreover,
\[
\tau^{-M}\gamma^{\f{n+1-2M}{2}}
=
\tau^{-\f{n+1}{2}}
\left(\f{r^2}{\tau}\right)^{\f{n+1-2M}{2}}
\leq
\tau^{-\f{n+1}{2}},
\]
because $2M>n+1$ and $\f{r^2}{\tau}>1$. Combining this with \eqref{pre-gaussian-zero} proves \eqref{intro-parabolic-zero}.

\subsection{Proof of Theorem \ref{intro-thm-parabolic-first}}

The first-order estimate requires differentiating the resolvent error with respect to the spectral parameter. Since a direct pointwise bound for the differentiated first-order error is not sufficient, we first smooth the error by composing it with powers of the distribution kernel of the homogenized resolvent.

\begin{lem}\label{lem-first-resolvent-derivative}
Let $\zeta\in\C_+$ and, for fixed $i\in\Z\cap[1,n]$, define
\[
E_{\va,\zeta}^{\al\gamma}(x,y)
=
\f{\pa}{\pa x_i}G_{\va,-\zeta}^{\al\gamma}(x,y)
-
\f{\pa\Phi_{\va,j}^{\al\beta}}{\pa x_i}(x)
\f{\pa}{\pa x_j}G_{0,-\zeta}^{\beta\gamma}(x,y).
\]
If $2M>n+2$, then
\begin{equation}\label{first-resolvent-derivative-exp}
|\pa_\zeta^M E_{\va,\zeta}(x,y)|
\leq
C_M\va
|\zeta|^{\f{n-2M}{2}}
\log\left(2+\f{1}{\va|\zeta|^{\f12}}\right)
\exp\{
-c_M\op{Re}(\sqrt{\zeta})|x-y|\}.
\end{equation}
Here $\sqrt{\zeta}$ is the principal square root fixed in \eqref{zeta12}, and $C_M,c_M>0$ depend only on $n$, $m$, $\mu$, $\theta_0$, $\tau$, $\nu$, $M$, and $\om$.
\end{lem}

\begin{proof}
Put $\rho=\op{Re}(\sqrt{\zeta})$. Since $\zeta\in\C_+$, we have $-\zeta\in\Sigma_{\theta_0}$, and \eqref{zeta12} gives $\rho\sim|\zeta|^{\f12}$. Set
\[
L_{\va,\zeta}
:=
\log\left(2+\f{1}{\va|\zeta|^{\f12}}\right).
\]
By Proposition~\ref{Green-ngeq3}, Proposition~\ref{Green-gradient}, Lemma \ref{lem-2d-massive-resolvent}, and Lemma~\ref{lem-special-k}, applied with $\lambda=-\zeta$, there exist constants $C>0$ and $a_0>0$ such that
\begin{equation}\label{first-proof-G-bound}
|G_{\va,-\zeta}(x,y)|+|G_{0,-\zeta}(x,y)|
\leq
C\mathcal K_{2,\rho}^{(a_0)}(|x-y|)
\end{equation}
and
\begin{equation}\label{first-proof-gradG-bound}
|\nabla_xG_{\va,-\zeta}(x,y)|+|\nabla_xG_{0,-\zeta}(x,y)|
\leq
C\mathcal K_{1,\rho}^{(a_0)}(|x-y|).
\end{equation}
Moreover, by \eqref{resolvent-zero-exp}, after decreasing $a_0$ if necessary,
\begin{equation}\label{first-proof-D-bound}
|D_{\va,\zeta}(x,y)|
\leq
C\va
\mathcal K_{1,\rho}^{(a_0)}(|x-y|).
\end{equation}
In the following proof, $a_0$ may be decreased a finite number of times, but it remains a positive constant depending only on the fixed data.

We first estimate the smoothed first-order error $E_{\va,\zeta}\star G_{0,-\zeta}^{\star M}$. Set
\[
H_M(x,y):=G_{0,-\zeta}^{\star M}(x,y),
\]
and define
\begin{align*}
U_{\va,M}(x,y)
&=
\int_\om G_{\va,-\zeta}(x,z)H_M(z,y)\ud z,
\\
U_{0,M}(x,y)
&=
\int_\om G_{0,-\zeta}(x,z)H_M(z,y)\ud z.
\end{align*}
Thus $U_{0,M}=G_{0,-\zeta}^{\star(M+1)}$. Moreover, for any $y\in\om$,
\be
(\cL_\va+\zeta I)U_{\va,M}(\cdot,y)
=
H_M(\cdot,y)
=
(\cL_0+\zeta I)U_{0,M}(\cdot,y)
\label{LvazetaI}
\ee
in $\om$, with zero Dirichlet boundary condition on $\pa\om$. By the definition of $E_{\va,\zeta}$,
\begin{equation}\label{EM-smoothed-identity}
\left(E_{\va,\zeta}\star G_{0,-\zeta}^{\star M}\right)^{\al\gamma}(x,y)
=
\f{\pa}{\pa x_i}U_{\va,M}^{\al\gamma}(x,y)
-
\f{\pa\Phi_{\va,j}^{\al\beta}}{\pa x_i}(x)
\f{\pa}{\pa x_j}U_{0,M}^{\beta\gamma}(x,y).
\end{equation}

By \eqref{first-proof-G-bound}, \eqref{first-proof-gradG-bound}, and Lemma~\ref{lem-Bessel-kernel-convolution}, successively applied, we have
\begin{align}
|H_M(x,y)|
&\leq
C_M\mathcal K_{2M,\rho}^{(\f{a_0}{2^M})}(|x-y|),
\label{HM-bound}
\\
|U_{0,M}(x,y)|
&\leq
C_M\mathcal K_{2M+2,\rho}^{(\f{a_0}{2^{M+1}})}(|x-y|),
\label{U0M-bound}
\\
|\nabla_xU_{0,M}(x,y)|
&\leq
C_M\mathcal K_{2M+1,\rho}^{(\f{a_0}{2^{M+1}})}(|x-y|).
\label{U0M-first-bound}
\end{align}

Let $r=|x-y|$ and set
\be
R=\f{1}{64}\left(r+\rho^{-1}\right).
\label{Rdef}
\ee
Then
\be
R\geq\f{r}{64},
\quad
R\geq\f{1}{64\rho},
\quad
R^{-1}\leq C\rho,
\quad
1+\rho^2R^2\leq C(1+\rho r)^2.
\label{Rdefproperty}
\ee
For $z\in\om_{8R}(x)$ and any $\ell>n$, we have
\begin{equation}\label{kernel-local-comparison}
\mathcal K_{\ell,\rho}^{(a)}(|z-y|)
\leq
C_\ell
\mathcal K_{\ell,\rho}^{(\f a2)}(r).
\end{equation}
Indeed, if $\rho r\leq2$, the exponential factors are comparable. If $\rho r>2$, then $\rho^{-1}<\f{r}{2}$ and
\[
|z-y|
\geq
r-8R
=
r-\f{r+\rho^{-1}}{8}
\geq
\f{13}{16}r,
\]
which gives \eqref{kernel-local-comparison} from the definition of $\mathcal K_{\ell,\rho}^{(a)}$ in the case $\ell>n$.

Since $\cL_0$ has constant coefficients, the local second-order estimate with right-hand side, together with \eqref{HM-bound}, \eqref{U0M-bound}, and \eqref{kernel-local-comparison}, gives
\begin{equation}\label{U0M-second-bound}
\begin{aligned}
\|\nabla_x^2U_{0,M}(\cdot,y)\|_{L^\infty(\om_{4R}(x))}
&\leq
C R^{-2}\|U_{0,M}(\cdot,y)\|_{L^\infty(\om_{8R}(x))}
\\
&\quad+
C\|H_M(\cdot,y)\|_{L^\infty(\om_{8R}(x))}
+
C|\zeta|\|U_{0,M}(\cdot,y)\|_{L^\infty(\om_{8R}(x))}
\\
&\leq
C_M(1+|\zeta|R^2)R^{-2}
\mathcal K_{2M+2,\rho}^{(\f{a_0}{2^{M+2}})}(r)
+
C_M\mathcal K_{2M,\rho}^{(\f{a_0}{2^{M+1}})}(r).
\end{aligned}
\end{equation}

Choose $\sigma\in(0,1)$ such that $2M>n+2+\sigma$. Using \eqref{HolderEstimate}, \eqref{LvazetaI}, \eqref{HM-bound}, \eqref{U0M-bound}, and \eqref{kernel-local-comparison}, we have
\begin{equation}\label{U0Mholder}
\begin{aligned}
&\left[U_{0,M}(\cdot,y)\right]_{C^{0,\sigma}(\om_{5R}(x))}
\\
&\quad\leq
C R^{-\sigma}
\left(
\|U_{0,M}(\cdot,y)\|_{L^\infty(\om_{6R}(x))}
+
R^2(1+|\zeta|R^2)^{N_0}
\|H_M(\cdot,y)\|_{L^\infty(\om_{6R}(x))}
\right)
\\
&\quad\leq
C_MR^{-\sigma}
\mathcal K_{2M+2,\rho}^{(\f{a_0}{2^{M+2}})}(r)
+
C_MR^{2-\sigma}(1+|\zeta|R^2)^{N_0}
\mathcal K_{2M,\rho}^{(\f{a_0}{2^{M+1}})}(r).
\end{aligned}
\end{equation}
Replacing $M$ by $M-1$ in the same argument gives
\begin{equation}\label{HMholder}
\begin{aligned}
\left[H_M(\cdot,y)\right]_{C^{0,\sigma}(\om_{5R}(x))}
&\leq
C_MR^{-\sigma}
\mathcal K_{2M,\rho}^{(\f{a_0}{2^{M+1}})}(r)+
C_MR^{2-\sigma}(1+|\zeta|R^2)^{N_0}
\mathcal K_{2M-2,\rho}^{(\f{a_0}{2^{M}})}(r).
\end{aligned}
\end{equation}
The $C^{2,\sigma}$ estimate for $\cL_0+\zeta I$, together with \eqref{U0Mholder} and \eqref{HMholder}, yields
\begin{equation}\label{U0M-holder-bound}
\begin{aligned}
R^\sigma[\nabla_x^2U_{0,M}(\cdot,y)]_{C^{0,\sigma}(\om_{4R}(x))}
&\leq
C_MR^{-2}(1+|\zeta|R^2)^{N_0}
\mathcal K_{2M+2,\rho}^{(\f{a_0}{2^{M+2}})}(r)+
C_M\mathcal K_{2M,\rho}^{(\f{a_0}{2^{M+1}})}(r).
\end{aligned}
\end{equation}

Furthermore,
\[
U_{\va,M}(x,y)-U_{0,M}(x,y)
=
D_{\va,\zeta}\star G_{0,-\zeta}^{\star M}(x,y).
\]
Using \eqref{first-proof-G-bound}, \eqref{first-proof-D-bound}, and Lemma~\ref{lem-Bessel-kernel-convolution} successively, we get
\[
|U_{\va,M}(z,y)-U_{0,M}(z,y)|
\leq
C_M\va
\mathcal K_{1+2M,\rho}^{(\f{a_0}{2^M})}(|z-y|).
\]
By \eqref{kernel-local-comparison}, this implies that
\begin{equation}\label{UM-difference-bound}
|U_{\va,M}(z,y)-U_{0,M}(z,y)|
\leq
C_M\va
\mathcal K_{1+2M,\rho}^{(c_M)}(r)
\end{equation}
for any $z\in\om_{4R}(x)$.

We now apply the local first-order estimate on $\om_{4R}(x)$. Assume first that $0<\va<R$. Applying Lemma~\ref{local-gradient-approx-lambda} to $U_{\va,M}(\cdot,y)$ and $U_{0,M}(\cdot,y)$ in $\om_{4R}(x)$, with the resolvent parameter $-\zeta$, gives
\begin{equation}\label{EM-local-first-step}
\begin{aligned}
&
\left|
\f{\pa}{\pa x_i}U_{\va,M}^{\al\gamma}(x,y)
-
\f{\pa\Phi_{\va,j}^{\al\beta}}{\pa x_i}(x)
\f{\pa}{\pa x_j}U_{0,M}^{\beta\gamma}(x,y)
\right|
\\
&\leq
\f{C_M}{R}
\fint_{\om_{4R}(x)}
|U_{\va,M}(z,y)-U_{0,M}(z,y)|\ud z
+
C_M\va(1+\rho^2R^2)R^{-1}
\|\nabla_xU_{0,M}(\cdot,y)\|_{L^\infty(\om_{4R}(x))}
\\
&\quad
+
C_M\va
\log\left(2+\f{R}{\va}\right)
\|\nabla_x^2U_{0,M}(\cdot,y)\|_{L^\infty(\om_{4R}(x))}
+
C_M\va R^\sigma
[\nabla_x^2U_{0,M}(\cdot,y)]_{C^{0,\sigma}(\om_{4R}(x))}.
\end{aligned}
\end{equation}

Using \eqref{UM-difference-bound}, \eqref{U0M-first-bound}, \eqref{U0M-second-bound}, and \eqref{U0M-holder-bound} in \eqref{EM-local-first-step}, we obtain
\[
\begin{aligned}
&
\left|
\f{\pa}{\pa x_i}U_{\va,M}^{\al\gamma}(x,y)
-
\f{\pa\Phi_{\va,j}^{\al\beta}}{\pa x_i}(x)
\f{\pa}{\pa x_j}U_{0,M}^{\beta\gamma}(x,y)
\right|
\\
&\leq
C_M\va
\Bigg\{
R^{-1}\mathcal K_{1+2M,\rho}^{(c_M)}(r)
+
(1+\rho^2R^2)R^{-1}
\mathcal K_{2M+1,\rho}^{(c_M)}(r)
\\
&\quad+
\log\left(2+\f{R}{\va}\right)
\left[
(1+\rho^2R^2)R^{-2}
\mathcal K_{2M+2,\rho}^{(c_M)}(r)
+
\mathcal K_{2M,\rho}^{(c_M)}(r)
\right]
\\
&\quad+
R^{-2}(1+|\zeta|R^2)^{N_0}
\mathcal K_{2M+2,\rho}^{(c_M)}(r)
+
\mathcal K_{2M,\rho}^{(c_M)}(r)
\Bigg\}.
\end{aligned}
\]
Since $2M>n+2$, all the Bessel kernels above are in the case $\ell>n$. By \eqref{Rdef} and \eqref{Rdefproperty}, we have
\[
R^{-1}\leq C\rho,
\quad
R^{-2}\leq C\rho^2,
\quad
1+\rho^2R^2\leq C(1+\rho r)^2,
\quad
|\zeta|\sim\rho^2.
\]
It follows that
\[
\begin{aligned}
&
\left|
\f{\pa}{\pa x_i}U_{\va,M}^{\al\gamma}(x,y)
-
\f{\pa\Phi_{\va,j}^{\al\beta}}{\pa x_i}(x)
\f{\pa}{\pa x_j}U_{0,M}^{\beta\gamma}(x,y)
\right|
\\
&\leq
C_M\va
\left[
1+\log\left(2+\f{R}{\va}\right)
\right]
(1+\rho r)^{2N_0+2}
\rho^{n-2M}
\exp\{-c_M\rho r\}.
\end{aligned}
\]
The polynomial factor can be absorbed into the exponential:
\[
(1+\rho r)^{2N_0+2}\exp\{-c_M\rho r\}
\leq
C_M\exp\left\{-\f{c_M}{2}\rho r\right\}.
\]
By \eqref{Rdef},
\[
\log\left(2+\f{R}{\va}\right)
\leq
C\log\left(2+\f{r}{\va}\right)
+
C\log\left(2+\f{1}{\va\rho}\right).
\]
The first logarithm is also controlled by the second after absorbing a further polynomial factor into the exponential. Indeed, setting $u=\rho r$ and $B=(\va\rho)^{-1}$,
\[
\log(2+Bu)\leq C(1+u)\log(2+B),
\]
and hence
\[
\log\left(2+\f{r}{\va}\right)
(1+\rho r)^{2N_0+2}
\exp\{-c_M\rho r\}
\leq
C_M
\log\left(2+\f{1}{\va\rho}\right)
\exp\left\{-\f{c_M}{2}\rho r\right\}.
\]
By \eqref{zeta12},
\[
\log\left(2+\f{1}{\va\rho}\right)
\leq
C\log\left(2+\f{1}{\va|\zeta|^{\f12}}\right)
=
C L_{\va,\zeta}.
\]
Combining this with \eqref{EM-smoothed-identity}, we obtain
\begin{equation}\label{EM-small-eps-case}
\left|
E_{\va,\zeta}\star G_{0,-\zeta}^{\star M}(x,y)
\right|
\leq
C_M\va L_{\va,\zeta}
\mathcal K_{2M,\rho}^{(c_M)}(r)
\end{equation}
when $\va\in(0,R)$.

It remains to consider the case $\va\geq R$. By \eqref{HM-bound}, \eqref{U0M-bound}, \eqref{U0M-first-bound}, and the uniform bound for the Dirichlet corrector gradient \eqref{Phiuniform}, we have
\[
|\nabla_xU_{\va,M}(x,y)|
+
\left|
\f{\pa\Phi_{\va,j}}{\pa x_i}(x)
\f{\pa U_{0,M}}{\pa x_j}(x,y)
\right|
\leq
C_M\mathcal K_{2M+1,\rho}^{(\f{a_0}{2^M})}(r).
\]
Since $\va\geq R\geq(64\rho)^{-1}$ and $L_{\va,\zeta}\geq\log2$,
\[
\mathcal K_{2M+1,\rho}^{(\f{a_0}{2^M})}(r)
=
\rho^{n-1-2M}\exp\left\{-\f{a_0}{2^M}\rho r\right\}
\leq
C_M\va L_{\va,\zeta}
\rho^{n-2M}\exp\left\{-\f{a_0}{2^M}\rho r\right\}.
\]
Hence,
\begin{equation}\label{EM-large-eps-case}
|E_{\va,\zeta}\star G_{0,-\zeta}^{\star M}(x,y)|
\leq
C_M\va L_{\va,\zeta}
\mathcal K_{2M,\rho}^{(\f{a_0}{2^M})}(r).
\end{equation}
Combining \eqref{EM-small-eps-case} and \eqref{EM-large-eps-case}, we conclude that
\begin{equation}\label{first-error-composition-estimate}
|E_{\va,\zeta}\star G_{0,-\zeta}^{\star M}(x,y)|
\leq
C_M\va L_{\va,\zeta}
\mathcal K_{2M,\rho}^{(c_M)}(|x-y|).
\end{equation}

We now derive the resolvent identity for $\pa_\zeta^M E_{\va,\zeta}$. The point is to perform the differentiation first at the operator level and then pass to distribution kernels. Let
\[
R_\va(\zeta)=(\cL_\va+\zeta I)^{-1},
\quad
R_0(\zeta)=(\cL_0+\zeta I)^{-1},
\quad
\operatorname{Re}\zeta>0.
\]
As in the proof of Lemma \ref{lem-resolvent-derivative-formula}, we have
\[
\pa_\zeta^M R_\va(\zeta)
=
(-1)^M M!R_\va(\zeta)^{M+1},
\quad
\pa_\zeta^M R_0(\zeta)
=
(-1)^M M!R_0(\zeta)^{M+1}.
\]

Next, we introduce the two first-order operators whose kernels appear in the definition of $E_{\va,\zeta}$. Let $\nabla_iR_\va(\zeta)$ denote the operator whose distribution kernel is
\[
\f{\pa}{\pa x_i}G_{\va,-\zeta}(x,y).
\]
Equivalently, for a test function $F$,
\[
\nabla_iR_\va(\zeta)F
=
\f{\pa}{\pa x_i}\left(R_\va(\zeta)F\right).
\]
Let $\mathcal P_{\va,i}R_0(\zeta)$ denote the operator whose distribution kernel is
\[
\f{\pa\Phi_{\va,j}^{\al\beta}}{\pa x_i}(x)
\f{\pa}{\pa x_j}G_{0,-\zeta}^{\beta\gamma}(x,y).
\]
In other words, if $v=R_0(\zeta)F$, then
\[
(\mathcal P_{\va,i}R_0(\zeta)F)^\al
=
\f{\pa\Phi_{\va,j}^{\al\beta}}{\pa x_i}
\f{\pa v^\beta}{\pa x_j}.
\]
Thus $E_{\va,\zeta}$ is the distribution kernel of
\[
\nabla_iR_\va(\zeta)-\mathcal P_{\va,i}R_0(\zeta).
\]

Since neither $\nabla_i$ nor $\mathcal P_{\va,i}$ depends on $\zeta$, differentiation with respect to $\zeta$ commutes with these operators in the sense of distributions. Indeed, for any $F,H\in C_0^\infty(\om;\C^m)$,
\[
\begin{aligned}
\pa_\zeta^M
\left\langle \nabla_iR_\va(\zeta)F,H\right\rangle
&=
\left\langle \nabla_i\pa_\zeta^M R_\va(\zeta)F,H\right\rangle,
\\
\pa_\zeta^M
\left\langle \mathcal P_{\va,i}R_0(\zeta)F,H\right\rangle
&=
\left\langle \mathcal P_{\va,i}\pa_\zeta^M R_0(\zeta)F,H\right\rangle.
\end{aligned}
\]
Therefore the distribution kernel $\pa_\zeta^M E_{\va,\zeta}$ is the kernel of
\[
\nabla_i\pa_\zeta^M R_\va(\zeta)
-
\mathcal P_{\va,i}\pa_\zeta^M R_0(\zeta).
\]
Using the resolvent derivative formula above, we get
\[
\pa_\zeta^M E_{\va,\zeta}
=
(-1)^M M!
\left[
\nabla_i\bigl(R_\va(\zeta)^{M+1}\bigr)
-
\mathcal P_{\va,i}\bigl(R_0(\zeta)^{M+1}\bigr)
\right]
\]
in the sense of distribution kernels.

We decompose the expression in brackets as
\[
\begin{aligned}
&\nabla_i\bigl(R_\va(\zeta)^{M+1}\bigr)
-
\mathcal P_{\va,i}\bigl(R_0(\zeta)^{M+1}\bigr)
\\
&\quad =
\left(\nabla_iR_\va(\zeta)-\mathcal P_{\va,i}R_0(\zeta)\right)R_0(\zeta)^M
+
\nabla_iR_\va(\zeta)\left(R_\va(\zeta)^M-R_0(\zeta)^M\right).
\end{aligned}
\]
Using
\[
R_\va(\zeta)^M-R_0(\zeta)^M
=
\sum_{\ell=0}^{M-1}
R_\va(\zeta)^\ell
\left(R_\va(\zeta)-R_0(\zeta)\right)
R_0(\zeta)^{M-1-\ell},
\]
we obtain the kernel identity
\[
\begin{aligned}
\pa_\zeta^M E_{\va,\zeta}
&=
(-1)^M M!
\Bigg[
E_{\va,\zeta}\star G_{0,-\zeta}^{\star M}+
\sum_{\ell=0}^{M-1}
\nabla_xG_{\va,-\zeta}
\star
G_{\va,-\zeta}^{\star\ell}
\star
D_{\va,\zeta}
\star
G_{0,-\zeta}^{\star(M-1-\ell)}
\Bigg].
\end{aligned}
\]
The first term is estimated by \eqref{first-error-composition-estimate}. For the sum, we use \eqref{first-proof-gradG-bound}, \eqref{first-proof-G-bound}, \eqref{first-proof-D-bound}, and Lemma~\ref{lem-Bessel-kernel-convolution}. For each $\ell=0,\ldots,M-1$,
\[
\begin{aligned}
&|
\nabla_xG_{\va,-\zeta}
\star
G_{\va,-\zeta}^{\star\ell}
\star
D_{\va,\zeta}
\star
G_{0,-\zeta}^{\star(M-1-\ell)}(x,y)|
\\
&\quad\leq
C_M\va
\mathcal K_{1+2\ell+1+2(M-1-\ell),\rho}^{(c_M)}(|x-y|)
\\
&\quad=
C_M\va
\mathcal K_{2M,\rho}^{(c_M)}(|x-y|).
\end{aligned}
\]
Since $L_{\va,\zeta}\geq\log2$, the sum is also bounded by
\[
C_M\va L_{\va,\zeta}
\mathcal K_{2M,\rho}^{(c_M)}(|x-y|).
\]
Therefore,
\[
|\pa_\zeta^M E_{\va,\zeta}(x,y)|
\leq
C_M\va L_{\va,\zeta}
\mathcal K_{2M,\rho}^{(c_M)}(|x-y|).
\]
Because $2M>n+2$, we have $2M>n$. Hence
\[
\mathcal K_{2M,\rho}^{(c_M)}(|x-y|)
=
\rho^{n-2M}\exp\{-c_M\rho|x-y|\}.
\]
Using $\rho\sim|\zeta|^{\f12}$ on $\C_+$ and the definition of $L_{\va,\zeta}$, we obtain \eqref{first-resolvent-derivative-exp}.

Here the notation $\nabla_i(R_\va(\zeta)^{M+1})$ means that we first take the distribution kernel of $R_\va(\zeta)^{M+1}$ and then differentiate this kernel with respect to its first spatial variable. Equivalently,
\[
\nabla_i\bigl(R_\va(\zeta)^{M+1}\bigr)
=
(\nabla_iR_\va(\zeta))R_\va(\zeta)^M.
\]
Similarly,
\[
\mathcal P_{\va,i}\bigl(R_0(\zeta)^{M+1}\bigr)
=
(\mathcal P_{\va,i}R_0(\zeta))R_0(\zeta)^M.
\]
Consequently,
\[
\begin{aligned}
&
\nabla_i\bigl(R_\va(\zeta)^{M+1}\bigr)
-
\mathcal P_{\va,i}\bigl(R_0(\zeta)^{M+1}\bigr)
\\
&\quad =
\left[
\nabla_iR_\va(\zeta)-\mathcal P_{\va,i}R_0(\zeta)
\right]R_0(\zeta)^M
+
\nabla_iR_\va(\zeta)
\left[
R_\va(\zeta)^M-R_0(\zeta)^M
\right].
\end{aligned}
\]
Finally, using the identity
\[
R_\va(\zeta)^M-R_0(\zeta)^M
=
\sum_{\ell=0}^{M-1}
R_\va(\zeta)^\ell
\left[
R_\va(\zeta)-R_0(\zeta)
\right]
R_0(\zeta)^{M-1-\ell},
\]
we obtain
\[
\begin{aligned}
\pa_\zeta^M E_{\va,\zeta}
&=
(-1)^M M!
\left[
E_{\va,\zeta}\star G_{0,-\zeta}^{\star M}
+
\sum_{\ell=0}^{M-1}
\nabla_xG_{\va,-\zeta}
\star
G_{\va,-\zeta}^{\star\ell}
\star
D_{\va,\zeta}
\star
G_{0,-\zeta}^{\star(M-1-\ell)}
\right].
\end{aligned}
\]
Here all products on the operator side have been written as spatial kernel compositions on the kernel side.

The first term is estimated by \eqref{first-error-composition-estimate}. For the remaining terms, fix $\ell\in\Z\cap[0,M-1]$. Using \eqref{first-proof-G-bound}, \eqref{first-proof-gradG-bound}, \eqref{first-proof-D-bound}, and Lemma~\ref{lem-Bessel-kernel-convolution} successively, we have
\[
\begin{aligned}
&
|\nabla_xG_{\va,-\zeta}\star
G_{\va,-\zeta}^{\star\ell}\star
D_{\va,\zeta}\star
G_{0,-\zeta}^{\star(M-1-\ell)}(x,y)|
\\
&\quad\leq
C_M\va
\mathcal K_{1+2\ell+1+2(M-1-\ell),\rho}^{(c_M)}(|x-y|)
=
C_M\va
\mathcal K_{2M,\rho}^{(c_M)}(|x-y|).
\end{aligned}
\]
Since $L_{\va,\zeta}\geq\log2$, this gives
\[
\begin{aligned}
&
|\nabla_xG_{\va,-\zeta}\star
G_{\va,-\zeta}^{\star\ell}
\star
D_{\va,\zeta}
\star
G_{0,-\zeta}^{\star(M-1-\ell)}(x,y)|
\\
&\quad\leq
C_M\va L_{\va,\zeta}
\mathcal K_{2M,\rho}^{(c_M)}(|x-y|).
\end{aligned}
\]
Combining this estimate with \eqref{first-error-composition-estimate}, and summing over $\ell$, we get
\[
|\pa_\zeta^M E_{\va,\zeta}(x,y)|
\leq
C_M\va L_{\va,\zeta}
\mathcal K_{2M,\rho}^{(c_M)}(|x-y|).
\]
Since $2M>n+2$, in particular $2M>n$, we have
\[
\mathcal K_{2M,\rho}^{(c_M)}(|x-y|)
=
\rho^{n-2M}
\exp\{-c_M\rho|x-y|\}.
\]
Finally, $\rho\sim|\zeta|^{\f12}$ and $n-2M<0$, so
\[
\rho^{n-2M}
\leq
C|\zeta|^{\f{n-2M}{2}}.
\]
This proves \eqref{first-resolvent-derivative-exp}.
\end{proof}

\begin{proof}[Proof of Theorem \ref{intro-thm-parabolic-first}]
Set
\[
\tau=t-s>0,
\quad
r=|x-y|.
\]
Using \eqref{parabolic-inversion-zeta}, we have
\[
\begin{aligned}
&
\f{\pa}{\pa x_i}G_\va^{\al\gamma}(x,t;y,s)
-
\f{\pa\Phi_{\va,j}^{\al\beta}}{\pa x_i}(x)
\f{\pa}{\pa x_j}G_0^{\beta\gamma}(x,t;y,s)=
\f{1}{2\pi\op{i}}
\int_{\mathcal C_\gamma}
\exp\{\tau\zeta\}
E_{\va,\zeta}^{\al\gamma}(x,y)\ud\zeta .
\end{aligned}
\]
Choose an integer $M\in\Z_+$ such that $2M>n+2$. Integrating by parts $M$ times along $\mathcal C_\gamma$ and using Lemma~\ref{lem-first-resolvent-derivative}, we get
\[
\begin{aligned}
&
\left|
\f{\pa}{\pa x_i}G_\va^{\al\gamma}(x,t;y,s)
-
\f{\pa\Phi_{\va,j}^{\al\beta}}{\pa x_i}(x)
\f{\pa}{\pa x_j}G_0^{\beta\gamma}(x,t;y,s)
\right|
\\
&\quad\leq
C\va\tau^{-M}\exp\{\tau\gamma\}
\int_{-\infty}^{+\infty}
|\gamma+\op{i}\eta|^{\f{n-2M}{2}}
\log\left(2+\f{1}{\va|\gamma+\op{i}\eta|^{\f12}}\right)\exp\{
-c r\operatorname{Re}\sqrt{\gamma+\op{i}\eta}
\}
\ud\eta .
\end{aligned}
\]
The corresponding contour estimate gives
\[
\begin{aligned}
&\int_{-\infty}^{+\infty}
|\gamma+\op{i}\eta|^{\f{n-2M}{2}}
\log\left(2+\f{1}{\va|\gamma+\op{i}\eta|^{\f12}}\right)
\exp\{
-c r\operatorname{Re}\sqrt{\gamma+\op{i}\eta}
\}
\ud\eta
\\
&\quad\leq
C
\gamma^{\f{n+2-2M}{2}}
\log\left(2+\f{1}{\va\sqrt\gamma}\right)
\exp\{-c r\sqrt\gamma\}.
\end{aligned}
\]
Therefore
\begin{equation}\label{pre-gaussian-first}
\begin{aligned}
&
\left|
\f{\pa}{\pa x_i}G_\va^{\al\gamma}(x,t;y,s)
-
\f{\pa\Phi_{\va,j}^{\al\beta}}{\pa x_i}(x)
\f{\pa}{\pa x_j}G_0^{\beta\gamma}(x,t;y,s)
\right|
\\
&\quad
\leq
C\va
\tau^{-M}
\gamma^{\f{n+2-2M}{2}}
\log\left(2+\f{1}{\va\sqrt\gamma}\right)
\exp\left\{
\tau\gamma-c r\sqrt\gamma
\right\}.
\end{aligned}
\end{equation}

If $r^2\leq\tau$, take $\gamma=\tau^{-1}$. Then
\[
\tau^{-M}\gamma^{\f{n+2-2M}{2}}
=
\tau^{-\f{n+2}{2}},
\quad
\log\left(2+\f{1}{\va\sqrt\gamma}\right)
=
\log\left(2+\f{\sqrt\tau}{\va}\right).
\]
The exponential term is bounded above by a constant, and the Gaussian factor is bounded below by a positive constant. This gives the desired estimate in the case $r^2\leq\tau$.

If $r^2>\tau$, take
\[
\gamma=\f{ar^2}{\tau^2},
\]
where $a>0$ is sufficiently small. Then
\[
\exp\{\tau\gamma-c r\sqrt\gamma\}
\leq
\exp\left\{-\f{cr^2}{\tau}\right\}.
\]
Since $2M>n+2$,
\[
\tau^{-M}\gamma^{\f{n+2-2M}{2}}
=
\tau^{-\f{n+2}{2}}
\left(\f{r^2}{\tau}\right)^{\f{n+2-2M}{2}}
\leq
\tau^{-\f{n+2}{2}}.
\]
Moreover, since $r^2>\tau$,
\[
\f{1}{\sqrt\gamma}
=
\f{\tau}{\sqrt a\, r}
\leq
C\sqrt\tau,
\]
and hence
\[
\log\left(2+\f{1}{\va\sqrt\gamma}\right)
\leq
C\log\left(2+\f{\sqrt\tau}{\va}\right).
\]
Combining these estimates with \eqref{pre-gaussian-first} proves \eqref{intro-parabolic-first}.
\end{proof}

\appendix

\section{Proof of Lemma \ref{lem-Bessel-kernel-convolution}}\label{ProofofTechnicalLemma}

It is enough to prove the estimate with $\om$ replaced by $\R^n$, since the integrand is nonnegative and $\om\subset\R^n$. We first remove the parameter $\rho$ by scaling. Set
\[
X=\rho x,
\quad
Y=\rho y,
\quad
Z=\rho z.
\]
Then
\[
\mathcal K_{\sigma,\rho}^{(a)}(|x-z|)
=
\rho^{n-\sigma}k_\sigma^{(a)}(|X-Z|),
\]
where
\[
k_\sigma^{(a)}(r)
=
\begin{cases}
r^{\sigma-n}\exp\{-ar\},
&\sg\in(0,n),\\[4pt]
\left(1+|\log r|\right)\exp\{-ar\},
&\sigma=n,\\[4pt]
\exp\{-ar\},
&\sg\in(n,+\ift).
\end{cases}
\]
Since $\ud z=\rho^{-n}\ud Z$, we have
\[
\begin{aligned}
&\int_{\R^n}
\mathcal K_{\sigma,\rho}^{(a)}(|x-z|)
\mathcal K_{\sigma',\rho}^{(a)}(|z-y|)
\ud z
\\
&\quad
=
\rho^{n-\sigma-\sigma'}
\int_{\R^n}
k_\sigma^{(a)}(|X-Z|)
k_{\sigma'}^{(a)}(|Z-Y|)
\ud Z.
\end{aligned}
\]
Thus, it suffices to prove that, for $R=|X-Y|$,
\begin{equation}\label{dimensionless-Bessel-convolution}
\int_{\R^n}
k_\sigma^{(a)}(|X-Z|)
k_{\sigma'}^{(a)}(|Z-Y|)
\ud Z
\leq
C k_{\sigma+\sigma'}^{(\f{a}{2})}(R).
\end{equation}
If $R=0$, then the left-hand side is finite when $\sigma+\sigma'>n$, by local integrability near the common singularity and exponential decay at infinity, while $k_{\sigma+\sigma'}^{(\f{a}{2})}(0)=1$ in this case. When $\sigma+\sigma'\leq n$, the right-hand side is infinite in the limiting sense. Hence, it is enough to consider $R>0$. By translation and rotation, we may assume that $X=0$ and $Y=Re_1$, where $e_1=(1,0,\ldots,0)\in\R^n$.

We prove \eqref{dimensionless-Bessel-convolution} by considering separately the cases $R\in(0,2]$ and $R>2$.

\subsection{The case $R\in(0,2]$}

Since exponential factors are bounded above by $1$, the contribution of $|Z|\leq4$ is controlled by
\[
\int_{B_4}
\psi_\sigma(|Z|)
\psi_{\sigma'}(|Z-Re_1|)
\ud Z,
\]
where
\[
\psi_\theta(r)
=
\begin{cases}
r^{\theta-n},&\theta\in(0,n),\\[4pt]
1+|\log r|,&\theta=n,\\[4pt]
1,&\theta\in(n,+\ift).
\end{cases}
\]
We claim that
\begin{equation}\label{local-Riesz-convolution}
\int_{B_4}
\psi_\sigma(|Z|)
\psi_{\sigma'}(|Z-Re_1|)
\ud Z
\leq
C\psi_{\sigma+\sigma'}(R).
\end{equation}

Split $B_4$ into
\[
E_1=\left\{|Z|\leq\f R2\right\},
\quad
E_2=\left\{|Z-Re_1|\leq\f R2\right\},
\quad
E_3=B_4\backslash(E_1\cup E_2).
\]
On $E_1$ we have
\[
\f R2\leq |Z-Re_1|\leq \f{3R}{2}.
\]
Hence
\[
\psi_{\sigma'}(|Z-Re_1|)
\leq
C\psi_{\sigma'}(R).
\]
It follows that
\be
\int_{E_1}\psi_\sigma(|Z|)\psi_{\sigma'}(|Z-Re_1|)\ud Z
\leq
C\psi_{\sigma'}(R)
\int_0^{\f{R}{2}}\psi_\sigma(r)r^{n-1}\ud r.
\label{checkproduct}
\ee
Moreover,
\[
\int_0^{\f{R}{2}}\psi_\sigma(r)r^{n-1}\ud r
\leq
C
\begin{cases}
R^\sigma,&\sg\in(0,n),\\[3pt]
R^n(1+|\log R|),&\sigma=n,\\[3pt]
R^n,&\sg\in(n,+\ift).
\end{cases}
\]
We now check that the right-hand side of \eqref{checkproduct} is bounded by $C\psi_{\sigma+\sigma'}(R)$.

\begin{itemize}
\item If $\sigma<n$ and $\sigma'<n$, this is the usual power-counting estimate:
\[
\psi_{\sigma'}(R)R^\sigma
=
R^{\sigma+\sigma'-n},
\]
with the borderline case $\sigma+\sigma'=n$ controlled by
\[
1\leq C(1+|\log R|)=C\psi_{\sigma+\sigma'}(R).
\]

\item If $\sigma=n$ and $0<\sigma'<n$, then
\[
\psi_{\sigma'}(R)R^n(1+|\log R|)
=
R^{\sigma'}(1+|\log R|)
\leq C
=
C\psi_{\sigma+\sigma'}(R),
\]
since $\sigma+\sigma'>n$. The case where $\sg\in(0,n)$ and $\sigma'=n$ is similar.

\item If $\sg\in(n,+\ift)$ and $0<\sigma'<n$, then
\[
\psi_{\sigma'}(R)R^n
=
R^{\sigma'}
\leq C
=
C\psi_{\sigma+\sigma'}(R).
\]

\item The remaining cases, in which $\sigma'\geq n$ or both parameters are at least $n$, are easier, since $\psi_{\sigma+\sigma'}(R)=1$ and the product is uniformly bounded for $R\in(0,2]$.
\end{itemize}

Therefore, the contribution of $E_1$ is bounded by $C\psi_{\sigma+\sigma'}(R)$. The estimate on $E_2$ is identical.

We now estimate the contribution of $E_3$. On $E_3$, both $|Z|$ and $|Z-Re_1|$ are bounded below by $\f{R}{2}$. Let $N\in\Z_+$ be chosen so that
\[
2^N R\leq4<2^{N+1}R.
\]
We split
\[
E_3=E_{3,0}\cup\left(\bigcup_{j=1}^N E_{3,j}\right),
\]
where
\[
E_{3,0}=E_3\cap\left\{\f R2\leq |Z|\leq2R\right\},
\]
and, for $j\in\Z\cap[1,N]$,
\[
E_{3,j}=E_3\cap\left\{2^jR<|Z|\leq2^{j+1}R\right\}.
\]

On $E_{3,0}$, we have $|Z|\sim R$ and $|Z-Re_1|\sim R$. Indeed, $|Z-Re_1|\geq \f{R}{2}$ by the definition of $E_3$, while
\[
|Z-Re_1|\leq |Z|+R\leq3R.
\]
Therefore
\begin{equation}\label{E30estimate}
\int_{E_{3,0}}
\psi_\sigma(|Z|)
\psi_{\sigma'}(|Z-Re_1|)
\ud Z
\leq
C R^n\psi_\sigma(R)\psi_{\sigma'}(R).
\end{equation}
Moreover,
\be
R^n\psi_\sigma(R)\psi_{\sigma'}(R)
\leq
C\psi_{\sigma+\sigma'}(R),
\quad R\in(0,2].
\label{E30estimate1inter}
\ee
This, together with \eqref{E30estimate}, implies
\begin{equation}\label{E30estimate1}
\int_{E_{3,0}}
\psi_\sigma(|Z|)
\psi_{\sigma'}(|Z-Re_1|)
\ud Z
\leq
C\psi_{\sigma+\sigma'}(R),
\quad R\in(0,2].
\end{equation}

The estimate \eqref{E30estimate1inter} follows from the following direct calculations.
\begin{itemize}
\item If $\sigma<n$ and $\sigma'<n$, then
\[
R^n\psi_\sigma(R)\psi_{\sigma'}(R)
=
R^{\sigma+\sigma'-n},
\]
and the conclusion follows directly, with the case $\sigma+\sigma'=n$ controlled by
\[
1\leq C(1+|\log R|).
\]

\item If one of $\sigma,\sigma'$ equals $n$, the possible logarithmic factor is multiplied by a positive power of $R$ unless both equal $n$; in all such cases $\sigma+\sigma'>n$, and the expression is uniformly bounded for $R\in(0,2]$.

\item If one of $\sigma,\sigma'$ is larger than $n$, the estimate is immediate from $\psi_{\sigma+\sigma'}(R)=1$.
\end{itemize}

For $j\in\Z\cap[1,N]$ and $Z\in E_{3,j}$, we have $|Z|\sim2^jR$. Moreover, since $j\geq1$,
\[
|Z-Re_1|\geq |Z|-R\geq \f12 |Z|,
\quad
|Z-Re_1|\leq |Z|+R\leq2|Z|.
\]
Thus $|Z-Re_1|\sim2^jR$, and hence
\be
\int_{E_{3,j}}
\psi_\sigma(|Z|)
\psi_{\sigma'}(|Z-Re_1|)
\ud Z
\leq
C(2^jR)^n
\psi_\sigma(2^jR)
\psi_{\sigma'}(2^jR).
\label{summationuse}
\ee

We now sum this estimate in $j$. First suppose that $\sigma<n$ and $\sigma'<n$. Then
\[
(2^jR)^n
\psi_\sigma(2^jR)
\psi_{\sigma'}(2^jR)
=
(2^jR)^{\sigma+\sigma'-n}.
\]

\begin{itemize}
\item If $\sigma+\sigma'<n$, then
\[
\sum_{j=1}^N
(2^jR)^{\sigma+\sigma'-n}
\leq
C R^{\sigma+\sigma'-n}
=
C\psi_{\sigma+\sigma'}(R).
\]

\item If $\sigma+\sigma'=n$, then each summand is bounded by a constant. Since the choice of $N$ gives $N\leq C(1+|\log R|)$,
\[
\sum_{j=1}^N 1
\leq
C(1+|\log R|)
=
C\psi_{\sigma+\sigma'}(R).
\]

\item If $\sigma+\sigma'>n$, then $\sigma+\sigma'-n>0$, and
\[
\sum_{j=1}^N
(2^jR)^{\sigma+\sigma'-n}
\leq
C
=
C\psi_{\sigma+\sigma'}(R).
\]
\end{itemize}

It remains to consider the case where at least one of $\sigma,\sigma'$ is greater than or equal to $n$. Then $\sigma+\sigma'>n$ and $\psi_{\sigma+\sigma'}(R)=1$. Set
\[
r_j=2^jR.
\]
Since $r_j\leq2^NR\leq4$, it suffices to prove a uniform bound for
\[
\sum_{j=1}^N
r_j^n\psi_\sigma(r_j)\psi_{\sigma'}(r_j).
\]
We use the elementary estimate
\begin{equation}\label{dyadic-elementary-estimate}
\sum_{j=1}^N r_j^b(1+|\log r_j|)^q\leq C,
\quad b>0,\quad q\in\{0,1,2\},
\end{equation}
where $C>0$ depends only on $b$ and $q$. To prove \eqref{dyadic-elementary-estimate}, consider first the indices for which $r_j\leq1$. If $j_0$ is the largest integer such that $r_{j_0}\leq1$, then $r_{j_0}>\f{1}{2}$ and, for $j\leq j_0$,
\[
r_j=2^{j-j_0}r_{j_0}\leq2^{j-j_0},
\quad
1+|\log r_j|\leq C(1+j_0-j).
\]
Hence
\[
\sum_{j\leq j_0}r_j^b(1+|\log r_j|)^q
\leq
C\sum_{\ell=0}^{\infty}2^{-b\ell}(1+\ell)^q
\leq C.
\]
For the indices with $r_j>1$, one has $1<r_j\leq4$, and the number of such dyadic indices is bounded by an absolute constant. This proves \eqref{dyadic-elementary-estimate}.

We apply \eqref{dyadic-elementary-estimate} to the remaining cases in the summation of \eqref{summationuse}.
\begin{itemize}
\item If $\sg\in(n,+\ift)$ and $0<\sigma'<n$, then
\[
r_j^n\psi_\sigma(r_j)\psi_{\sigma'}(r_j)
\leq
C r_j^{\sigma'},
\]
and the sum is bounded. The case $\sg\in(0,n)$ and $\sigma'>n$ is identical.

\item If $\sigma=n$ and $0<\sigma'<n$, then
\[
r_j^n\psi_\sigma(r_j)\psi_{\sigma'}(r_j)
\leq
C r_j^{\sigma'}(1+|\log r_j|),
\]
and the sum is bounded. The case $\sg\in(0,n)$ and $\sigma'=n$ is identical.

\item If $\sigma=n$ and $\sigma'=n$, then
\[
r_j^n\psi_\sigma(r_j)\psi_{\sigma'}(r_j)
\leq
C r_j^n(1+|\log r_j|)^2,
\]
and the sum is bounded. If $\sg\in(n,+\ift)$ and $\sigma'=n$, then
\[
r_j^n\psi_\sigma(r_j)\psi_{\sigma'}(r_j)
\leq
C r_j^n(1+|\log r_j|),
\]
and the sum is bounded. The case $\sigma=n$ and $\sigma'>n$ is identical.

\item If $\sg\in(n,+\ift)$ and $\sigma'>n$, then
\[
r_j^n\psi_\sigma(r_j)\psi_{\sigma'}(r_j)
\leq
C r_j^n,
\]
and the sum is bounded.
\end{itemize}

Therefore,
\[
\sum_{j=1}^N
\int_{E_{3,j}}
\psi_\sigma(|Z|)
\psi_{\sigma'}(|Z-Re_1|)
\ud Z
\leq C
\leq
C\psi_{\sigma+\sigma'}(R),
\]
where we have used $\psi_{\sg+\sg'}(R)=1$. Combining this with \eqref{E30estimate1} gives
\[
\int_{E_3}
\psi_\sigma(|Z|)
\psi_{\sigma'}(|Z-Re_1|)
\ud Z
\leq
C\psi_{\sigma+\sigma'}(R).
\]
Together with the estimates on $E_1$ and $E_2$, this proves \eqref{local-Riesz-convolution}.

For $R\in(0,2]$, it remains to estimate the contribution of $|Z|>4$. Since $R\leq2$, for $|Z|>4$ we have
\[
|Z-Re_1|\geq |Z|-R\geq\f12 |Z|.
\]
Thus the exponential decay gives
\[
\int_{|Z|>4}
k_\sigma^{(a)}(|Z|)
k_{\sigma'}^{(a)}(|Z-Re_1|)
\ud Z
\leq C.
\]
On the other hand, for $R\in(0,2]$ one has $\psi_{\sigma+\sigma'}(R)\geq c>0$. Indeed, this is clear if $\sigma+\sigma'\geq n$; if $\sigma+\sigma'<n$, then
\[
\psi_{\sigma+\sigma'}(R)=R^{\sigma+\sigma'-n}\geq 2^{\sigma+\sigma'-n}.
\]
Since $\exp\{-\f{aR}{2}\}\sim1$ for $R\in(0,2]$, we conclude that
\be
\int_{\R^n}
k_\sigma^{(a)}(|Z|)
k_{\sigma'}^{(a)}(|Z-Re_1|)
\ud Z
\leq
Ck_{\sigma+\sigma'}^{(\f a2)}(R),
\quad R\in(0,2].
\label{CKsgsgprimR02}
\ee

\subsection{The case $R>2$}

In this case we use the triangle inequality to separate the exponential decay from the singular factors. Since
\[
|Z|+|Z-Re_1|\geq R,
\]
we have
\[
\exp\{-a|Z|\}\exp\{-a|Z-Re_1|\}
\leq
\exp\left\{-\f{aR}{2}\right\}
\exp\left\{-\f a2\left(|Z|+|Z-Re_1|\right)\right\}.
\]
Therefore,
\[
\begin{aligned}
&\int_{\R^n}
k_\sigma^{(a)}(|Z|)
k_{\sigma'}^{(a)}(|Z-Re_1|)
\ud Z
\\
&\quad
\leq
\exp\left\{-\f{aR}{2}\right\}
\int_{\R^n}
\psi_\sigma(|Z|)
\psi_{\sigma'}(|Z-Re_1|)
\exp\left\{-\f a2\left(|Z|+|Z-Re_1|\right)\right\}
\ud Z.
\end{aligned}
\]
We claim that
\begin{equation}\label{large-R-psi-estimate}
\int_{\R^n}
\psi_\sigma(|Z|)
\psi_{\sigma'}(|Z-Re_1|)
\exp\left\{-\f a2\left(|Z|+|Z-Re_1|\right)\right\}
\ud Z
\leq
C\psi_{\sigma+\sigma'}(R).
\end{equation}

First, suppose that $\sigma+\sigma'\geq n$. Since $\psi_{\sigma+\sigma'}(R)\geq1$ for $R>2$, it is enough to prove a uniform bound. Split
\[
\R^n=A_1\cup A_2\cup A_3,
\]
where
\[
A_1=\{Z:|Z|\leq1\},
\quad
A_2=\{Z:|Z-Re_1|\leq1\},
\quad
A_3=\R^n\backslash(A_1\cup A_2).
\]
On $A_1$, since $R>2$, we have
\[
|Z-Re_1|\geq R-1\geq\f R2.
\]
Thus
\[
\psi_{\sigma'}(|Z-Re_1|)
\exp\left\{-\f a2|Z-Re_1|\right\}
\leq C,
\]
where $C$ is independent of $R$. Therefore,
\[
\int_{A_1}
\psi_\sigma(|Z|)
\psi_{\sigma'}(|Z-Re_1|)
\exp\left\{-\f a2\left(|Z|+|Z-Re_1|\right)\right\}
\ud Z
\leq
C\int_{|Z|\leq1}\psi_\sigma(|Z|)\ud Z
\leq C.
\]
The estimate on $A_2$ is identical, with $\sigma$ and $\sigma'$ interchanged.

On $A_3$, both $|Z|$ and $|Z-Re_1|$ are at least $1$. For any $\theta>0$,
\[
\psi_\theta(r)\exp\left\{-\f{ar}{4}\right\}
\leq C_\theta,
\quad r\geq1.
\]
Consequently, on $A_3$,
\[
\begin{aligned}
&
\psi_\sigma(|Z|)
\psi_{\sigma'}(|Z-Re_1|)
\exp\left\{-\f a2\left(|Z|+|Z-Re_1|\right)\right\}
\\
&\quad
\leq
C
\exp\left\{-\f a4|Z|\right\}
\exp\left\{-\f a4|Z-Re_1|\right\}.
\end{aligned}
\]
Hence
\[
\begin{aligned}
&\int_{A_3}
\psi_\sigma(|Z|)
\psi_{\sigma'}(|Z-Re_1|)
\exp\left\{-\f a2\left(|Z|+|Z-Re_1|\right)\right\}
\ud Z
\\
&\quad
\leq
C
\int_{\R^n}
\exp\left\{-\f a4|Z|\right\}
\exp\left\{-\f a4|Z-Re_1|\right\}
\ud Z.
\end{aligned}
\]
The last integral is uniformly bounded in $R$. Indeed, if $h(Z)=\exp\{-\f{a}{4}|Z|\}$, then the last integral equals $(h*h)(Re_1)$, and
\[
\|h*h\|_{L^\infty(\R^n)}
\leq
\|h\|_{L^1(\R^n)}\|h\|_{L^\infty(\R^n)}
\leq C.
\]
This proves \eqref{large-R-psi-estimate} when $\sigma+\sigma'\geq n$.

It remains to consider $\sigma+\sigma'<n$. Then necessarily $\sigma<n$ and $\sigma'<n$. Dropping the exponential factor, we obtain
\[
\begin{aligned}
&\int_{\R^n}
\psi_\sigma(|Z|)
\psi_{\sigma'}(|Z-Re_1|)
\exp\left\{-\f a2\left(|Z|+|Z-Re_1|\right)\right\}
\ud Z
\\
&\quad
\leq
\int_{\R^n}
|Z|^{\sigma-n}
|Z-Re_1|^{\sigma'-n}
\ud Z.
\end{aligned}
\]
By the scaling $Z=RW$, the last integral equals
\[
R^{\sigma+\sigma'-n}
\int_{\R^n}
|W|^{\sigma-n}
|W-e_1|^{\sigma'-n}
\ud W.
\]
The integral in $W$ is finite: near $0$ this follows from $\sigma>0$, near $e_1$ from $\sigma'>0$, and at infinity from $\sigma+\sigma'<n$. Hence
\[
\int_{\R^n}
|Z|^{\sigma-n}
|Z-Re_1|^{\sigma'-n}
\ud Z
\leq
C R^{\sigma+\sigma'-n}
=
C\psi_{\sigma+\sigma'}(R).
\]
This also proves \eqref{large-R-psi-estimate} when $\sigma+\sigma'<n$.

Combining \eqref{large-R-psi-estimate} with the exponential factor gives
\[
\int_{\R^n}
k_\sigma^{(a)}(|Z|)
k_{\sigma'}^{(a)}(|Z-Re_1|)
\ud Z
\leq
C\psi_{\sigma+\sigma'}(R)\exp\left\{-\f{aR}{2}\right\}.
\]
Since
\[
k_{\sigma+\sigma'}^{(\f a2)}(R)
=
\psi_{\sigma+\sigma'}(R)\exp\left\{-\f{aR}{2}\right\},
\]
the preceding estimate is exactly \eqref{dimensionless-Bessel-convolution} for $R>2$. Together with the case $R\in(0,2]$ in \eqref{CKsgsgprimR02}, this proves \eqref{dimensionless-Bessel-convolution}. Rescaling back gives \eqref{Bessel-kernel-convolution}.

\section*{Acknowledgment}

The author acknowledges the use of ChatGPT, developed by OpenAI, for assistance with some preliminary calculations and for checking the consistency of intermediate estimates. All AI-assisted calculations and suggestions were independently verified by the author, who takes full responsibility for the mathematical content of the article.

\bibliographystyle{plain}

\begin{thebibliography}{10}

\bibitem{AS16}
S.~N. Armstrong and Z.~Shen.
\newblock Lipschitz estimates in almost-periodic homogenization.
\newblock {\em Commun. Pure Appl. Math.}, 69(10):1882--1923, 2016.

\bibitem{AL87}
M.~Avellaneda and F.~Lin.
\newblock Compactness methods in the theory of homogenization.
\newblock {\em Commun. Pure Appl. Math.}, 40(6):803--847, 1987.

\bibitem{BLP78}
A.~Bensoussan, J.-L. Lions, and G.~Papanicolaou.
\newblock {\em Asymptotic analysis for periodic structures}, volume~5 of {\em
  Stud. Math. Appl.}
\newblock Elsevier, Amsterdam, 1978.

\bibitem{BS04}
M.~Sh. Birman and T.~A. Suslina.
\newblock Second order periodic differential operators. {Threshold} properties
  and homogenization.
\newblock {\em St. Petersbg. Math. J.}, 15(5):639--714, 2004.

\bibitem{BS06}
M.~Sh. Birman and T.~A. Suslina.
\newblock Homogenization with corrector term for periodic elliptic differential
  operators.
\newblock {\em St. Petersbg. Math. J.}, 17(6):897--973, 2006.

\bibitem{CDK08}
S.~Cho, H.~Dong, and S.~Kim.
\newblock On the {Green}'s matrices of strongly parabolic systems of second
  order.
\newblock {\em Indiana Univ. Math. J.}, 57(4):1633--1677, 2008.

\bibitem{CDK12}
S.~Cho, H.~Dong, and S.~Kim.
\newblock Global estimates for {Green}'s matrix of second order parabolic
  systems with application to elliptic systems in two dimensional domains.
\newblock {\em Potential Anal.}, 36(2):339--372, 2012.

\bibitem{DK14}
H.~Dong and S.~Kim.
\newblock Green's functions for parabolic systems of second order in
  time-varying domains.
\newblock {\em Commun. Pure Appl. Anal.}, 13(4):1407--1433, 2014.

\bibitem{Gen23}
J.~Geng.
\newblock Gaussian bounds and asymptotic expansions of {Green} function in
  parabolic homogenization.
\newblock {\em Calc. Var. Partial Differ. Equ.}, 62(6):40, 2023.
\newblock Id/No 170.

\bibitem{GengNiu25}
J.~Geng and W.~Niu.
\newblock Homogenization of locally periodic parabolic operators with
  non-self-similar scales.
\newblock {\em Calc. Var. Partial Differ. Equ.}, 64(4):34, 2025.
\newblock Id/No 115.

\bibitem{GS15}
J.~Geng and Z.~Shen.
\newblock Uniform regularity estimates in parabolic homogenization.
\newblock {\em Indiana Univ. Math. J.}, 64(3):697--733, 2015.

\bibitem{GS17}
J.~Geng and Z.~Shen.
\newblock Convergence rates in parabolic homogenization with time-dependent
  periodic coefficients.
\newblock {\em J. Funct. Anal.}, 272(5):2092--2113, 2017.

\bibitem{GS20b}
J.~Geng and Z.~Shen.
\newblock Asymptotic expansions of fundamental solutions in parabolic
  homogenization.
\newblock {\em Anal. PDE}, 13(1):147--170, 2020.

\bibitem{GengShen20NS}
J.~Geng and Z.~Shen.
\newblock Homogenization of parabolic equations with non-self-similar scales.
\newblock {\em Arch. Ration. Mech. Anal.}, 236(1):145--188, 2020.

\bibitem{GS20}
J.~Geng and B.~Shi.
\newblock Green's matrices and boundary estimates in parabolic homogenization.
\newblock {\em J. Differ. Equations}, 269(4):3031--3066, 2020.

\bibitem{Gri04}
G.~Griso.
\newblock Error estimate and unfolding for periodic homogenization.
\newblock {\em Asymptotic Anal.}, 40(3-4):269--286, 2004.

\bibitem{Gri06}
G.~Griso.
\newblock Interior error estimate for periodic homogenization.
\newblock {\em Anal. Appl., Singap.}, 4(1):61--79, 2006.

\bibitem{HK04}
S.~Hofmann and S.~Kim.
\newblock Gaussian estimates for fundamental solutions to certain parabolic
  systems.
\newblock {\em Publ. Mat., Barc.}, 48(2):481--496, 2004.

\bibitem{KLS14}
C.~Kenig, F.~Lin, and Z.~Shen.
\newblock Periodic homogenization of {Green} and {Neumann} functions.
\newblock {\em Commun. Pure Appl. Math.}, 67(8):1219--1262, 2014.

\bibitem{KLS12}
C.~E. Kenig, F.~Lin, and Z.~Shen.
\newblock Convergence rates in {{$ L^{2} $}} for elliptic homogenization
  problems.
\newblock {\em Arch. Ration. Mech. Anal.}, 203(3):1009--1036, 2012.

\bibitem{KLS13}
C.~E. Kenig, F.~Lin, and Z.~Shen.
\newblock Estimates of eigenvalues and eigenfunctions in periodic
  homogenization.
\newblock {\em J. Eur. Math. Soc. (JEMS)}, 15(5):1901--1925, 2013.

\bibitem{kLS13b}
C.~E. Kenig, F.~Lin, and Z.~Shen.
\newblock Homogenization of elliptic systems with {Neumann} boundary
  conditions.
\newblock {\em J. Am. Math. Soc.}, 26(4):901--937, 2013.

\bibitem{MengNiu24}
Q.~Meng and W.~Niu.
\newblock Homogenization of fundamental solutions for parabolic operators
  involving non-self-similar scales.
\newblock {\em Ann. Mat. Pura Appl. (4)}, 203(5):2357--2382, 2024.

\bibitem{Niu24Para}
W.~Niu.
\newblock Reiterated homogenization of parabolic systems with several spatial
  and temporal scales.
\newblock {\em J. Funct. Anal.}, 286(9):61, 2024.
\newblock Id/No 110365.

\bibitem{NSX18}
W.~Niu, Z.~Shen, and Y.~Xu.
\newblock Convergence rates and interior estimates in homogenization of higher
  order elliptic systems.
\newblock {\em J. Funct. Anal.}, 274(8):2356--2398, 2018.

\bibitem{NSX20}
W.~Niu, Z.~Shen, and Y.~Xu.
\newblock Quantitative estimates in reiterated homogenization.
\newblock {\em J. Funct. Anal.}, 279(11):39, 2020.
\newblock Id/No 108759.

\bibitem{NX19}
W.~Niu and Y.~Xu.
\newblock Uniform boundary estimates in homogenization of higher-order elliptic
  systems.
\newblock {\em Ann. Mat. Pura Appl. (4)}, 198(1):97--128, 2019.

\bibitem{NXZ26}
W.~Niu, Y.~Xu, and J.~Zhuge.
\newblock Optimal convergence rates in multiscale elliptic homogenization.
\newblock {\em Comm. Pure Appl. Math.}, page e70047, 2026.

\bibitem{OSY92}
O.~A. Ole{\u{\i}}nik, A.~S. Shamaev, and G.~A. Yosifian.
\newblock {\em Mathematical problems in elasticity and homogenization},
  volume~26 of {\em Stud. Math. Appl.}
\newblock Amsterdam etc.: North-Holland, 1992.

\bibitem{OV07}
D.~Onofrei and B.~Vernescu.
\newblock Error estimates for periodic homogenization with non-smooth
  coefficients.
\newblock {\em Asymptotic Anal.}, 54(1-2):103--123, 2007.

\bibitem{She08}
Z.~Shen.
\newblock {{$W^{1,p}$}} estimates for elliptic homogenization problems in
  nonsmooth domains.
\newblock {\em Indiana Univ. Math. J.}, 57(5):2283--2298, 2008.

\bibitem{She18}
Z.~Shen.
\newblock {\em Periodic homogenization of elliptic systems}, volume 269 of {\em
  Oper. Theory: Adv. Appl.}
\newblock Cham: Birkh{\"a}user, 2018.

\bibitem{Shi25Neu}
B.~Shi.
\newblock Uniform boundary estimates for {Neumann} problems in parabolic
  homogenization.
\newblock {\em Anal. Appl., Singap.}, 23(1):31--63, 2025.

\bibitem{Sus13a}
T.~A. Suslina.
\newblock Homogenization of the {Dirichlet} problem for elliptic systems:
  {{$L_2$}}-operator error estimates.
\newblock {\em Mathematika}, 59(2):463--476, 2013.

\bibitem{Sus13b}
T.~A. Suslina.
\newblock Homogenization of the {Neumann} problem for elliptic systems with
  periodic coefficients.
\newblock {\em SIAM J. Math. Anal.}, 45(6):3453--3493, 2013.

\bibitem{Wan23}
W.~Wang.
\newblock Uniform estimates of resolvents in homogenization theory of elliptic
  systems.
\newblock {\em J. Differ. Equations}, 370:1--65, 2023.

\bibitem{WZ23}
W.~Wang and T.~Zhang.
\newblock Homogenization theory of elliptic system with lower order terms for
  dimension two.
\newblock {\em Commun. Pure Appl. Anal.}, 22(3):787--824, 2023.

\bibitem{Xu16Lower}
Q.~Xu.
\newblock Uniform regularity estimates in homogenization theory of elliptic
  system with lower order terms.
\newblock {\em J. Math. Anal. Appl.}, 438(2):1066--1107, 2016.

\bibitem{Xu16LowerNeu}
Q.~Xu.
\newblock Uniform regularity estimates in homogenization theory of elliptic
  systems with lower order terms on the {Neumann} boundary problem.
\newblock {\em J. Differ. Equations}, 261(8):4368--4423, 2016.

\bibitem{Xu23Para}
Y.~Xu.
\newblock Convergence rates in homogenization of parabolic systems with locally
  periodic coefficients.
\newblock {\em J. Differ. Equations}, 367:1--39, 2023.

\bibitem{ZKO94}
V.~V. Zhikov, S.~M. Kozlov, and O.~A. Olejnik.
\newblock {\em Homogenization of differential operators and integral
  functionals. {Transl}. from the {Russian} by {G}. {A}. {Yosifian}}.
\newblock Berlin: Springer-Verlag, 1994.

\end{thebibliography}

\end{document}